\documentclass{amsart}
\usepackage{enumerate}
\usepackage[latin1]{inputenc}
\usepackage{amsmath}
\usepackage[utf8]{luainputenc}
\usepackage{graphicx}
\usepackage{bbm}
\usepackage{bm}
\usepackage{listings}
\usepackage{amsthm}
\usepackage{amssymb}
\usepackage{hyperref}
\usepackage{cancel}
\usepackage[hang,sf,small]{caption}
 
\usepackage{fullpage}
\usepackage{color}
\usepackage{float}
\usepackage[table, svgnames, dvipsnames]{xcolor}
\usepackage{makecell, cellspace, caption}
\newcommand{\vertiii}[1]{{\left\vert\kern-0.25ex\left\vert\kern-0.25ex\left\vert #1 
    \right\vert\kern-0.25ex\right\vert\kern-0.25ex\right\vert}}%
\theoremstyle{definition}

\newtheorem{lemma}{Lemma}

\numberwithin{equation}{section}

\newcommand{\DG}{\mathrm{DG}}
\newcommand{\mesh}{\mathcal{E}}
\newcommand{\uDG}{u_h^{\mathrm{DG}}}
\newcommand{\uCG}{u_h^{\mathrm{CG}}}

\newcommand{\Bk}{\color{black}}

\newtheorem{theorem}{Theorem} 
\newtheorem{remark}{Remark}
\begin{document}
%%-----------------------------
%%      the top matter
%%-----------------------------
\title{Discontinuous Galerkin Approximations to Elliptic and Parabolic Problems with a Dirac Line Source} %\thanks{...}\thanks{...}% At most 5 thanks
%'
%\runningtitle{DG for elliptic and parabolic problems with line sources}
\author{Rami Masri}  \email{rami.masri@rice.edu} 
\author{Boqian Shen}\email{bs58@rice.edu} %{Department of Computational and Applied Mathematics, Rice University}
\author{Beatrice Riviere}  \email{riviere@rice.edu} %{Department of Computational and Applied Mathematics, Rice University}
%\
\address{Department of Computational and Applied Mathematics, Rice University}
\date{\today}
\begin{abstract} 
The analyses of interior penalty discontinuous Galerkin methods of any order $k$ for solving elliptic and parabolic problems with Dirac line sources are presented. For the steady state case, we prove convergence of the method by deriving a priori error estimates in the $L^2$ norm and in weighted energy norms. In addition, we prove almost optimal local error estimates in the energy norm for any approximation order. Further, almost optimal local error estimates in the $L^2$ norm are obtained for the case of piecewise linear approximations whereas suboptimal error bounds in the $L^2$ norm are shown for any polynomial degree. %Almost optimal local error estimates in energy norms are proved for any approximation order. 
For the time-dependent case, convergence of semi-discrete and of backward Euler fully discrete scheme is established by proving error estimates in $L^2$ in time and in space. Numerical results for the elliptic problem are added to support the theoretical results.
\end{abstract}
%
%\begin{resume} 
%... \end{resume}
%
\subjclass{65M60, 65N30, 35J75}
\keywords{Interior penalty dG, convergence, local $L^2$ estimates, local energy estimates, singular solutions}
\maketitle
%%-----------------------------
%%      your text
%%----------------------------- 

\section{Introduction}
In this paper, we analyze interior penalty discontinuous Galerkin (dG) approximations to elliptic and parabolic problems with a Dirac measure concentrated on a line. Consider a convex domain $\Omega \subset \mathbbm{R}^3$ containing a one-dimensional curve $\Lambda \subset\mathbbm{R}$ which is strictly  included in $\Omega$. The elliptic model problem reads 
\begin{alignat}{2}
-\Delta u  & = f \delta_{\Lambda},&& \quad \text{in } \Omega, \label{eq:model_pb} \\ 
u & = 0,&& \quad \text{on } \partial \Omega. \label{eq:model_pb_2}
\end{alignat}
where $f \in L^2(\Lambda)$ and $f\delta_{\Lambda}$ is a Dirac measure concentrated on $\Lambda$ defined as follows.  
\begin{equation}\label{eq:def_dirac_measure}
    \langle f\delta_{\Lambda}, v \rangle = \int_{\Lambda} fvds, \quad \forall v \in L^\infty(\Omega). 
\end{equation}
For the parabolic problem, let $T$ be the final time, let $u^0$ be in $L^2(\Omega)$ and assume that $f$ belongs to $L^2(0,T; L^2(\Lambda))$.
 We  consider the following problem. 
\begin{alignat}{2}
\partial_t u - \Delta u &= f \delta_{\Lambda}, && \quad  \mathrm{in } \,\,  \Omega \times (0,T], \label{eq:parabolic_1}\\   
u  & = 0, && \quad \mathrm{on}\,\,  \partial \Omega \times (0,T], \label{eq:parabolic_2} \\ 
u  & = u^0,  && \quad \mathrm{in} \,\,  \{0\} \times \Omega.  \label{eq:parabolic_3}
\end{alignat}
The main contributions of this work are as follows. For the elliptic problem, we show global convergence in the $L^2$ norm and in weighted energy norms. Further, in regions excluding the line $\Lambda$,   we derive 
almost optimal $L^2$ error estimates for linear polynomials and suboptimal error bounds of order almost $k$ for dG approximations of degree $k\geq 2$. In addition, almost optimal error rates are established in local energy norms for approximations of  any polynomial degree. For the parabolic problem, we show global convergence in the $L^2(0,T;L^2(\Omega))$ norm for both the semi-discrete approximation and for the backward Euler fully discrete scheme.   

Partial differential equations with Dirac right-hand sides  can model organ perfusion where blood vessels are considered as one dimensional fractures embedded in the tissue \cite{d2008coupling}. In this case, $f$ can be a function of the blood pressure in the vessel leading to a coupled 1D-3D problem for the pressures in the tissue and in the vessels \cite{d2012finite,d2008coupling}. Medical applications of such formulations include modeling drug delivery to tissues with the help of implantable devices \cite{d2007multiscale} and drug delivery to tumors where different treatment options are compared \cite{cattaneo2014computational}. In addition, Dirac measures concentrated on lines arise in optimal control problems \cite{gong2014approximations}.  Thanks to favorable properties of dG methods, including local mass conservation and adaptability to complex domains \cite{riviere2008discontinuous}, these methods are well suited to model physical phenomena such as organ perfusion. In this paper we study  dG methods applied to  \eqref{eq:model_pb}-\eqref{eq:model_pb_2} and to \eqref{eq:parabolic_1}-\eqref{eq:parabolic_3}.

The analysis of finite element approximations to model problems \eqref{eq:model_pb}-\eqref{eq:model_pb_2} and \eqref{eq:parabolic_1}-\eqref{eq:parabolic_3} is non--standard since the true solution is not smooth enough in space, namely it does not belong to $H^1(\Omega)$ and it exhibits a logarithmic singularity near the line $\Lambda$ \cite{d2012finite,koppl2016local,ariche2017regularity}. Nevertheless, continuous Galerkin (cG) approximations have been extensively studied; we refer to the work by Scott \cite{scott1973finite} and  Casas \cite{casas19852} where global error bounds are established. More recently and in the context of optimal control problems, Gong et al. derived improved global $L^2$ error bounds \cite{gong2014approximations}. Such bounds are polluted by the singularity of the true solution where the rate of convergence in the $L^2$ norm for any polynomial degree is at most $\mathcal{O}(h)$ where $h$ is the mesh-size. For continuous Galerkin approximations to \eqref{eq:parabolic_1}-\eqref{eq:parabolic_3}, global error estimates for semi-discrete and fully-discrete formulations are derived in \cite{gong2016finite,gong2013error}. 

In addition, convergence of the cG approximations to the elliptic model problem \eqref{eq:model_pb}-\eqref{eq:model_pb_2} has been investigated in different non--classical norms. For example, local $L^2$ optimal error estimates (up to a log factor for linear polynomials) are derived by K\"{o}ppl et al. \cite{koppl2014optimal,koppl2016local}, and  local energy error estimates are obtained by Bertoluzza et al. \cite{bertoluzza2018local}. Such improved estimates are possible since the solution is smooth in regions excluding the line $\Lambda$ \cite{ariche2017regularity}. In addition, D'Angelo obtained error estimates in weighted norms and showed that with graded meshes the finite element solution converges optimally in these norms \cite{d2012finite}. We also mention the recent splitting technique to numerically approximate the model problem \eqref{eq:model_pb}-\eqref{eq:model_pb_2} introduced by Gjerde et al. where the solution is split into an explicit singular part and an implicit smooth part \cite{gjerde2020singularity}. A finite element discretization is then formulated for the smooth part and optimal error rates are recovered \cite{gjerde2020singularity}.

To the best of our knowledge, discontinuous Galerkin approximations to \eqref{eq:model_pb}-\eqref{eq:model_pb_2} and to \eqref{eq:parabolic_1}-\eqref{eq:parabolic_3} are missing from the literature. However, there are papers which formulate and study dG methods for elliptic problems with Dirac sources concentrated at a point. To this end, we mention the work by Houston and Wihler where global a priori and a posteriori error bounds are derived \cite{houston2012discontinuous}. Recently, Choi and Lee derived local $L^2$ error estimates \cite{choi2020optimal}. 
The analysis of dG methods for elliptic problems  is particularly challenging since consistency of the numerical method cannot be assumed since the traces of the solution and its gradient are not well defined.   

The rest of this paper is organized as follows. Weak formulations in usual and in weighted Sobolev spaces are presented and shown to be equivalent in Section~\ref{sec:weak_formulations}. Then, Section~\ref{sec:numerical_approximation} defines the cG and dG discrete solutions to model problem \eqref{eq:model_pb}-\eqref{eq:model_pb_2}. We show global convergence in the $L^2$ norm  in Section~\ref{sec:global_conv} and in weighted dG norms in Section~\ref{sec:weighted_global_conv}. The local convergence of the solution is analyzed in Section~\ref{sec:local_estimates}. We devote Section~\ref{sec:parabolic_analysis} to the analysis of dG formulations for \eqref{eq:parabolic_1}-\eqref{eq:parabolic_3}. Numerical results for the elliptic problem are presented in Section~\ref{sec:numerical_results}. 

\section{Weak formulation} \label{sec:weak_formulations}
%Let $W^{1,s}(\Omega)$ for $s \in \mathbbm{R}$ denote the space of all real valued functions which belong to $L^s(\Omega)$ and have first order partial derivatives in $L^s(\Omega)$. 
Fix  $p_0 \in [1,3/2)$ and $q_0$ be such that $1/q_0+1/p_0 = 1$.   Let $W^{1,p_0}(\Omega)$ denote the usual Sobolev space and 
recall that  \[
W_0^{1,p_0}(\Omega) = \{ v\in W^{1,p_0}(\Omega), \quad v=0 \quad \mbox{on}\quad \partial\Omega\}.
\]
The weak formulation for problem \eqref{eq:model_pb}-\eqref{eq:model_pb_2} is \cite{casas19852}: Find $u \in W_0^{1,p_0}(\Omega)$ such that: 
\begin{equation}
    \int_{\Omega} \nabla u \cdot \nabla v  = \int_{\Lambda} fv, \quad \forall v \in W_0^{1,q_0}(\Omega). \label{eq:weak_form_u}
\end{equation}
This weak formulation is well posed and a unique solution $u \in W_0^{1,p_0}(\Omega)$ for $p_0 \in [1,3/2)$ exists \cite{casas19852}. 
%The integrals in the above formulation are well defined. To see that, we apply H
%\"{o}lder's inequality and obtain
%\[  \int_{\Omega} \nabla u \cdot \nabla v  \leq \Vert \nabla u \Vert_{L^p(\Omega)} \Vert \nabla v \Vert_{L^q(\Omega)} < \infty. \]
%Further, note that for $ p \in [1,3/2)$, we have $q \in (3,\infty]$. By Sobolev's inequality (Theorem 1.4.6  \cite{brenner2007mathematical} ), we have that for $v \in W^{1,q}(\Omega)$,
%\[ \Vert v \Vert_{L^{\infty}(\Omega)} \leq C \Vert v \Vert_{W^{1,q}(\Omega)}. \]
%This implies that the right hand side is bounded as follows. 
%\[  \int_{\Lambda} fv  \leq \Vert v \Vert_{L^{\infty}(\Omega)} \Vert f \Vert_{L^1(\Lambda)} < \infty.   \]
Next, in a similar way to \cite{d2012finite}, we present another weak formulation of problem~\eqref{eq:model_pb}-\eqref{eq:model_pb_2} in weighted Sobolev spaces. 
%We introduce weighted $L^2$ spaces.
Define the distance function to $\Lambda$:
\begin{equation}\label{eq:distdef}
 d(\bm{x},\Lambda) = \text{dist}(\bm{x}, \Lambda) = \min_{\bm{y}\in \Lambda} \Vert \bm{x} -\bm{y} \Vert, \quad \forall \bm{x} \in \Omega. 
\end{equation}
We first remark that $d^{\alpha}$ is an $A_2$ weight for $|\alpha| < 2$ (see Lemma 3.3 in \cite{duran2010solutions}) where $A_2$ is the Muckenhoupt class of weights satisfying:
 \[A_2 = \Big\{  w \in L^1_{\mathrm{loc}}(\mathbbm{R}^3), \,\, \sup_{B(\bm{x},r)}\left(\frac{1}{|B(\bm{x},r)|}\int_{B(\bm{x},r)} w \right) \left(\frac{1}{|B(\bm{x},r)|} \int_{B(\bm{x},r)} w^{-1}\right) < \infty \Big\},  \]
 where the supremum is taken over all balls  $B(\bm{x},r)$ centered at $\bm{x}$ and of radius $r$.
This implies that $d^\alpha$ belongs to $L^2(\Omega)$ if $\vert \alpha\vert < 1$. 
We assume that the distance function satisfies the following bounds %, which hold in the case of $\Lambda$ being
%a straight line 
(see Theorem 3.4 in \cite{delfour1994shape}).
\begin{alignat}{2}
\|\nabla d\|_{L^{\infty}(\Omega)}   \leq 1, \quad
 \|\nabla^2 d^2 \|_{L^{\infty}(\Omega)}\leq C.
\label{eq:boundond}
\end{alignat}
 Using the fact the $\nabla d^\alpha = \alpha d^{\alpha-1} \nabla d$, we then have that $d^\alpha \in H^1(\Omega)$
if $0<\alpha<1$.
For $\alpha \in (-1,1)$, define the weighted $L^2$ norm as follows.
\begin{equation}
\|u\|_{L^2_{\alpha}(\Omega)} =  \left( \int_{\Omega} |u|^2 d^{2\alpha} \right)^{\frac{1}{2}}. 
\end{equation}
The  $L^2_{\alpha}(\Omega)$ space and the weighted inner product are defined as:
\[L^2_{\alpha}(\Omega) = \{ v: \, \|v\|_{L^2_{\alpha}(\Omega)} < \infty \}, \quad (u,v)_\alpha = \int_{\Omega} uvd^{2\alpha}, \quad \forall u,v \in L^2_{\alpha}(\Omega). \]
Similarly, we introduce the weighted Sobolev spaces as:
\[ H^{m}_{\alpha} (\Omega) = \{u : D^{\bm{\beta} }u \in L^2_{\alpha}(\Omega), |\bm{\beta}| \leq m  \}, \,\,\, \mathring{H}^{m}_{\alpha}(\Omega) = \{ u \in H^{m}_{\alpha}(\Omega), u \vert_{\partial \Omega} = 0\}. \]
where $\bm{\beta}$ is a multi-index and $D^{\bm{\beta}}$ is the corresponding weak derivative.
The weighted Sobolev semi-norms and norms are denoted by:
\begin{align*}
    |u|_{H^{m}_{\alpha}(\Omega)}^2 =  \sum_{|\bm{\beta}| = m} \|D^{\bm{\beta}}u\|_{L^2_{\alpha}(\Omega)}^2, \quad  \| u\|_{H^{m}_{\alpha}(\Omega)}^2 = \sum_{k=0}^{m}  |u|_{H^m_{\alpha}(\Omega)}^2.
\end{align*}
\begin{lemma}\label{lemma:equiv_weak_forms}
%    Let $p \in [1,3/2)$ and $q$ be such that $1/q+1/p = 1$. 
Let $\alpha$ be such that $-2/p_0 + 1 < \alpha < 2/p_0 -1$.
Then, the  weak formulation \eqref{eq:weak_form_u} is equivalent to the following weak formulation:  find $u_\alpha \in \mathring{H}^1_{\alpha}(\Omega)$ such that
\begin{equation}
\int_{\Omega} \nabla u_\alpha \cdot \nabla v = \int_{\Lambda} fv, \quad \forall v \in \mathring{H}^{1}_{-\alpha}(\Omega).\label{eq:weak_form_weighted}  
\end{equation}
\end{lemma}
\begin{proof}
Let $u_\alpha$ be a solution of \eqref{eq:weak_form_weighted}. The existence and uniqueness of $u_\alpha$ is established in \cite{d2012finite}, see also \cite{drelichman2020weighted}. Observe that the condition on $\alpha$ implies that $(\alpha p_0 )/(2-p_0) = (\alpha q_0)/(q_0-2) \in (-1,1)$. Since $d^{\gamma} \in L^1_{\mathrm{loc}}(\mathbbm{R}^3)$ for $|\gamma| \leq 2$, we use H\"{o}lder's inequality and obtain
\begin{align}
\int_{\Omega} d^{-2\alpha} v^2 \leq \left( \int_{\Omega} d^{-2\alpha \frac{q_0}{q_0-2}} \right)^{(q_0-2)/q_0} \|v\|^{2/q_0}_{L^{q_0}(\Omega)} < \infty, \quad \forall v \in L^{q_0}(\Omega).
\end{align}
This implies that $W_0^{1,q_0}(\Omega) \subset \mathring{H}^1_{-\alpha}(\Omega)$. Hence $u_\alpha$ satisfies \eqref{eq:weak_form_u} for all $v \in W_0^{1,q_0}(\Omega)$.
Similarly, for $v \in L^{2}_{\alpha}(\Omega)$, we have
\begin{equation*}
    \int_{\Omega} v^{p_0} = \int_\Omega v^{p_0} d^{p_0\alpha} d^{-p_0\alpha} \leq \left(\int_\Omega v^2 d^{2\alpha}\right)^{p_0/2} \left(\int_\Omega d^{-2\alpha \frac{p_0}{2-p_0} }\right)^{(2-p_0)/2} < \infty, \quad \forall v \in L^2_{\alpha}(\Omega).
\end{equation*}
This implies that $\mathring{H}^1_{\alpha}(\Omega) \subset W_0^{1,p_0}(\Omega)$. Thus, $u_\alpha$ solves \eqref{eq:weak_form_u}. Since the solution to \eqref{eq:weak_form_u} is unique (see Theorem 2.1 case (ii) in \cite{gong2014approximations}), we conclude that $u_\alpha = u$.
\end{proof}
\section{Numerical approximations}\label{sec:numerical_approximation}
Let $\mathcal{E}_h$ denote a partition of $\Omega$, made of simplices:
\begin{equation}
 \bigcup_{E \in \mesh_h} \bar{E} = \bar{\Omega}. 
\end{equation}
The diameter of a given element $E$ is denoted by $h_E$ and the mesh size is denoted by $h = \max_{E \in \mesh_h} h_E$. We assume that $\mesh_h$ is regular in the sense that there exists a constant $\rho > 0$ such that 
\begin{equation}
 \frac{h_E}{\rho_E} \leq \rho, \quad \forall E\in \mesh_h,  \label{eq:mesh_regular}
\end{equation}
where $\rho_E$ is the maximum diameter of a ball inscribed in $E$. In addition, we assume that $\mesh_h$ is quasi-uniform: there is a constant $\gamma >0$ independent of $h$ such that
\begin{equation}
    h \leq \gamma h_E, \quad \forall E \in \mesh_h. \label{eq:quasi_uniform}
\end{equation} The broken Sobolev space is denoted by $H^m(\mesh_h)$ for $m\geq 1$, and the broken gradient is denoted
by $\nabla_h$.  In the remaining of the paper, $k\geq 1$ is a fixed positive integer 
and $C$ is a generic constant independent of $h$.
\subsection{Finite element approximation}
Let $W^k_h(\mesh_h)$ be the finite element space defined as follows.  
\begin{equation} W_h^k(\mesh_h)= \{ w_h \in H^1_0(\Omega):  w_h \vert_E \in \mathbbm{P}^k(E), \,\, \forall E \in \mesh_h\}.
\end{equation} 
Here, $\mathbbm{P}^k(E)$ denotes the space of polynomials of degree at most $k$. Let $\uCG \in W^k_h(\mesh_h)$ be the finite element
approximation to $u$ satisfying 
\begin{equation}
\int_{\Omega} \nabla \uCG \cdot \nabla v_h  = \int_{\Lambda} f v_h, \quad \forall v_h \in W^k_h(\mesh_h). \label{eq:discrete_F.E}
\end{equation} 
%In this paper, we will also use the notation $u_h=u_h^{\mathrm{CG}}$ (resp. $u_h=u_h^2$) for the continuous piecewise linear (resp. quadratic) finite element solution. 
\subsection{Discontinuous Galerkin approximation}
We now introduce the interior penalty discontinuous Galerkin discrete solution \cite{riviere2008discontinuous}. We define the broken polynomial space as follows.
\begin{align} 
V_h^{k}(\mesh_h) = \{ v_h \in L^2(\Omega): v_h \vert_{E} \in \mathbbm{P}^k(E), \forall E \in \mathcal{E}_h \}. 
\end{align} 
We also denote by $\Gamma_h$ the set of all interior faces in $\mesh_h$.  For each interior face $e$, we associate a unit normal
vector $\bm{n}_e$ and we denote by $E_e^1$ and $E_e^2$ the two elements that share $e$ such that the vector $\bm{n}_e$ points from 
$E_e^1$ to $E_e^2$. We denote the average and the jump of a function $v_h \in V_h^k(\mesh_h)$ by $\{v_h\}$
and $[v_h]$ respectively.  
\begin{align}
\{v_h\} = \frac{1}{2} \left( v_h\vert_{E_e^1} + v_h \vert_{E_e^2}  \right), \quad [v_h] = v_h \vert_{E_e^1} - v_h\vert_{E_e^2},
\quad \forall e\in\Gamma_h.  
\end{align} 
If $e$ belongs to the boundary of the domain, $e = \partial \Omega \cap \partial E_e^1$, then we define the average and the jump  as follows. 
\begin{equation}
    [v] = \{ v\} = v\vert_{E_e^1}. 
\end{equation}
Let $u_h^\DG \in V_h^{k}(\mesh_h)$ be the discontinuous Galerkin solution satisfying:
\begin{align} a_{\epsilon}(u_h^{\DG},v_h) = \int_{\Lambda} fv_h,\quad \forall  v_h \in V_h^{k}(\mesh_h), \label{eq:DG_weak_form}\end{align}
where $a_{\epsilon}(\cdot, \cdot): V_h^{k}(\mesh_h) \times V_h^{k}(\mesh_h) \rightarrow \mathbbm{R}$ is given by: 
\begin{eqnarray} 
a_{\epsilon}(u,v) &=& \sum_{E \in \mesh_h} \int_E \nabla u \cdot  \nabla v   - \sum_{e \in \Gamma_h \cup \partial \Omega} \int_e  \{\nabla u\}\cdot \bm{n}_e [v] \nonumber  \\ 
&&+ \epsilon \sum_{e \in \Gamma_h \cup \partial \Omega} \int_e  \{\nabla v\}\cdot \bm{n}_e [u]  + \sum_{e \in \Gamma_h \cup \partial \Omega} \int_e \frac{\sigma}{h^{\beta}} [u][v].  \label{eq:dg_bilinear_form} %\\
%&&-\sum_{e\in\partial \Omega} \int_e \nabla u \cdot \bfn v
%+\epsilon \sum_{e\in\partial \Omega} \int_e \nabla v \cdot \bfn u 
%+\sum_{e\in\partial\Omega} \frac{\sigma}{h}\int_e u v.  \label{eq:dg_bilinear_form}
\end{eqnarray}
In the above, $\epsilon \in \{-1,0,1\}$, $\sigma$ is a user specified parameter and $\beta\geq 1$ is a parameter to be specified in the subsequent sections.  
We define the following energy semi-norm. For $B \subseteq \Omega$ or $B = \overline{\Omega}$ and $ v_h \in V_h^k(\mesh_h)$, 
\begin{equation}
\|v_h \|^2_{\DG(B)} = \sum_{E \in \mesh_h} \| \nabla v_h \|^2_{L^2(E \cap B)} + \sum_{e\in \Gamma_h \cup \partial \Omega} \sigma h^{-1} \|[v_h]\|^2_{L^2(e \cap  B)}. \label{eq:dg_norm}
\end{equation}
For simplicity, we write $\|\cdot \|^2_{\DG} = \|\cdot\|^2_{\DG(\overline{\Omega})}$. We also note that $\|\cdot \|_{\DG}$ defines a norm and the following Poincare inequality holds \cite{di2010discrete}.
\begin{equation}
\|v_h\|_{L^p(\Omega)} \leq C\|v_h\|_{\DG}, \quad \forall 1\leq p\leq 6, \,\,\, \forall v_h \in V_h^k(\mesh_h). \label{eq:3d_1d_Poincare_inequality}
\end{equation}
In the analysis, we will also use the following semi-norm. For $v \in H^2(\mesh_h)$ and $B \subseteq \Omega$ or $B = \overline{\Omega}$, 
\begin{equation}
\vertiii{v}^2_{\DG(B)}  = \|v\|^2_{\DG(B)} + \sum_{e\in \Gamma_h \cup \partial \Omega} h \|\{\nabla v\}\|^2_{L^2(e\cap B  )}. 
\end{equation}
Similarly, denote $\vertiii{\cdot}_{\DG}^2 = \vertiii{\cdot}_{\DG(\overline{\Omega})}^2$. 
We then have the following continuity properties of the form $a_\epsilon$ \cite{chen2004pointwise,riviere2008discontinuous}. %For $v,w \in H^2(\mesh_h)$ and for $v_h,w_h \in V_h^k(\mesh_h)$, 
\begin{equation}
 a_\epsilon  (v,w) \leq C \vertiii{v}_{\DG} \vertiii{w}_{\DG} , \quad a_\epsilon  (v_h,w_h) \leq C \|v_h\|_{\DG} \|w_h\|_{\DG},
\quad \forall v,w\in H^2(\mesh_h), \, \forall v_h,w_h \in V_h^k(\mesh_h). \label{eq:continuity_prop}
\end{equation} 
In addition, the following coercivity property \begin{equation}
    a_\epsilon(w_h, w_h) \geq \frac12 \|w_h\|^2_{\DG}, \quad \forall w_h \in V_h^k(\mesh_h), \label{eq:coercivity_property}
    \end{equation}  
 is valid for any value $\sigma\geq 1$ if $\epsilon = +1$ and for $\sigma$ large enough if $\epsilon = -1, 0$. We recall the following important inverse inequalities, see Section 4.5 in \cite{brenner2007mathematical}. 
 \begin{equation}
    \|v_h\|_{L^q(\Omega)} \leq Ch^{\frac{3}{q} - \frac{3}{p}} \|v_h\|_{L^p(\Omega)}, \quad \forall \, 1\leq p \leq q \leq \infty, \,\,\, \forall \,  v_h \in V_h^{k} (\mesh_h). \label{eq:3d_1d_inverse_estimate}
  \end{equation}
For the trace estimates, we will make use of the following. %For $v \in H^1(\mesh_h)$, 
\begin{align}
\|v\|_{L^2(e)} \leq C h^{-1/2}(\|v\|_{L^2(E)} + h \|\nabla v\|_{L^2(E)}), \quad \forall e \subset\partial E, \quad \forall E \in \mesh_h, \quad \forall v \in H^1(\mesh_h).\label{eq:3d_1d_trace_estimate_continuous} 
\end{align}
For discrete functions,  the above estimate reads 
\begin{align}
    \|v_h\|_{L^2(e)} \leq C h^{-1/2}\|v_h\|_{L^2(E)}, \quad \forall e \subset\partial E, \quad  \forall E \in \mesh_h, \quad \forall v_h \in V_h^k(\mesh_h).\label{eq:3d_1d_trace_estimate_discrete}
    \end{align}
Further, we recall that for any $p\in[1,\infty]$, 
\begin{align}
\|\nabla_h v_h\|_{L^p(\Omega)} \leq C h^{-1} \|v_h\|_{L^p(\Omega)}, \quad \forall v_h \in V_h^k(\mesh_h). \label{eq:3d_1d_inverse_esimate_nabla}
\end{align}
\section{Global error estimate in the $L^2$ norm} \label{sec:global_conv}
The goal of this section is to show a global $L^2$ estimate for the error $u$ - $u^{\DG}_h$.
We first recall  important global $L^2$ estimates  
 for the finite element discretization \eqref{eq:discrete_F.E}. %Here and in what follows, $C$ is a generic constant independent of $h$. 
For $k=1$. Casas obtained the following estimate \cite{casas19852}, 
\begin{equation} \Vert u - \uCG \Vert_{L^2(\Omega)} \leq C h^{1/2} \Vert f \Vert_{L^2(\Lambda)}. \label{eq:L2estimate_F.E} \end{equation}
If the line $\Lambda$ is a $\mathcal{C}^2$ curve that does not intersect the boundary $\partial\Omega$, the improved estimate 
\begin{equation} 
\Vert u - \uCG \Vert_{L^2(\Omega)} \leq C(\theta) h^{1- \theta} \Vert f \Vert_{L^2(\Lambda)},   \quad 0 < \theta < \frac12, \label{eq:improved_L2estimate_F.E_h} \end{equation}
was proved by Gong et al. for $k=1$ in \cite{gong2014approximations}. 
Similar arguments yield the same error bounds for $k\geq 2$.
 The parameter $\theta$ arises from the fact that $u\in W_0^{1,\frac{6}{2\theta+3}}(\Omega)$ when $0<\theta<1/2$. 
We follow the ideas of Scott \cite{scott1973finite} and Houston and Wihler \cite{houston2012discontinuous} presented for a problem with a Dirac source concentrated at a point, and we construct an intermediate problem with an $L^2$ source term. Let $\mathcal{T}_{\Lambda} \subset \mesh_h$ be the set of elements that intersect the line $\Lambda$, 
\[ \mathcal{T}_{\Lambda} = \{ E \in \mesh_h, \,\,\,  E \cap \Lambda \neq \emptyset \}.\]
Define $f_h \in V_h^k(\mesh_h)$ as  
\begin{equation}
    \forall E \in\mathcal{E}_h, \quad f_h|_E = \begin{cases}
        f_{h,E}, & \mbox{if } \, E \in \mathcal{T}_{\Lambda}, \\ 
        0,   & \text{otherwise}, 
    \end{cases}
\end{equation} 
where $f_{h,E} \in \mathbbm{P}^k(E)$ is defined as follows. For $E \in \mathcal{T}_\Lambda$, 
\begin{equation}
\label{eq:def_fh}
    \int_E f_{h,E} v_h  = \int_{E\cap \Lambda} f v_h, \quad \forall v_h  \in \mathbbm{P}^k(E).  
\end{equation}
 Clearly, the function $f_{h,E}$ is well defined. 
Further, consider the following intermediate problem:  find $U\in H_0^1(\Omega)$ such that
\begin{alignat}{2} 
- \Delta U   & = f_h,  && \quad \text{in}\,\, \Omega, \label{eq:intermediate_pb} \\ 
U & = 0, && \quad \text{on} \,\, \partial \Omega. 
\end{alignat} 
 Since $f_h$ belongs to $L^2(\Omega)$, Lax-Milgram's theorem yields existence and uniqueness of $U$.  In addition, since
$\Omega$ is convex, the function $U$ belongs to $H^2(\Omega)$.  
We proceed by obtaining a bound on $f_h$ in the following lemma. 
\begin{lemma}\label{lemma:bound_on_fh} 
The following estimate holds
\begin{equation}
    \Vert f_{h} \Vert_{L^2(\Omega)} \leq Ch^{-3/2} \Vert f \Vert_{L^2(\Lambda)}. \label{eq:bd_fh}
\end{equation}
In addition, if $\Lambda$ is a $\mathcal{C}^2$ curve  and the mesh satisfies $|\Lambda \cap E | \leq C h$ for all $E \in \mesh_h$, we have 
\begin{equation}
    \Vert f_{h} \Vert_{L^2(\Omega)} \leq Ch^{-1} \Vert f \Vert_{L^2(\Lambda)}. \label{eq:improved_bd_fh}
\end{equation}
\end{lemma}
\begin{proof}
With the definition of $f_h$ given in \eqref{eq:def_fh}, we have
\begin{align*} 
\Vert f_h \Vert^2_{L^2(\Omega)} = \int_{\Omega} f_h^2 &= \sum_{E \in \mesh_h}  \int_{E}  (f_h|_E)^2    =\sum_{E \in \mathcal{T}_\Lambda}  \int_{E \cap \Lambda} f_{h,E} f. 
\end{align*}
Using H\"{o}lder's inequality, we obtain  
\[  \int_{E \cap \Lambda} f_{h,E} f  \leq \Vert f_{h,E} \Vert_{L^{\infty}(E)} \Vert f \Vert_{L^1(E \cap \Lambda)}.\]
% Let $\hat{E}$ be a reference element of diameter $d = 1$ and $\hat{f}_{h,E} (\hat{x}) = f_{h,E}(h \hat{x}) \in \mathbbm{P}^k(\hat{E})$. We have by a Sobolev embedding that: 
% \[ \Vert f_{h,E} \Vert_{L^{\infty}(E)}  = \Vert \hat{f}_{h,E} \Vert_{L^{\infty}(\hat{E})} \leq C \Vert \hat{f}_{h,E} \Vert_{H^2(\hat{E})}. \]
% By a change of variables, we obtain: 
% \begin{align*} 
% \Vert \hat{f}_{h,E} \Vert_{H^2(\hat{E})}^2  & = \sum_{ 0\leq s\leq 2} |  \hat{f}_{h,E} |^{2}_{H^s(\hat{E})}  \\          
% & = \sum_{0\leq s\leq 2} h_E^{2s - 3} \vert f_{h,E} \vert^{2}_{H^s(E)}
% \end{align*}
% By the inverse inequality for polynomials, we have: 
% % \[ \vert f_{h,E} \vert^{2}_{H^s(E)} \leq C h_E^{-2s} \Vert f_{h,E} \Vert^2_{L^2(E)} \]
% % Thus, we obtain: 
% \begin{equation} 
%     \Vert f_{h,E} \Vert_{L^{\infty}(E)} \leq C h_E^{-3/2} \Vert f_{h,E} \Vert_{L^2(E)}. \label{eq:inverse_fhE} 
% \end{equation}    
Hence, with \eqref{eq:3d_1d_inverse_estimate} ($q= \infty, p =2$), and  \eqref{eq:quasi_uniform}, we obtain  
\begin{align*} 
\Vert f_h \Vert^2_{L^2(\Omega)} \leq \sum_{E \in \mathcal{T}_{\Lambda}} \Vert f_{h,E} \Vert_{L^{\infty}(E)} \Vert f \Vert_{L^1(E \cap \Lambda)}  & \leq Ch^{-3/2}  \sum_{E \in \mathcal{T}_{\Lambda}}  \Vert f_{h,E} \Vert_{L^2(E)}   \Vert f \Vert_{L^1(E \cap \Lambda)} \\ &  \leq   Ch^{-3/2}  \sum_{E \in \mathcal{T}_{\Lambda}}  \Vert f_{h,E} \Vert_{L^2(E)} |\Lambda \cap E |^{1/2} \Vert f \Vert_{L^2(E \cap \Lambda)}. 
\end{align*}
If $|\Lambda \cap E| \leq Ch$, we apply H\"{o}lder's inequality for sums and obtain \eqref{eq:improved_bd_fh}. Otherwise, we have \eqref{eq:bd_fh}.
\end{proof}
The following a priori error bounds hold.
\begin{lemma}\label{lem:Uresults}
There exists a constant $C$ independent of $h$ such that
\begin{align}
\Vert U- \uCG \Vert_{L^2(\Omega)} + h \Vert \nabla (U-\uCG) \Vert_{L^2(\Omega)} \leq Ch^2 \Vert U \Vert_{H^2(\Omega)},  \label{eq:global_bd_interm_f.e}\\
\vertiii{U - u_h^{\mathrm{DG}}}_{\DG} \leq C h \Vert U\Vert_{H^2(\Omega)}.\label{eq:xiDGfull}
\end{align}
If in addition, $\beta = 1$ and $\sigma$ is large enough if $\epsilon = -1$
or $\beta > 3/2$ and $\sigma$ is large enough for $\epsilon = 0$ or $\epsilon = 1$,  
there exists a constant $C$ independent of $h$ such that
\begin{equation}
\Vert  U - u_h^{\mathrm{DG}} \Vert_{L^2(\Omega)} \leq C h^2 \Vert U\Vert_{H^2(\Omega)}.
\label{eq:global_bd_dg}
\end{equation}
\end{lemma}
\begin{proof}
We have for any $v_h \in  V_h^k(\mesh_h)$,  
\[
\int_\Omega  f_h v_h  = \sum_{E \in \mesh_h} \int_E  f_h|_E \, v_h    
= \sum_{E \in \mathcal{T}_{\Lambda}} \int_{E \cap \Lambda} f v_h = \int_{\Lambda} fv_h. 
\]
Thus, since $W_h^k(\mathcal{E}_h)$ is a subset of $V_h^k(\mathcal{E}_h$), 
the discrete functions $\uCG$ and $u^{\DG}_h$ can be viewed as finite element and discontinuous Galerkin approximations to the intermediate problem \eqref{eq:intermediate_pb}. Since $f_h \in L^2(\Omega)$, standard approximation and error bounds hold. In particular, 
\eqref{eq:global_bd_interm_f.e} and \eqref{eq:xiDGfull} hold. 
%we have: 
%\begin{equation}
%\Vert \uCG -U \Vert_{H^1(\Omega)} \leq Ch \Vert U \Vert_{H^2(\Omega)}. 
%\end{equation}
%and
%\begin{equation}
%\Vert \uCG -U \Vert_{L^2(\Omega)} \leq Ch^2 \Vert U \Vert_{H^2(\Omega)}. \label{eq:global_bd_interm_f.e}    
%\end{equation}
%Similarly, we have
%\begin{equation}
%\Vert \uCG - U\Vert_{DG} \leq C h \Vert U \Vert_{H^2(\Omega)}.
%\end{equation}
%A similar result hold for the symmetric  discontinuous Galerkin  approximation where $\epsilon = -1$ and $\sigma$ large enough. 
%\begin{equation}
    %\Vert U - u^{\DG}_h \Vert_{L^2(\Omega)} \leq Ch^2 \Vert U \Vert_{H^2(\Omega)}.     \label{eq:global_bd_dg}
    %\end{equation}
%If $\beta > 3/2$, the above result also holds for the non-symmetric ($\epsilon = 1)$ and for the incomplete  ($\epsilon = 0$)  cases if $\sigma$ is large enough. 
For a proof of \eqref{eq:global_bd_dg}, we refer to  Theorem 2.13 in \cite{riviere2008discontinuous}.  
\end{proof}
We are now ready to present and prove the main result of this section. 
\begin{theorem}
    \label{theorem:global_estimate}
Assume the penalty parameter $\sigma$ is chosen so that \eqref{eq:coercivity_property} holds. 
In addition, if $\epsilon = \{0,1\}$, select $\beta > 3/2$ and if $\epsilon = -1$, choose $\beta = 1$. 
%If $\epsilon = \{-1, 0\} $, we assume $\sigma$ to be large enough \Rd and choose $\beta =1$?? \Bk. 
%If $\epsilon = \{0,1\}$, we assume super-penalization: $\beta > 3/2$. For $\epsilon = -1$, we set $\beta = 1$. 
Then,  there exists a constant $C$ independent of $h$  such that 
\begin{equation}
    \Vert u - u^{\DG}_h \Vert_{L^2(\Omega)} \leq Ch^{1/2}  \Vert f \Vert_{L^2(\Lambda)}.   \label{eq:global_dg_bd} 
\end{equation}
In addition, if $\Lambda$ is a $\mathcal{C}^2$ curve and $| \Lambda \cap \overline E| \leq Ch $ for all $E \in \mesh_h$, we have the following improved estimate.
\begin{equation}
    \Vert u - u^{\DG}_h \Vert_{L^2(\Omega)} \leq C(\theta) h^{1-\theta}  \Vert f \Vert_{L^2(\Lambda)}, \quad  0 < \theta < 1/2.   \label{eq:improved_global_dg_bd}
\end{equation}
\end{theorem}
\begin{proof}
We use triangle inequality to obtain: 
\begin{equation}
\label{eq:triangle411}
\Vert u - u^{\DG}_h \Vert_{L^2(\Omega)} \leq \Vert u - \uCG \Vert_{L^2(\Omega)} + \Vert \uCG - U\Vert_{L^2(\Omega)} + \Vert U - u^{\DG}_h \Vert_{L^2(\Omega)} .    
\end{equation}
 %Note that $f_h$ and $f$ are the same linear functionals on $W_h^k(\mesh_h)$ and $V_h^k(\mesh_h)$ since 
We have for any $v_h \in  V_h^k(\mesh_h)$,  
\[
\int_\Omega  f_h v_h  = \sum_{E \in \mesh_h} \int_E  f_h|_E \, v_h    
= \sum_{E \in \mathcal{T}_{\Lambda}} \int_{E \cap \Lambda} f v_h = \int_{\Lambda} fv_h. 
\]
%%Since $W_h^k(\mathcal{E}_h)$ is a subset of $V_h^k(\mathcal{E}_h$), 
%%he discrete functions $u_h$ and $u^{\DG}_h$ can be viewed as finite element and discontinuous Galerkin approximations to the intermediate problem \eqref{eq:intermediate_pb}. Since $f_h \in L^2(\Omega)$, standard approximation results hold. In particular, we have: 
%%\begin{equation}
%%\Vert \uCG -U \Vert_{L^2(\Omega)} \leq Ch^2 \Vert U \Vert_{H^2(\Omega)}. \label{eq:global_bd_interm_f.e}    
%%\end{equation}
%%A similar result hold for the symmetric  discontinuous Galerkin  approximation where $\epsilon = -1$ and $\sigma$ large enough. 
%%\begin{equation}
    %%\Vert U - u^{\DG}_h \Vert_{L^2(\Omega)} \leq Ch^2 \Vert U \Vert_{H^2(\Omega)}.     \label{eq:global_bd_dg}
    %%\end{equation}
%%If $\beta > 3/2$, the above result also holds for the non-symmetric ($\epsilon = 1)$ and for the incomplete  ($\epsilon = 0$)  cases if $\sigma$ is large enough. For a proof of this estimate, we refer to  Theorem 2.13 in \cite{riviere2008discontinuous}.  
Since the domain $\Omega$ is convex, we have the following elliptic regularity result: 
\begin{equation}
    \Vert U \Vert_{H^2(\Omega)} \leq C \Vert f_h \Vert_{L^2(\Omega)}. \label{eq:U_convex_bd} 
\end{equation} 
Using the bounds \eqref{eq:global_bd_interm_f.e} and \eqref{eq:global_bd_dg} in \eqref{eq:triangle411}  yields: 
\begin{equation}
    \Vert u - u^{\DG}_h \Vert_{L^2(\Omega)} \leq  \Vert u - \uCG \Vert_{L^2(\Omega)}  + Ch^{2} \Vert f_h \Vert_{L^2(\Omega)}.  \label{eq:interm_bd}
\end{equation}
Bounds \eqref{eq:L2estimate_F.E} and \eqref{eq:bd_fh} give \eqref{eq:global_dg_bd}. Under the additional assumptions, bounds \eqref{eq:improved_L2estimate_F.E_h} and \eqref{eq:improved_bd_fh} yield \eqref{eq:improved_global_dg_bd}.
\end{proof}
Hereinafter, we only consider the symmetric dG discretization ($\epsilon = -1$) and we set $\beta = 1$. Hence, for simplicity, we denote by $ a = a_{-1}$. We also assume that $\Lambda$ is a $\mathcal{C}^2$ curve, $ f \in L^2(\Lambda)$, and that
$|E \cap \Lambda| \leq C h, \,\, \forall E \in \mesh_h.$ Therefore, with \eqref{eq:U_convex_bd} and \eqref{eq:improved_bd_fh}, there is a constant $C$ independent of $h$ such that:
\begin{equation}\label{eq:boundhU}
h \Vert U \Vert_{H^2(\Omega)} \leq C \Vert f \Vert_{L^2(\Lambda)}.
\end{equation}
We recall Lemma 4.1 proved by Chen and Chen in \cite{chen2004pointwise}. Consider any two sets $D, \tilde{D} \subset \Omega$ such that the  $\mathrm{distance}$ between $D$ and $(\partial \tilde{D} \backslash \partial D)$ is strictly positive. Then, for $h$ small enough, 
we have
 \begin{equation}
\vertiii{U-\uDG}_{\DG(D)} \leq C(h^{k} \|U\|_{H^{k+1}(\widetilde{D})} +   \|U-\uDG\|_{L^2(\widetilde{D})}). \label{eq:Lemma_chen_chen}
 \end{equation}

%\input{local_energy_estimate.tex}
%\input{local_l2_estimate.tex}
%Hereinafter, we only consider the symmetric dG discretization ($\epsilon = -1$) and we set $\beta = 1$. Hence, for simplicity, we denote by $ a = a_{-1}$. We also assume that $\Omega$ is convex, $\Lambda$ is a $\mathcal{C}^2$ curve, $ f \in L^2(\Lambda)$, and that 
%$|E \cap \Lambda| \leq C h, \,\, \forall E \in \mesh_h.$ 
\section{Weighted Energy Estimate} \label{sec:weighted_global_conv}
We first show that the dG solution is stable in the weighted energy norm defined by: 
\begin{align}
\|v\|_{\DG, \alpha}^2 &= \sum_{E \in \mesh_h} \|\nabla v\|_{L^2_\alpha (E)}^2 + \sum_{e\in \Gamma_h \cup \partial \Omega} \frac{\sigma}{h} \|d^{\alpha} [v]\|^2_{L^2(e)}, \quad v \in H^1(\mesh_h), \,\, \alpha \in (0,1). 
\end{align}
\begin{lemma}[Stability]\label{lemma:stability_dg_sol_alpha} 
For $\alpha \in (0,1)$, there exists a constant $C_{\alpha}$ independent of $h$ but dependent on $\max_{\bm{x} \in \Omega} d^{2\alpha}(\bm{x})$ such that the dG solution, $u_h^{\DG}$, satisfies: 
\begin{equation}
\|\uDG\|_{\DG,\alpha} \leq C_{\alpha}( \|f\|_{L^2(\Lambda)} + | u |_{H^1_{\alpha}(\Omega)}). 
\end{equation}
\end{lemma}
\begin{proof}
Recall the intermediate problem \eqref{eq:intermediate_pb}. %Since $U \in H^2(\Omega)$, the following standard estimate holds \cite{riviere2008discontinuous}: 
%\begin{equation}
%\|U - u_h^{\DG} \|_{\DG} + \|\nabla(U - \uCG)\|_{L^2(\Omega)} \leq C h \| U \|_{H^2(\Omega)} \leq C\| f\|_{L^2(\Lambda)}, \label{eq:error_eq_energy_interm_pb}
%\end{equation}
Since $U \in H^2(\Omega) \cap H_0^1(\Omega) $, we immediately have with \eqref{eq:xiDGfull} and \eqref{eq:boundhU}
\begin{equation}
    \sum_{e\in \Gamma_h \cup \partial \Omega} \frac{\sigma}{h} \| d^{2\alpha} [u_h^{\DG}] \|_{L^2(e)}^2 
\leq  \|d^{2\alpha}\|^2_{L^{\infty}(\Omega)}\sum_{e\in \Gamma_h \cup \partial \Omega} \frac{\sigma}{h} \|[u_h^{\DG} - U] \|_{L^2(e)}^2  
\leq C \|d^{2\alpha}\|^2_{L^{\infty}(\Omega)} \| f\|^2_{L^2(\Lambda)}. \label{eq:bound_jump_terms}
\end{equation}
We use the triangle inequality, \eqref{eq:global_bd_interm_f.e} and \eqref{eq:boundhU}: %  and \eqref{eq:stability_cg_discrete}:  
\begin{align}
\|\nabla U \|_{L^2_\alpha(\Omega)}  
 \leq \|d^{2\alpha}\|_{L^\infty(\Omega)}\|\nabla(U - \uCG) \|_{L^2(\Omega)} + \|\nabla \uCG\|_{L^2_{\alpha}(\Omega)}
\leq C_{\alpha} \|f\|_{L^2(\Lambda)} + \|\nabla \uCG\|_{L^2_{\alpha}(\Omega)}.
\end{align}
From Theorem 3.5 in \cite{drelichman2020weighted} and Lemma~\ref{lemma:equiv_weak_forms}, we have  
\begin{equation}
\|\nabla \uCG\|_{L^2_\alpha(\Omega)} \leq C \|\nabla u\|_{L^2_{\alpha}(\Omega)}, \quad \alpha \in (0,1). \label{eq:stability_cg_discrete}
\end{equation}
This implies
\[
\|\nabla U \|_{L^2_\alpha(\Omega)} \leq C_\alpha \|f\|_{L^2(\Lambda)} + C | u |_{H^1_{\alpha}(\Omega)}.
\]
By the triangle inequality, \eqref{eq:xiDGfull},  \eqref{eq:boundhU} and the above bound, we obtain
\begin{align}
&\sum_{E\in \mesh_h} \| \nabla \uDG \|_{L^2_{\alpha}(E)}^2 \leq 2 \sum_{E \in \mesh_h} \| \nabla (\uDG - U)\|_{L^2_{\alpha}(E)}^2 + 2 \sum_{E \in \mesh_h}\|\nabla U \|_{L^2_{\alpha}(E)}^2  \nonumber \\ & \leq C_{\alpha} \| \uDG -U \|^2_{\DG} +2 \|\nabla U\|^{2}_{L^2_{\alpha}(\Omega)} \leq C_{\alpha}( \|f\|_{L^2(\Lambda)} + | u |_{H^1_{\alpha}(\Omega)})^{2}. \label{eq:bound_volume_terms}
\end{align}
 We conclude the result by combining \eqref{eq:bound_jump_terms} and \eqref{eq:bound_volume_terms}.
\end{proof}
We have an a priori bound for $U$ in the $H^2_\alpha$ norm, which can be seen as a generalization of \eqref{eq:boundhU}. We denote by $\bar{d}_E = \max_{\bm{x} \in E} d(\bm{x}, \Lambda)$ for $E \in \mesh_h$. 
\begin{lemma}\label{lemma:weighted_reg_U}
%\Rd Assume that $\partial\Omega$ is $C^2$. \Bk
For $\alpha \in (-1, 1)$, there exists a constant $C$ independent of $h$ such that 
\begin{align}
    \|U\|_{H^2_\alpha(\Omega)} &\leq C h^{\alpha - 1} \|f\|_{L^2(\Lambda)}, \quad \alpha \in (-1,1) . \label{eq:weighted_elliptic_regularity}  
   % \|U\|_{W^{2,3}_{\alpha}(\Omega)} & \leq C h^{\frac{2\alpha}{3}-3/2} \|f\|_{L^2(\Lambda)}, \quad \alpha \in (0,4). \label{eq:weighted_elliptic_regularity_w3}
    \end{align}
\end{lemma}
\begin{proof}
Since $d^{2\alpha} \in A_2$, it follows from Theorem 3.1 in \cite{ojea2021optimal} that %\Rd but this paper assumes $C^2$ smooth domain \Bk
 \begin{align}
\|U\|_{H^2_{\alpha}(\Omega)} \leq C \|f_h\|_{L^2_{\alpha}(\Omega)}. \label{eq:weighted_regularity_1} 
  \end{align}
 Thus, to show \eqref{eq:weighted_elliptic_regularity}, we find a bound on $\|f_h\|_{L^2_\alpha(\Omega)}$. %With the definition of $f_{h,E}$ (see \eqref{eq:def_fh}), Holder's inequality, an  inverse estimate , and \eqref{eq:poincare_avg}, we have 
Thanks to the shape-regularity of the mesh, for $E \in \mathcal{T}_{\Lambda}$, $ch_E \leq \bar{d}_E \leq C h_E$
(see Lemma 3.1 in \cite{d2012finite}). Hence, using \eqref{eq:equivalence_norms},\eqref{eq:improved_bd_fh} and \eqref{eq:quasi_uniform}, yield 
 \begin{align}
\|f_h\|^2_{L^{2}_\alpha(\Omega)} & = \sum_{E \in \mathcal{T}_\Lambda} \|d^{\alpha} f_h\|^2_{L^2(E)} \leq \sum_{E \in \mathcal{T}_\Lambda}  \bar{d}_E^{2\alpha}\|f_{h,E}\|^2_{L^2(E)} \nonumber \\ 
& \leq C h^{2\alpha} \sum_{E \in \mathcal{T}_\Lambda} \|f_{h,E}\|^2_{L^2(E)} \leq Ch^{2\alpha -2} \|f\|^2_{L^2(\Lambda)}. \label{eq:weighted_regularity_2}
 % = \sum_{E \in \tau_\Lambda}   \bar{d}_E^{2\alpha} \int_{E} f_{h,E} f  \\ 
%& \leq C \sum_{E \in \tau_\Lambda} \bar{d}_E^{2\alpha}  h_E^{-1} \|f_{h,E}\|_{L^2(E)} \|f\|_{L^2(\lambda \cap E)} \leq Ch^{-1} \left(\sum_{E \in \tau_\Lambda} \bar{d}_E^{4\alpha} \|f_{h,E}\|_{L^2(E)}^2\right)^{1/2} \|f\|_{L^2(\Lambda)}. 
 \end{align}
Substituting \eqref{eq:weighted_regularity_2} in \eqref{eq:weighted_regularity_1} yields \eqref{eq:weighted_elliptic_regularity}.% To show \eqref{eq:weighted_elliptic_regularity_w3}, 
%\begin{align*}
%\|d^{2\alpha} f_h\|^3_{L^3(\Omega)} \leq \sum_{E \in \tau_\Lambda} \bar{d^{2\alpha}_{E}} \|f_h\|^3_{L^3(E)} \leq C \sum_{E \in \tau_\Lambda} \bar{d^{2\alpha}_{E}} h^{-3/2} \|f_h\|^3_{L^2(E)} 
%\leq C h^{2\alpha - 1/2}\sum_{E \in \tau_\Lambda} \|f_h\|^3_{L^2(E)}. 
%\end{align*}
%Thus, by \eqref{eq:improved_bd_fh}, $$\|f_h\|_{L^3_\alpha(\Omega)} \leq C h^{2\alpha/3 - 1/2} \|f_h\|_{L^2(\Omega)} \leq C h^{2\alpha/3 -3/2}\|f\|_{L^2(\Lambda)}.$$
\end{proof}
% We also repeatedly use the cell averages of functions. To this end, we denote 
% $$\avg{v}= \frac{1}{|E|} \int_{E} v, \quad E \in \mesh_h, \,\, v \in L^1(E).   $$
% We recall the following Poincare inequality (see Section 5.8 in \cite{evans2010partial}). 
% \begin{equation}
% \|v  - \avg{v}\|_{L^r(E)} \leq Ch_E |v|_{W^{1,r}(E)}, \quad r \in [1,\infty], \,\, E \in \mesh_h, \,\, v \in W^{1,r}(E). \label{eq:poincare_avg}
% \end{equation}
%We denote by $\bar{d}_E = \max_{\bm{x} \in E} d(\bm{x}, \Lambda)$. 
The following equivalence of norms holds 
(see proof of Lemma 3.2 in \cite{d2012finite}).  There exist positive constants $\gamma_1, \gamma_2$ 
independent of $h$ such that for $-1 < \alpha < 1$, $E \in \mesh_h$, and $v_h \in \mathbbm{P}^k(E)$, 
\begin{equation}
\gamma_1 \|d^{\alpha} v_h\|_{L^2(E)} \leq \bar{d}_E^{\alpha} \|v_h\|_{L^2(E)} \leq \gamma_2 \|d^{\alpha} v_h\|_{L^2(E)}. \label{eq:equivalence_norms}
\end{equation}  
Note that with \eqref{eq:boundond} and the chain rule, we have for $E \in \mesh_h$, and $v \in L^{\infty}(E)$, 
\begin{align}{2}
\|v \nabla (d^{\alpha})  \|_{L^2(E)} & \leq  \alpha  \|d^{\alpha- 1} v \|_{L^2(E)}, \, \alpha>1/2\label{eq:prop_grad_dalpha}  \\ % \quad \alpha > \frac{1}{2}, \,\,  v \in L^{\infty}(E), \\ 
\|v \nabla^2 (d^{2\alpha}) \|_{L^2(E)} & \leq C\|d^{2\alpha - 2} v \|_{L^2(E)}, \,  3/2>\alpha >1/2. \Bk  \label{eq:prop_grad2_dalpha} % \,\,   \quad \alpha > \frac{1}{2}, v \in L^{\infty}(E).
\end{align}
In addition, since $d^{2\alpha} \in A_2$ for $\alpha \in (-1,1)$, we use the interpolant $\Pi_h: \mathring{H}^2_{\alpha}(\Omega) \rightarrow W_h^1(\mesh_h)$ introduced in \cite{nochetto2016piecewise}. 
This interpolant is independ ent of $\alpha$ and satisfies the following approximation properties (see Theorem~5.2 in \cite{nochetto2016piecewise}).  
For any $\alpha\in (-1,1)$ and for any  $w$ in $\mathring{H}^{2}_{\alpha}(\Omega)$, there is a constant $C$ independent of $h$ such that
\begin{align}
    \| w - \Pi_h w\|_{H^m_{\alpha}(E)} \leq Ch^{2-m}  \vert w \vert_{H^{2}_{\alpha}(\Delta_E)},  \quad 0\leq m \leq 2, \quad \forall E\in\mesh_h,\label{eq:approximation_weighted}
\end{align}
 where $\Delta_E$ is a macro element containing $E$. 
We also recall the definition of Kondratiev-type weighted Sobolev spaces, $V_{\alpha}^m(\Omega)$, for any $\alpha>0$ and 
$m\in \mathbb{N}$: 
\[
V_\alpha^m(\Omega) = \{ u \in L_{\alpha-m}^2(\Omega): \, \forall 0\leq \vert \bm{\beta}\vert  \leq m, \, d^{\vert \bm{\beta}\vert  +\alpha-m} D^{\bm{\beta}} u 
\in L^2(\Omega)\}, 
\]
equipped with the norm
\begin{equation}
\|u\|^2_{V^{m}_\alpha(\Omega)} = \sum_{s=0}^m |u|^2_{H^{s}_{\alpha -m +s }(\Omega)}, \quad m \geq 1.
\end{equation}
 Ariche et al. proved that the 
%We refer to Ariche et al. for conditions by which the 
solution $u$ to \eqref{eq:model_pb}-\eqref{eq:model_pb_2} belongs to $V^2_{1+\alpha}(\Omega)$  for $ \alpha \in (0,1)$ under certain conditions on $\Omega$ and $\Lambda$, see Theorem 1.1 in \cite{ariche2017regularity}. 
The main result of this section reads as follows. 
\begin{theorem}
Fix $\alpha \in (1/2,1)$ and let $\delta \in (0,\alpha)$. Assume that $u \in V^2_{1+\delta }(\Omega)$. For all $ 1< s < \frac{1}{1-\alpha}\,\, $,  there exist  constants $C$ and $C_\ast$ independent of $h$ such that if $\sigma > C_\ast$, 
\begin{equation}
\|\nabla_h (u - \uDG)\|_{L^2_{\alpha}(\Omega)} + \left(\sum_{e\in \Gamma_h \cup \partial \Omega} \frac{\sigma}{h} \|d^{\alpha} [\uDG]\|^2_{L^2(e)}\right)^{1/2} \leq C \left(h^{\alpha - \delta} + h^{1 -\frac{3}{2} s(1-\alpha)}\right) . 
\end{equation}
\end{theorem}
\begin{proof}
 Let $\uCG \in W_h^1(\mesh_{h})$ solve \eqref{eq:discrete_F.E} for $k=1$. We apply the triangle inequality.
\begin{align}
&\|\nabla_h (u - \uDG)\|_{L^2_{\alpha}(\Omega)} + \left(\sum_{e\in \Gamma_h \cup \partial \Omega} \frac{\sigma}{h} \|d^{\alpha} [\uDG]\|^2_{L^2(e)}\right)^{1/2} \nonumber \\ &   \leq \|\nabla(u-\uCG)\|_{L^2_{\alpha}(\Omega)} + \|U - \uDG\|_{\DG, \alpha} + \|\nabla(\uCG - U)\|_{L^2_{\alpha}(\Omega)} . \label{eq:triangle_inq_weighted}
\end{align}
Considering Lemma \ref{lemma:equiv_weak_forms}, the first term is bounded in Corollary 3.8 in \cite{d2012finite} 
\begin{equation}
\|\nabla(u-\uCG)\|_{L^2_{\alpha}(\Omega)} \leq Ch^{\alpha - \delta} |u|_{V^{2}_{1+\delta}(\Omega) }. 
    \label{eq:weighted_finite_element_dAngelo}
\end{equation} 
Bound \eqref{eq:weighted_finite_element_dAngelo} can also be derived from Theorem 3.5 in \cite{drelichman2020weighted} and Theorem 3.6 in \cite{d2012finite}. It remains to bound $\Vert U - \uDG\|_{\DG, \alpha}$ and $\Vert \nabla(\uCG - U)\|_{L^2_{\alpha}(\Omega)}$, 
which is the object of Lemma~\ref{lemma:conv_dg_alpha_local} and   Lemma~\ref{lemma:weighted_bd_fe_U} respectively.
\end{proof}

% We now present and show the main convergence in weighted energy norms  result of this section. 
\begin{lemma}\label{lemma:conv_dg_alpha_local}
For $\alpha \in  (\frac{1}{2},1)$, there exists a constant $C_\ast$  independent of $h$ such that if $\sigma > C_\ast$,  
\begin{equation}
    \label{eq:conv_dg_alpha}
    \|U - u_h^{\DG}\|_{\DG,\alpha} \leq C  (h^\alpha + h^{1 -\frac{3}{2} s(1-\alpha)}),  \quad \forall  1 < s< \frac{1}{1-\alpha}.
 \end{equation}
% \begin{equation}
% \|U -\uDG\|^2_{\DG,\alpha} \leq  \begin{cases}
%         C(h^{2\alpha} + h^{1/2}),  & \alpha \in (1/2,1) \\ 
%           C (h^{2(1-\delta)} + h^{2}), & \alpha = 1+ \delta, \delta > 0.  
%        \end{cases}\label{eq:conv_dg_alpha}
%     \end{equation}
    \end{lemma}
    \begin{proof}
    Let $\chi_h = \Pi_h U -\uDG$. With triangle inequality and the bounds \eqref{eq:approximation_weighted}, \eqref{eq:xiDGfull}, \eqref{eq:global_bd_dg}, \eqref{eq:boundhU}, we have
    \begin{equation}
    \|\chi_h\|_{L^2(\Omega)} + h\|\chi_h\|_{\DG} \leq C h^2\|U\|_{H^2(\Omega)}\leq  C h \|f\|_{L^2(\Lambda)}. \label{eq:standard_error_chih}
    \end{equation}
  With several manipulations, as is done in \cite{waluga2013quasi}, we  have formally
    \begin{align}
    \|\chi_h\|_{\DG,\alpha}^2 = a(\chi_h, d^{2\alpha}\chi_h) &
- 2\sum_{E \in \mesh_h} \int_{E} (d^{\alpha} \nabla \chi_h \cdot ( \chi_h  \nabla (d^{\alpha}) ) \nonumber\\ &  + 2 \sum_{e\in \Gamma_h \cup \partial \Omega} \int_e \{\nabla (d^{\alpha} \chi_h)\}\cdot \bm{n}_e [d^{\alpha} \chi_h] = \sum_{i=1}^3 T_i.  \label{eq:first_expansion_weighted}
    \end{align}
We now explain why each term $T_i$ above is well defined.
    %Throughout the proof, we present bounds that show that $T_1,T_2$ and $T_3$ are well defined. For now, it suffices to briefly state the main arguments to demonstrate that. 
    From \eqref{eq:prop_grad_dalpha}-\eqref{eq:prop_grad2_dalpha}, the term $T_1$ is well defined since $d^{2\alpha} \chi_h \in H^2(\mesh_h)$. Property \eqref{eq:prop_grad_dalpha} and Cauchy-Schwarz's inequality guarantee that $T_2$ is well defined. For $T_3$, we write
\[
\{ \nabla (d^{\alpha} \chi_h)\}\cdot \bm{n}_e [d^{\alpha} \chi_h] =
 \{ d^\alpha \nabla (d^{\alpha} \chi_h)\}\cdot \bm{n}_e [\chi_h].
\]
Observe that since $\chi_h$ is a polynomial, the function $d^\alpha \nabla (d^\alpha \chi_h)$ belongs to $H^1(\mesh_h)^3$.  Indeed we have
\[
d^\alpha \nabla (d^\alpha \chi_h) = \alpha d^{2\alpha-1} \chi_h \nabla d 
+ d^{2\alpha} \nabla \chi_h,
\]
and with \eqref{eq:prop_grad2_dalpha}, each term belongs to $H^1(E)$ for each mesh element $E$. 
%\Rd and $d^{2\alpha}, d^{2\alpha-1}$ are  in $L^1_{loc}$ ?? do we need to show $d^\alpha\chi$ is in $H^2(E)$? \Bk. and by \eqref{eq:prop_grad2_dalpha}, $d^{2\alpha} \nabla \chi_h \in (H^1(\mesh_h))^3$ and $d^{2\alpha - 1} \nabla d \chi_h \in (H^1(\mesh_h))^3$. 
This implies that $\|\{d^{\alpha} (\nabla d^{\alpha} \chi_h)\} \|_{L^2(e)}$ is bounded and the term $T_3$ is well defined.   
 To handle the first term, we use the following Galerkin orthogonality
    \begin{equation}
       a(U - \uDG, v_h ) = 0, \quad \forall v_h \in V_h^{k}(\mesh_h). \label{eq:consistency_interm_pb}
    \end{equation}
Let $\eta = \Pi_h U - U $ and $\xi = U -\uDG$ so that $\chi_h = \eta + \xi$. Since $[d^{\alpha}\eta] = 0$ a.e. on $e \in \Gamma_h \cup \partial \Omega$, we have
    \begin{align*}
     T_1 & = a(\eta, d^{2\alpha} \chi_h ) + a(\xi, d^{2\alpha} \chi_h - w_h) \\
& = \sum_{E \in \mesh_h} \int_E \nabla \eta  \cdot \nabla (d^{2\alpha} \chi_h )  - \sum_{e\in\Gamma_h \cup \partial \Omega } \int_e \{d^{\alpha}\nabla \eta \} \cdot \bm{n}_e [d^{\alpha} \chi_h]  +  a(\xi, d^{2\alpha} \chi_h - w_h) \\
& = \sum_{i=1}^3 T_{1,i},
    \end{align*}
    where $w_h \in V_h^1(\mesh_h)$ is a piecewise Lagrange interpolant of $d^{2\alpha} \chi_h$ such that 
%(see Section 4.3 in \cite{arnold2002unified}
%\Rd but this paper does not exhibit an interpolant, it just says we need one for the analysis to work.  Can we simply use the Lagrange interpolant element by element? It is sufficient for what we need here? \Bk)
    \begin{equation}
    \vertiii{d^{2\alpha} \chi_h - w_h}_{\DG} \leq Ch |d^{2\alpha} \chi_h |_{H^2(\mesh_h)}. \label{eq:bd_T1_0}
    \end{equation}
We begin by bounding $T_{1,3}$. With \eqref{eq:continuity_prop}, \eqref{eq:xiDGfull}, \eqref{eq:bd_T1_0}, we have 
\begin{equation} 
        T_{1,3} = a(\xi, d^{2\alpha} \chi_h - w_h)  \leq  C \vertiii{\xi}_{\DG} \vertiii{ d^{2\alpha} \chi_h  - w_h}_{\DG} 
\leq C h^2 \|U\|_{H^2(\Omega)} |d^{2\alpha} \chi_h |_{H^2(\mesh_h)}. \label{eq:bound_T13int}
\end{equation}
 Using \eqref{eq:boundond} and \eqref{eq:prop_grad2_dalpha}, we obtain
\[
|d^{2\alpha} \chi_h |_{H^{2}(\mesh_h)}    \leq C \|d^{2\alpha -2 } \chi_h\|_{L^2(\Omega)} + C \|d^{2\alpha -1} \nabla_h \chi_h\|_{L^2(\Omega)}.
\]
Since  $d^{\gamma} \in  L^2(\Omega)$ for $ |\gamma| <1 $, we have $d^{2(\alpha - 1)} \in  L^{\frac{1}{s(1-\alpha)}}(\Omega)$ for $1 < s < \frac{1}{2(1-\alpha)}$. Note that $\frac{1}{s(1-\alpha)} > 2$.   
    Further, since $\chi_h \in V_h^{1}(\mesh_h)$ and  by using  and H\"{o}lder's inequality, we have 
    \begin{align} 
        |d^{2\alpha} \chi_h |_{H^{2}(\mesh_h)}    
%\leq \|d^{2\alpha -2 } \chi_h\| + \|d^{2\alpha -1} \nabla_h \chi_h\| \\  
\leq  & C \| d^{2\alpha-2}\|_{L^{\frac{1}{s(1-\alpha)}}(\Omega)} \|\chi_h\|_{L^{\frac{2}{1-2s(1-\alpha)}}(\Omega)} + \|d^{\alpha-1}\|_{L^{\frac{2}{s(1-\alpha)}}(\Omega)} \|d^{\alpha}\nabla_h  \chi_h \|_{L^{\frac{2}{1-s(1-\alpha)}}(\Omega)} \nonumber\\
 \leq & C \|\chi_h\|_{L^{\frac{2}{1-2s(1-\alpha)}}(\Omega)}+ \|d^{\alpha}\nabla_h  \chi_h \|_{L^{\frac{2}{1-s(1-\alpha)}}(\Omega)}.   \label{eq:h2_d2alpha_chi_0}
    \end{align}
By inverse estimate \eqref{eq:3d_1d_inverse_estimate}($q= 2/(1-2s(1-\alpha)),  p =2 $) and \eqref{eq:standard_error_chih}, we have 
\begin{equation}
\|\chi_h\|_{L^{\frac{2}{1-2s(1-\alpha)}}(\Omega)} \leq C h^{-3s(1-\alpha)}\|\chi_h\|_{L^2(\Omega)} \leq C h^{-3s(1-\alpha) + 1}
\Vert f \Vert_{L^2(\Lambda)}. \label{eq:bound_chi_L2_Lq}
\end{equation}
For the second term, we first derive an inverse inequality for any $v_h \in V_h^k(\mesh_h)$ and $q \geq 2$. With the local version of the  inverse inequality \eqref{eq:3d_1d_inverse_estimate}, \eqref{eq:equivalence_norms}  and Jensen's inequality, we have 
\begin{multline}
\|d^{\alpha} v_h\|_{L^q(\Omega)} \leq  \left(\sum_{E \in \mesh_h} \bar{d}_{E}^{\alpha q}\|v_h\|_{L^q(E)}^q \right)^{1/q} \leq C h^{\frac{3}{q} - \frac{3}{2}}\left(\sum_{E \in \mesh_h}\bar{d}_{E}^{\alpha q} \|v_h\|_{L^2(E)}^q \right)^{1/q} \\ 
\leq C h^{\frac{3}{q} - \frac{3}{2}}\left(\sum_{E \in \mesh_h} \|v_h\|_{L^2_{\alpha} (E)}^q \right)^{1/q} \leq C h^{\frac{3}{q} - \frac{3}{2}} \|v_h \|_{L^2_{\alpha}(\Omega)}. \label{eq:inverse_estimate_lq_alpha}
\end{multline}
Hence, with \eqref{eq:inverse_estimate_lq_alpha},  the second term in \eqref{eq:h2_d2alpha_chi_0} is bounded as  
\begin{equation}
\|d^{\alpha} \nabla_{h} \chi_h\|_{L^{\frac{2}{1-s(1-\alpha)}}(\Omega)} \leq Ch^{-\frac{3}{2}s(1-\alpha)} \|\nabla_h \chi_h\|_{L^2_\alpha(\Omega)}. \label{eq:bd_dalpha_chih_Lq}
\end{equation}
Thus, with \eqref{eq:bound_chi_L2_Lq} and \eqref{eq:bd_dalpha_chih_Lq}, \eqref{eq:h2_d2alpha_chi_0} reads 
\begin{equation}
   |d^{2\alpha} \chi_h|_{H^2(\mesh_h)} \leq C (h^{-3s(1-\alpha) + 1} +h^{-\frac{3}{2}s(1-\alpha)} \|\nabla_h \chi_h\|_{L^2_\alpha(\Omega)}.)\label{eq:bd_H2_chi}
\end{equation}
Thus, with \eqref{eq:boundhU} and \eqref{eq:bd_H2_chi}, \eqref{eq:bound_T13int} reads
    %With \eqref{eq:bd_H2_chi}, \eqref{eq:bd_T1_0} reads 
%    \begin{equation}
%        \vertiii{d^{2\alpha} \chi_h - w_h}_{\DG} \leq C( h^{2-3s(1-\alpha)} +h^{1 -\frac{3}{2}s(1-\alpha) } \|\nabla_h \chi_h\|_{L^2_\alpha(\Omega)}) . 
%    \end{equation}
%    Hence, with the above bound, continuity property \eqref{eq:continuity_prop}, a standard error estimate, the term $T_{1,3}$ is bounded as follows. 
    \begin{equation} 
        T_{1,3} 
\leq C( h^{2-3s(1-\alpha)} +h^{1 -\frac{3}{2}s(1-\alpha)} \|\nabla_h \chi_h\|_{L^2_\alpha(\Omega)}). \label{eq:bound_T13}
    \end{equation}
    We now turn to $T_{1,1}$ and $T_{1,2}$. We write 
    \begin{align*}
    T_{1,1} & = \sum_{E \in \mesh_h} \int_E \nabla \eta \cdot d^{2\alpha }\nabla \chi_h + \int_E \nabla \eta \cdot 2\alpha  d^{2\alpha - 1} \nabla d \, \chi_h  \\ 
& \leq  \|\nabla \eta \|_{L^2_{\alpha}(\Omega)} \|\nabla_h \chi_h \|_{L^2_{\alpha}(\Omega)} +  C \|\nabla \eta \|_{L^2_{2\alpha - 1}(\Omega)} \|\chi_h\|_{L^2(\Omega)}. 
    \end{align*}
    With \eqref{eq:approximation_weighted}, \eqref{eq:weighted_elliptic_regularity},  \eqref{eq:global_bd_dg} and \eqref{eq:boundhU}, we obtain  
    \begin{align}
    |T_{1,1}| &\leq  Ch |U|_{H^2_{\alpha}(\Omega)} \|\nabla_h \chi_h\|_{L^2_{\alpha}(\Omega)} + C h^2 |U|_{H^{2}_{2\alpha -1} (\Omega)} \nonumber \\
    & \leq Ch^{\alpha} \|\nabla_h \chi_h\|_{L^2_{\alpha}(\Omega)} + Ch^{2\alpha}. \label{eq:bound_T11}
    \end{align}
%     \begin{align}
%   T_{1,1}   & \leq  \begin{cases}
%         Ch |U|_{H^2_{\alpha}(\Omega)} \|\nabla_h \chi_h\|_{L^2_{\alpha}(\Omega)} + C h^2 |U|_{H^{2}_{2\alpha -1} (\Omega)}, & \alpha \in (1/2,1), \\ 
%         Ch \|d^{2\delta}\|_{L^{\infty}(\Omega)} |U|_{H^2_{\alpha - 2\delta}(\Omega)} \|\nabla_h \chi_h\|_{L^2_{\alpha}(\Omega)} + C h^2\|d^{4\delta}\|_{L^{\infty}(\Omega)} |U|_{H^{2}_{2(\alpha - 2\delta) -1}(\Omega)},  & \alpha = 1+ \delta.  
%     \end{cases} \nonumber \\ & \leq   \begin{cases}
%         Ch^{\alpha} \|\nabla_h \chi_h\|_{L^2_{\alpha}(\Omega)} + Ch^{2\alpha}, & \alpha \in (1/2,1), \\ 
%         Ch^{\alpha - 2\delta} \|\nabla_h \chi_h\|_{L^2_{\alpha}(\Omega)}  + C h^{2(\alpha - 2\delta)}, & \alpha = 1+ \delta. \label{eq:bound_T11}  
%     \end{cases}
%     \end{align}
To handle $T_{1,2}$, consider a mesh element $E$ and let $e\in \partial E$. Since $d^\alpha \eta$ belongs to $H^1_\alpha(\Omega)$, trace estimate  \eqref{eq:3d_1d_trace_estimate_continuous} yields 
\begin{align*}
    \|d^{\alpha} \nabla \eta\|_{L^2(e)}  \leq C h^{-1/2} \|d^{\alpha} \nabla \eta\|_{L^2(E)}  + Ch^{1/2} ( \|d^{\alpha} \nabla^2 \eta\|_{L^2(E)} + \|d^{\alpha - 1} \nabla \eta\|_{L^2(E)} ).
\end{align*}
Thus, with Cauchy-Schwarz's inequality, \eqref{eq:approximation_weighted} and \eqref{eq:weighted_elliptic_regularity}, we obtain 
\begin{align}
|T_{1,2}| \leq C \left(\|\nabla \eta\|_{L^2_{\alpha}(\Omega)} + h(\|U\|_{H^2_{\alpha}(\Omega)} +  h \Vert U \Vert_{H_{\alpha-1}^2(\Omega)}
)\right)\|\chi_h\|_{\DG,\alpha} \leq  Ch^{\alpha} \|\chi_h\|_{\DG,\alpha} \label{eq:bd_T1_2}.
\end{align}
For $T_2$, we apply   Cauchy-Schwarz's inequality and \eqref{eq:boundond},  
    \begin{align}
    |T_2| & \leq \|\nabla_h \chi_h\|_{L^2_{\alpha}(\Omega)} \|d^{\alpha-1} \chi_h\|_{L^2(\Omega)}. 
    \end{align}
    With \eqref{eq:approximation_weighted}, Holder's inequality, the observation that $d^{\alpha - 1} \in L^{\frac{2}{s(1-\alpha)}}(\Omega)$ 
    , \eqref{eq:standard_error_chih}, and \eqref{eq:3d_1d_inverse_estimate}, we obtain 
    \begin{multline}
    |T_2| \leq  \|\nabla_h \chi_h \|_{L^2_{\alpha}(\Omega)} \|d^{\alpha-1}\|_{L^{\frac{2}{s(1-\alpha)}}(\Omega)}  \|\chi_h\|_{L^{\frac{2}{1-s(1-\alpha)}}(\Omega)} \\  \leq C \|\nabla_h \chi_h \|_{L^2_\alpha(\Omega)} h^{-\frac{3}{2}s(1-\alpha) } \|\chi_h\|_{L^2(\Omega)} \leq Ch^{-\frac{3}{2}s(1-\alpha) +1 } \|\nabla_h \chi_h \|_{L^2_\alpha(\Omega)}. \label{eq:bd_T2} 
    \end{multline}
  Hence, with \eqref{eq:bound_T13}, \eqref{eq:bound_T11}, \eqref{eq:bd_T1_2}, \eqref{eq:bd_T2}, and Young's inequality, we obtain
  \begin{align}
    |T_1| + |T_2| \leq \frac{1}{8}\|\chi_h\|^2_{\DG,\alpha} + C(h^{2-3s(1-\alpha)} + h^{2\alpha}).\label{eq:bound_T1_T2}
  \end{align}
  It remains to handle $T_3$. Fix a face $e \in \Gamma_h$, shared by two elements, $e = \partial E_e^1 \cap \partial E_e^2$. 
We write
% and the following term. The other term for $e \in \Gamma_h$ and the term corresponding to $e \in \partial \Omega$ are handled similarly. 
    \begin{align*}
     \int_e (\nabla (d^{\alpha}\chi_h))\vert_{E_e^1}\cdot \bm{n}_e [d^{\alpha}\chi_h] 
= & \int_e d^{\alpha} \nabla \chi_h \vert_{E_e^1}  \cdot \bm{n}_e  [d^{\alpha} \chi_h] 
+ \int_e (\alpha d^{\alpha - 1}\nabla d  \cdot \bm{n}_e  ) \chi_h\vert_{E_e^1} [d^{\alpha} \chi_h] \\
= & A_{e,1} +A_{e,2}.
    \end{align*}
    For $A_{e,1}$, recall the definition of $\bar{d}_{E_e^1}$.  With  \eqref{eq:3d_1d_trace_estimate_discrete} and \eqref{eq:equivalence_norms},  we have 
    \begin{align}
    A_{e,1} &\leq C \bar{d}^{\alpha}_{E_e^1} \|\nabla \chi_h \|_{L^2(E_e^1)} h^{-1/2}\|[d^{\alpha}\chi_h]\|_{L^2(e)} \leq C \gamma_2 \|\nabla \chi_h\|_{L^2_{\alpha}(E_e^1)}  h^{-1/2}\|[d^{\alpha}\chi_h]\|_{L^2(e)}.
    %& \leq C \left(\|\nabla \xi\|_{L^2_{\alpha}(E_e^1)} + h|U|_{H^2_{\alpha}(\Delta_{E_e^1})}\right)h^{-1/2}\|[d^{\alpha}\xi]\|_{L^2(e)}. \nonumber
    \end{align}
    Hence, with Young's  inequality, we obtain for a positive constant $C_0$
    \begin{equation}
    \sum_{e \in \Gamma_h \cup\partial\Omega } A_{e,1} \leq \frac{1}{16} \Bk \|\nabla_h \chi_h \|_{L^2_{\alpha}(\Omega)}^2
 + C_0 \sum_{e\in \Gamma_h \cup \partial \Omega} h^{-1} \|[d^{\alpha} \chi_h ]\|^2_{L^2(e)}. \label{eq:bound_Ae1}
    \end{equation}
For the term $A_{e,2}$, we have with \eqref{eq:boundond}
    \begin{align}
    A_{e,2} \leq  \alpha   \|d^{2\alpha - 1} \nabla d  \cdot \bm{n}_e \Bk \, \chi_h\vert_{E_e^1}\|_{L^2(e)} \|[\chi_h]\|_{L^2(e)}\leq C \|d^{2\alpha -1} \chi_h \vert_{E_e^1}\|_{L^{2}(e)} \|[\chi_h]\|_{L^2(e)}. \nonumber
    \end{align}
    With the trace inequality \eqref{eq:3d_1d_trace_estimate_continuous},  H\"older's inequality and  \eqref{eq:boundond},
 we have 
    \begin{align}
    \|d^{2\alpha -1} \chi_h \vert_{E_e^1}\|_{L^2(e)} &\leq Ch^{-1/2} \|d^{2\alpha-1}\chi_h \|_{L^2(E_e^1)} + C h^{1/2}\|\nabla (d^{2\alpha -1}) \chi_h \|_{L^2(E_e^1)} + Ch^{1/2} \|d^{2\alpha -1} \nabla \chi_h \|_{L^2(E_e^1)} \nonumber \\ 
    & \leq  Ch^{-1/2} \|d^{2\alpha - 1}\|_{L^\infty(\Omega)}\|\chi_h\|_{L^2(E_e^1)} + Ch^{1/2}\| d^{2\alpha-2}\|_{L^{\frac{1}{s(1-\alpha)}}(E_e^1)} \|\chi_h\|_{L^{\frac{2}{1-2s(1-\alpha)}}(E_e^1)}\nonumber \\ & \quad + C h^{1/2} \|d^{\alpha-1}\|_{L^{\frac{2}{s(1-\alpha)}}(E_e^1)} \|d^{\alpha}\nabla_h  \chi_h \|_{L^{\frac{2}{1-s(1-\alpha)}}(E_e^1)}  \nonumber . 
  %  & \leq  Ch^{-1/2}\|\chi_h\|_{L^2(E_e^1)} + Ch^{1/2}h^{-3s(1-\alpha)}\|\chi_h\|_{L^2(E_e^1)} + C h^{1/2}  \|\nabla \chi_h\|_{L^2(E_e^1)}.  \nonumber
    \end{align}
 With similar arguments as the derivation of bound \eqref{eq:bd_H2_chi}, with \eqref{eq:standard_error_chih}, \eqref{eq:bound_chi_L2_Lq}, \eqref{eq:bd_dalpha_chih_Lq}, and H\"older's  inequality, we obtain 
    \[
    \sum_{e \in \Gamma_h \cup\partial\Omega } A_{e,2} 
\leq C (h^{-3s(1-\alpha) + 2} + h^{-\frac{3}{2}s(1-\alpha) + 1} \|\nabla_h \chi_h\|_{L_{\alpha}^2(\Omega)})
\left( \sum_{e\in \Gamma_h \cup\partial\Omega}  h^{-1} \|[\chi_h]\|^2_{L^2(e)} \right)^{1/2}.
\]
With Young's inequality and the bound \eqref{eq:standard_error_chih}, this leads to
\begin{equation}
    \sum_{e \in \Gamma_h \cup\partial\Omega } A_{e,2} 
 \leq \frac{1}{16} \|\nabla_h \chi_h\|_{L^2_{\alpha}(\Omega)}^2 + C h^{-3s(1-\alpha) + 2}. 
       \label{eq:bound_Ae2}
\end{equation}
Therefore we can bound $T_3$ with \eqref{eq:bound_Ae1} and \eqref{eq:bound_Ae2}.
\begin{equation}
|T_3| \leq \frac14 \|\nabla_h \chi_h\|_{L^2_{\alpha}(\Omega)}^2 + C_0 \sum_{e\in \Gamma_h \cup \partial \Omega} h^{-1} \|[d^{\alpha} \chi_h ]\|^2_{L^2(e)} + C h^{-3s(1-\alpha)+2}.
\label{eq:bound_T3}
\end{equation}
    We substitute  \eqref{eq:bound_T1_T2}, \eqref{eq:bound_T3} in \eqref{eq:first_expansion_weighted}. With the assumption that $\sigma > 4 C_0$, we obtain the result with an application of triangle's inequality  and the bound $\Vert U-\Pi_h U\Vert_{\DG,\alpha} \leq C h^\alpha$.  
%\Rd I add some details:
%\[
%\Vert \chi_h \Vert_{DG,\alpha}^2 \leq C (h^{2\alpha}+ h^{2-3s(1-\alpha)})
%\]
%\[
%\Vert U-u_h^{DG}\Vert_{DG,\alpha} \leq \Vert U-\Pi_h U\Vert_{DG,\alpha} + C (h^{2\alpha}+ h^{2-3s(1-\alpha)})^{1/2}
%\]
%We have
%\[
%\Vert U-\Pi_h U\Vert_{DG,\alpha} \leq C h \vert U\vert_{H_\alpha^2(\Omega)} \leq h h^{\alpha-1}  \leq C h^\alpha
%\]
%We have then
%\[
%\Vert U-u_h^{DG}\Vert_{DG,\alpha} \leq  C h^{1-\frac32 s(1-\alpha)} + C h^\alpha
%\]
%\Bk
\end{proof}
\begin{lemma}\label{lemma:weighted_bd_fe_U}
    For $\alpha \in (1/2,1)$, there exists a constant $C$ independent of $h$ such that 
    \begin{equation}
    \|\nabla (U -\uCG)\|_{L^2_{\alpha}(\Omega)} \leq   C ( h^{\alpha} + h^{1 -\frac{3}{2} s(1-\alpha)})  ,   \quad \forall  1 < s< \frac{1}{1-\alpha}.
    \end{equation}
    \end{lemma}
    \begin{proof}
    Let $\zeta_h = \Pi_h U - \uCG$. We have 
    \begin{align*}
      \sum_{E \in \mesh_h } \int_{E} d^{2\alpha}\nabla \zeta_h  \cdot \nabla \zeta_h   
= \sum_{E \in \mesh_h} \int_E \nabla \zeta_h \cdot \nabla(d^{2\alpha}\zeta_h ) 
-  2   \sum_{E \in \mesh_h} \int_E d^{\alpha} \zeta_h \nabla \zeta_h  \cdot \nabla (d^{\alpha})  = X_1 + X_2 . 
    \end{align*}
Let $w_h$ be the continuous Lagrange interpolant of $d^{2\alpha} \zeta_h$. 
%    Similar to the proof of step 2, we note that there exists a $q_h \in W_h^1(\mesh_h)$ such that 
    \begin{equation}
    \|\nabla(d^{2\alpha} \zeta_h - w_h )\|_{L^2(\Omega)} \leq C h|d^{2\alpha} \zeta_h|_{H^2(\mesh_h)}.
    \end{equation}
%For example, since $d^{2\alpha}\zeta_h \in C(\Omega)$, then $q_h$ can be chosen as the Lagrange interpolant of $d^{2\alpha} \zeta_h$ which belongs to $C(\Omega)$. In addition, since $d^{2\alpha} \zeta_h \in H^2(E)$ for $E \in \mesh_h$, the above approximation property holds (see Theorem 4.4.4 in \cite{brenner2007mathematical}).  
Using the Galerkin orthogonality of the finite element method, we write
\[
    X_1  = \sum_{E \in \mesh_h} \int_E \nabla (U -\uCG)\cdot \nabla(d^{2\alpha} \zeta_h - w_h)  - \sum_{E \in \mesh_h} \int_E  \nabla (U -\Pi_h U)\cdot \nabla(d^{2\alpha} \zeta_h ).
\]
 The  terms in the right-hand side are bounded using  similar arguments as  in \eqref{eq:h2_d2alpha_chi_0} - \eqref{eq:bound_T11}.
    \begin{align*}
    X_1 \leq \frac{1}{4} \|\nabla_h \zeta_h \|_{L^2_{\alpha}(\Omega)}^2 + C (h^{2-3s(1-\alpha)} + h^{2\alpha}).
    \end{align*}
 For $X_2$, similar arguments to the bound \eqref{eq:bd_T2} for the term $T_2$ hold:
\[
X_2 \leq C h^{1 -\frac32 s(1-\alpha)} \Vert \nabla_h \zeta_h\Vert_{L^2_\alpha(\Omega)}.
\]
We skip some details for brevity. The result is concluded by using triangle inequality.
\end{proof}

\section{Local $L^2$ and energy error estimates}\label{sec:local_estimates}

We show that the dG solution converges with an almost optimal rate in regions excluding the line $\Lambda$ for $k=1$ in subsection \ref{subsec:local_l2_k1}. For $k \geq 2$, we show that the dG solution converges with a rate of $k$ in subsection \ref{subsec:local_l2_k2}
%\Rd We assume here that $\Lambda$ is a $\mathcal{C}^2$ curve and the mesh satisfies $\vert \Lambda \cap E\vert \leq C h$ for
%all $E\in\mathcal{E}_h$. \Bk
In this section, we make the following assumption on the weak solution $u$ to \eqref{eq:weak_form_u}.

\textbf{A 1}. For any neighborhood $N$ of $\Lambda$, namely $\Lambda \subset N \subset \overline{N}\subset \Omega$, the weak solution $u$ belongs to $H^2(\Omega\setminus N)$.

This assumption is justified in the following two cases. If $f \in H^2(\Lambda)$, then $u \in H^2(\Omega \backslash N)$. This result was established using a splitting technique by Gjerde et al.  \cite{gjerde2019splitting}. Further, Ariche et al. show that if $f \in L^2(\Lambda)$ and $\Lambda$ is of class $ \mathcal{C}^4$, then $u$ belongs to a Kondratiev's type space \cite{ariche2017regularity}. This implies that $u \in H^2(\Omega \backslash N)$, see also \cite{d2012finite}. 

We first establish a local a priori bound on the solution of the intermediate problem 
\eqref{eq:intermediate_pb}.
\begin{lemma} \label{lemma:bounding_U_H2_local}
Let $N_0$ and $N_1$ be nested neighborhoods of $\Lambda$ satisfying  
\[
\Lambda  \subsetneq N_0 \subset \overline{N_0} \subset N_1 \subset \Omega.
\]
There exist $h_0>0$ and  a constant $C$ independent of $h$ such that for all $h\leq h_0$
\begin{equation}
    \|U\|_{H^2(\Omega \backslash N_1)} \leq C \left( \|f\|_{L^2(\Lambda)} + \|u\|_{H^2(\Omega \backslash N_0)}\right).\label{eq:strict_h2_bd_U}
\end{equation}
\end{lemma}
\begin{proof}
There exists a neighborhood  $N_{1/2}$ of $\Lambda$ such that	
\[
\overline{N_0} \subset N_{1/2} 
\subset \overline{N_{1/2}} \subset N_1 \subset \overline{N_1}\subset \Omega.
\]
Define a mollifier function $\phi \in C^{\infty}(\Omega)$ which is equal to  $1$ in $\Omega \backslash N_1$ and to $0$ in $N_{1/2}$.   
Recall that by definition of $U$ \eqref{eq:intermediate_pb} and $f_h$ \eqref{eq:def_fh}, there
exists $h_0>0$ such that for $h\leq h_0$, we have 
$$-\Delta U = 0, \quad  \mathrm{in} \quad  \Omega \backslash N_0. $$
In addition, set $g$ as follows. 
\begin{equation}
    g = \Delta (U \phi), \quad \mathrm{in} \quad  \Omega. 
\end{equation}
Clearly, $g \in L^2(\Omega)$ and 
\begin{align}
g  = \phi \Delta U + 2 \nabla U \cdot \nabla \phi + U \Delta \phi 
  = \begin{cases}
0, & \mathrm{in} \quad  N_{1/2}, \\ 
2\nabla U \cdot \nabla \phi + U \Delta \phi, & \mathrm{in } \quad  N_1 \setminus N_{1/2},  \\ 
 0,  & \mathrm{in } \quad \Omega \setminus N_1.
\end{cases}
\end{align} 
Hence, with Cauchy-Schwarz's inequality, we obtain 
\begin{align}
        \|g\|_{L^2(\Omega)}  \leq C \|U\|_{H^1(N_1 \backslash N_{1/2})}\left(\| \nabla \phi\|_{L^2(N_1 \backslash N_{1/2})}+\| \Delta \phi \|_{L^2(N_1 \backslash N_{1/2})}\right) \leq  C \| U \|_{H^1(N_1\backslash N_{1/2})}. 
    \end{align}
    In the above, the constant $C$ depends on the choice of the cut-off function $\phi$ but it is independent of $h$ for all $h\leq h_0$. We remark that $U\phi$ vanishes on the boundary $\partial\Omega$. 
    By convexity of the domain and the above bound, we have 
    \begin{equation}
        \|U \phi\|_{H^2(\Omega)} \leq C \|g\|_{L^2(\Omega)} \leq C \| U\|_{H^1(N_1 \backslash N_{1/2})}.  
    \end{equation}
    By the definition of $\phi$, the above bound,  and the triangle inequality (with $ \uCG \in W_h^{1}(\mesh_h)$ satisfying \eqref{eq:discrete_F.E} for $k=1$), we obtain 
    \begin{align}
   \| U\|_{H^2(\Omega \backslash N_1)} & = \Vert U\phi \Vert_{H^2(\Omega \backslash N_1)}  \leq \|U \phi \|_{H^2(\Omega)} \nonumber   
    \leq C \|U\|_{H^1(N_1\backslash N_{1/2})} \\ 
& \leq C(\| U- \uCG\|_{H^1(N_1 \backslash N_{1/2})} + \| u - \uCG \|_{H^1(N_1\backslash N_{1/2})} + \|u\|_{H^1(N_1 \backslash N_{1/2})}). \label{eq:triangle_ineq}
    \end{align}
    A standard finite element bound \eqref{eq:global_bd_interm_f.e}, the convexity of the domain and \eqref{eq:improved_bd_fh} yield 
    \begin{align}
    \|U - \uCG\|_{H^1(\Omega )} & \leq C h\|U\|_{H^2(\Omega)} \leq Ch\|f_h\|_{L^2(\Omega)} \leq C \| f\|_{L^2(\Lambda)}. \label{eq:bounding_first_term}
    \end{align}
    To bound the second term in \eqref{eq:triangle_ineq}, we use Theorem 9.1 in \cite{wahlbin1991local}.  
    \begin{equation}
        \| u- \uCG \|_{H^1(N_1 \backslash N_{1/2})} \leq \|u -\uCG\|_{H^{1}(\Omega \backslash N_{1/2})} \leq C (h \| u\|_{H^2(\Omega \backslash N_0)} + \|u -\uCG\|_{L^2(\Omega)}). \label{eq:estimate_local_u_H1}
    \end{equation}
    Using the global bound \eqref{eq:improved_L2estimate_F.E_h}, we obtain for $0<\theta<\frac12,$
    \begin{equation}
        \|u -\uCG\|_{H^{1}(N_1 \backslash N_{1/2})} \leq C (h \| u\|_{H^2(\Omega \backslash N_0)} 
+ h^{1-\theta}  \Vert f \Vert_{L^2(\Lambda)}  ). \label{eq:sec_bd_first_tr}
    \end{equation}
    Substituting \eqref{eq:bounding_first_term} and \eqref{eq:sec_bd_first_tr} in \eqref{eq:triangle_ineq} yields the result.
\end{proof}
\subsection{Local $L^2$ bound for $k=1$ }\label{subsec:local_l2_k1}
Let $N$ be a neighborhood of $\Lambda$ such that $\overline N \subset\Omega$ Further, we will make use of the following assumption. \\
\textbf{A.2.} There exist sets $N_0, N_1, N_2, N_3$ such that
\[
\Lambda \subsetneq  N_0 \subsetneq  N_1 \subset \overline N_1 \subsetneq N_2 \subset \overline N_2 \subsetneq N_3 \subsetneq N \subsetneq \Omega. 
\]
It is important to note that the choice of the above sets is fixed and does not depend on the mesh.

The main result of this section is the following local $L^2$ estimate. 
\begin{theorem}\label{thm:local_l2_k1}
Let $k =1$ and let Assumption~\textbf{A.2.} holds. There exist $h_0\geq 0$ and  a constant $C$ independent of $h$ such that for $0 < \theta < \frac{1}{2}$  
and all $h\leq h_0$
\begin{equation}
\|u - \uDG\|_{L^2(\Omega\backslash N)} \leq Ch^{ 2-\theta} + Ch^2 |\ln(h)|. \label{eq:localL2DGk1}
\end{equation}
\end{theorem}
The proof of this estimate also relies on establishing local bounds for the continuous and discontinuous discretizations of the intermediate problem \eqref{eq:intermediate_pb}. As before, this will be established in several Lemmas. 

\begin{lemma}
    \label{lemma:linking_U_w}
Assume \textbf{A.2}  holds. There exist $h_0>0$ and a constant $C$ independent of $h$ such that for all $h\leq h_0$
\begin{equation}
    \Vert U - \uDG \Vert_{L^2(\Omega \backslash N)} \leq Ch^{2-\theta}, \quad \forall 0 < \theta < \frac{1}{2}. 
\end{equation}
\end{lemma}
\begin{proof}
Define the characteristic function associated to $\Omega \setminus N$: 
$$ \chi_{\Omega \setminus N} (x) =  \begin{cases}
    1, & x \in \Omega\backslash N, \\ 
    0, & x \in N. 
\end{cases} $$
For readibility, set $\xi = U -\uDG$ and
consider the auxiliary problem:
\begin{alignat}{2} 
    - \Delta w & = \xi \chi_{\Omega \backslash N}, && \quad \text{in } \Omega, \label{eq:dual_pb_1}\\ 
    w  & = 0,&&  \quad  \text{on } \partial \Omega. 
\end{alignat}
Clearly, since $\xi \chi_{\Omega \backslash N}$  belongs to $L^2(\Omega)$, the function $w$ belongs to $H^2(\Omega)\cap H^1_0(\Omega)$. 
Multiplying \eqref{eq:dual_pb_1} by $\xi$ and integrating over $\Omega$, we obtain
\begin{align} 
\Vert \xi \Vert^2_{L^2(\Omega \backslash N)} 
%= & \sum_{E \in \mesh_h} \int_E \nabla \xi \cdot \nabla w - \sum_{E \in \mesh_h} \int_{\partial E} \xi \nabla w \cdot \bm{n}_E\\
= \sum_{E \in \mesh_h} \int_E \nabla \xi \cdot \nabla w  -\sum_{e \in \Gamma_h \cup \partial \Omega } \int_{e}  \{ \nabla w \}  \cdot \bm{n}_e[\xi] 
= a(\xi, w).  \label{eq:integ_parts_1}
\end{align} 
Let $S_h w  \in W_h^1(\mathcal{E}_h) $ be the Scott-Zhang interpolant of $w$. With the consistency property \eqref{eq:consistency_interm_pb}, we have
\begin{align}
\|\xi\|^2_{L^2(\Omega \backslash N)} = &  a(\xi, w- S_h w) \nonumber\\
=  & \sum_{E \in \mesh_h} \int_E \nabla \xi \cdot \nabla (w - S_h w)   - \sum_{e \in \Gamma_h \cup \partial \Omega} \int_{e} \left\{ \nabla(w - S_h w) \right\}\cdot \bm{n}_e  [\xi]\nonumber\\
= & 
\Theta_1 + \Theta_2. \label{eq:startpoint}
\end{align} 
We proceed by providing bounds for $\Theta_1$ and $\Theta_2$. %Let $D$ be such that $B \subset D \subset N$.
 We follow \cite{koppl2014optimal,choi2020optimal}, split $\Theta_1$ into two terms, and use Holder's inequality,
\begin{align}
\Theta_1 & = \sum_{E \in \mesh_h} \int_{E \cap  N_2  } \nabla \xi \cdot \nabla(w-S_h w) +\sum_{E \in \mesh_h} \int_{E \cap (\Omega \backslash N_2)} \nabla \xi\cdot \nabla(w-S_h w) \nonumber \\ 
& \leq  \| \nabla(w-S_hw)\|_{L^{\infty}(N_2)}\sum_{E \in \mesh_h} \|\nabla \xi \|_{L^1(E \cap N_2)}   +   \|\nabla_h \xi \|_{L^2(\Omega \backslash N_2)} \|\nabla(w - S_h w)\|_{L^2(\Omega \backslash N_2)}  \nonumber \\ &
=  \Theta_1^1 + \Theta_1^2. 
\end{align} 
Fix $\theta \in (0,1/2)$,  define $\alpha = 1 -\theta^2 $, which implies that $3/4 < \alpha < 1$.   Take $s = 2/(3\theta)$ in Lemma~ \ref{lemma:conv_dg_alpha_local}. We have 
\begin{equation}
    \|\xi \|_{\DG,\alpha} \leq Ch^{1-\theta}. \label{eq:bound_xiDG_alpha_in_localproof}  
\end{equation}
Hence, with Cauchy-Schwarz's inequality and the fact that $d^{-\alpha} \in L^2(\Omega)$, (recall $d$ is the distance function defined in
\eqref{eq:distdef}), we obtain 
\begin{equation}
\sum_{E\in \mesh_h} \|\nabla \xi\|_{L^1(E \cap N_2)} \leq \sum_{E\in \mesh_h} \|d^{-\alpha}\|_{L^2(E\cap N_2)}\|\nabla \xi\|_{L^2_{\alpha} (E \cap N_2)} \leq C \|\nabla_h \xi\|_{L^2_{\alpha}(\Omega)} \leq Ch^{1-\theta}.   \label{eq:bounding_L1norm_local_iterm}
\end{equation}
In addition, observe that since $ -\Delta w = 0 \,\, \mathrm{in} \,\, N_3$,  Theorem 8.10 in \cite{gilbarg2015elliptic} and elliptic regularity due to the convexity of the domain yield
\begin{equation}
    \Vert w \Vert_{W^{4,2}(N_3)} \leq  C\Vert w \Vert_{H^2(\Omega)} \leq C \|\xi\|_{L^{2}(\Omega\backslash N)}.  \label{eq:boundonw}
\end{equation}
Hence, by a Sobolev embedding result and approximation properties %\Rd this is true if $h$ is small enough, otherwise how
%do we know $\Delta_E$ belongs to $N_3$ and should we have then $\overline{N_2} \subset N_3$ ?\Bk, \Rd
there is $h_1>0$ such that for all $h\leq h_1$ 
\begin{equation}
\|\nabla(w-S_h w)\|_{L^{\infty}(N_2)} \leq C h |w|_{W^{2,\infty}(N_3)} \leq Ch \Vert w \Vert_{W^{4,2}(N_3)} \leq Ch\|\xi\|_{L^{2}(\Omega\backslash N)}. \label{eq:inf_bound_w}
\end{equation}
With \eqref{eq:bounding_L1norm_local_iterm} and \eqref{eq:inf_bound_w}, we obtain
\begin{equation}
|\Theta_1^1| \leq Ch^{2-\theta} \|\xi\|_{L^2(\Omega \backslash N)}. \label{eq:bound_theta11}
\end{equation} 
For $\Theta_1^2$, we apply Lemma 4.1 by Chen and Chen \cite{chen2004pointwise} (see \eqref{eq:Lemma_chen_chen} with $D = \Omega \backslash N_1 $ and 
$\widetilde{D} = \Omega \backslash N_2$).   There exists $h_2\geq 0$ such that for
all $h\leq h_2$
\[
\|\nabla_h \xi \|_{L^2(\Omega \backslash N_2)} \leq C h \Vert U\Vert_{H^2(\Omega\setminus N_1)}
+ C \Vert \xi \Vert_{L^2(\Omega\setminus N_1)}.
\]
%\Rd BR: I applied Lemma 4.1 with $\Omega_0 = \Omega\setminus N_2$ and $\Omega_1 = \Omega\setminus N_1$.  
%We need to assume $\overline N_1 \subset N_2$ for the distance to be positive.  (I give lots of details here, to be deleted later) \Bk
With Lemma \ref{lemma:bounding_U_H2_local}, \eqref{eq:global_bd_dg}, and \eqref{eq:boundhU},  we have
\begin{equation}
\|\nabla_h \xi \|_{L^2(\Omega \backslash N_2)} \leq C h (\Vert f \Vert_{L^2(\Lambda)} + \Vert u \Vert_{H^2(\Omega\setminus N_0)})
+C h^2 \Vert U \Vert_{H^2(\Omega)} \leq C  h (\Vert f \Vert_{L^2(\Lambda)} + \Vert u \Vert_{H^2(\Omega\setminus N_0)}). \nonumber
\end{equation}
% With \eqref{eq:global_bd_dg}, we have
% \[
% \|\nabla_h \xi \|_{L^2(\Omega \backslash N_2)} \leq C h (\Vert f \Vert_{L^2(\Lambda)} + \Vert u \Vert_{H^2(\Omega\setminus N_0)})
% + C h^2 \Vert U \Vert_{H^2(\Omega)}.
% \]
% With \eqref{eq:U_convex_bd} and \eqref{eq:improved_bd_fh} we have
% \[
% \|\nabla_h \xi \|_{L^2(\Omega \backslash N_2)} \leq C h (\Vert f \Vert_{L^2(\Lambda)} + \Vert u \Vert_{H^2(\Omega\setminus N_0)})
% + C h \Vert f \Vert_{L^2(\Lambda)}.
% \]
With approximation properties and an elliptic bound, we have
\[
\Vert \nabla (w-S_h w)\Vert_{L^2(\Omega\setminus N)} \leq C h \Vert w\Vert_{H^2(\Omega)}
\leq C h \Vert \xi \Vert_{L^2(\Omega\setminus N)}.
\]
So we combine the bounds above:
\begin{equation}
   |\Theta_1^2| \leq Ch^2 \|\xi\|_{L^2(\Omega \backslash N)}. \label{eq:bound_theta12}
\end{equation}
%C (h \|U\|_{H^2(\Omega \backslash N)} + \| \xi \|_{L^2(\Omega)} ) \|\nabla (w- S_h w)\|_{L^2(\Omega\backslash N)} \\ \leq Ch^2 \|w\|_{H^2(\Omega)} \leq Ch^2 \|\xi\|_{L^2(\Omega \backslash N)} \label{eq:bound_theta12}
Similarly, we split and bound $\Theta_2$.  For any domain $\mathcal{O}$, let $\Gamma_h(\mathcal{O})$ 
denote the set of all faces $e$ such that $e\cap \mathcal{O} \neq \emptyset$ and let $\Gamma_h^c(\mathcal{O})$ be the complementary
set of faces, namely $\Gamma_h^c(\mathcal{O}) = (\Gamma_h \cup \{e: e\subset\partial\Omega\})\setminus \Gamma_h(\mathcal{O})$. 
%Let $\Gamma_h(N_1)$ denote all the faces $e$ such that $e \cap N_1 \neq \emptyset$ and $\Gamma^c_h(N_1) = (\Gamma_h \cup \partial \Omega) \backslash \Gamma_h(N_1)$.  
 There exists $h_3>0$ such that for all $h\leq h_3$:
    \begin{align*}
    |\Theta_2| \leq & \|\nabla(w-S_hw )\|_{L^\infty(N_2)}\sum_{e\in \Gamma_h (N_1)}\|[\xi]\|_{L^1(e)}  \\ 
& + \sum_{e \in \Gamma_h^c(N_1)} \|\{\nabla(w-S_hw)\} \cdot \bm{n}_e\|_{L^2(e) } \|[\xi]\|_{L^2(e)}= \Theta_2^1 + \Theta_2^2. 
    \end{align*}
     Using \eqref{eq:inf_bound_w}, we have
\[
\Theta_2^1 \leq C h \Vert \xi\Vert_{L^2(\Omega\setminus N)} \sum_{e \in \Gamma_h(N_1)}\|[\xi]\|_{L^1(e)}.
\]
To handle the second factor in the left-hand side of the inequality above, we introduce a  tubular domain  $B_h$ containing $\Lambda$. 
That is, $B_h$ is the set of elements $E$ such that for any $\bm{x}\in E$, the distance $d(\bm{x},\Lambda) \leq 2 h$.
%\[
%B_h = \Bigcup\{ \bm{x}\in\Omega: \, d(\bm{x},\Lambda) \leq 2 h \}.
%\]%That is, $B_h$ contains all $\bm{x} \in \Omega$ with $d(\bm{x},\Lambda) \leq h$. 
This implies that the number of elements in $B_h$ is bounded above by $C h^{-1}$ for some constant
$C$ independent of $h$.  
%Thus, for any $e \in \Gamma_h(N_1) \backslash \Gamma_h(B_h)$,  we  use trace  estimate \eqref{eq:3d_1d_trace_estimate_continuous}. 
%\Rd Can you give details here? I do not see how we get $h^2$ for the first term. The number of elements in $B_h$ is
%bounded by $C h^{-1}$, right? \Bk 
\[
\sum_{e \in \Gamma_h(N_1\cap B_h)} \|[\xi]\|_{L^1(e)} \leq
C \left(\sum_{e\in \Gamma_h(B_h)} h \|1\|_{L^2(e)}^2 \right)^{1/2} \|\xi\|_{\DG}
\leq C h \|\xi\|_{\DG}.
\]
Any face $e\in \Gamma_h(N_1\setminus B_h)$ belongs to two elements, say $E_e^1$ and $E_e^2$. Since $d^{-\alpha-1}|_{E_e^i} \leq h^{-\alpha-1}$,
the function $d^{-\alpha}$ belongs to $H^1(E_e^i)$, for $i=1,2$. \Bk
%For $E \subset \Omega \backslash B_h$, $h \leq d(\bm{x},\Lambda), \,\, \forall \bm{x} \in E$ and $d^{-\alpha} \in H^1(E)$. \Rd Why? \Bk 
With the trace inequality \eqref{eq:3d_1d_trace_estimate_continuous} and with \eqref{eq:boundond}
\begin{align*}
\sum_{e \in \Gamma_h(N_1\setminus B_h)} \|[\xi]\|_{L^1(e)} \leq &
C \left(\sum_{e\in \Gamma_h(N_1\setminus B_h)} h \Vert d^{-\alpha} \Vert_{L^2(e)}^2\right)^{1/2}
\Vert \xi \Vert_{\DG,\alpha}\\
\leq &  C \left(\sum_{e\in \Gamma_h(N_1\setminus B_h)} ( \Vert d^{-\alpha}\Vert_{L^2(E_e^1\cup E_e^2)}^2 
+ h^2 \Vert d^{-\alpha-1}\Vert_{L^2(E_e^1\cup E_e^2)}^2)\right)^{1/2} \Vert \xi \Vert_{\DG,\alpha}\\
\leq &  C \left(\sum_{e\in \Gamma_h(N_1\setminus B_h)} \Vert d^{-\alpha}\Vert_{L^2(E_e^1\cup E_e^2)}^2 \right)^{1/2}
\Vert \xi \Vert_{\DG,\alpha}\\
\leq &  C \Vert d^{-\alpha}\Vert_{L^2(\Omega)} \Vert \xi \Vert_{\DG,\alpha}.  
\end{align*}
%\begin{align*}
%\sum_{e \in \Gamma_h(N_1)}& \|[\xi]\|_{L^1(e)} \leq C \left(\sum_{e\in \Gamma_h(B_h)} h \|1\|_{L^2(e)}^2 \right)^{1/2} \|\xi\|_{\DG}\\ &+  C \left(\sum_{e\in \Gamma_h (N_1 \backslash B_h)} \|d^{-\alpha}\|^2_{L^2(E_e)} + h^2\|d^{-\alpha - 1}\|^2_{L^2(E_e)} \right)^{1/2} \|\xi\|_{\DG,\alpha}. 
%\end{align*}
%In the above, $E_e$ denotes the union of the two elements sharing the face $e$. Since $h \leq d(\bm{x},\Lambda)$ for $x \in E \subset \Omega \backslash B_h$, we have 
%\[h \|d^{-\alpha - 1}\|_{L^2(E_e)} \leq \|d^{-\alpha}\|_{L^2(E_e)}.\]
Hence,  we use \eqref{eq:bound_xiDG_alpha_in_localproof}, \eqref{eq:inf_bound_w} and the fact $\|\xi\|_{\DG} \leq Ch\|U\|_{H^2(\Omega)} \leq C.$ We have 
\begin{equation}
|\Theta_2^1| \leq Ch^{2-\theta} \|\xi\|_{L^2(\Omega \setminus N)}.  \label{eq:bound_Theta_21}
\end{equation}
To handle $\Theta_2^2$, we use \eqref{eq:3d_1d_trace_estimate_continuous}, approximation properties,  Lemma 4.1 in \cite{chen2004pointwise} (see \eqref{eq:Lemma_chen_chen} with $D = \Omega\setminus N_2$ and $\widetilde{D} = \Omega\setminus N)$,  and 
\eqref{eq:strict_h2_bd_U}.
\begin{align}
|\Theta_2^2|  &\leq C \left(\sum_{e \in \Gamma_h^c(N_1)} \|\nabla (w- S_h w)\|_{L^2(  E_e^1 \cup E_e^2  )}^2 + h^2 \|\nabla^2 w\|_{L^2(E_e^1\cup E_e^2)}^2 \right)^{1/2} \|\xi\|_{\DG(\Omega \backslash N_2)}\nonumber  \\ 
&\leq Ch|w|_{H^2(\Omega)} (h|U|_{H^2(\Omega\backslash N)} + \|\xi\|_{L^2(\Omega\setminus N)}) %\\
%&\Rd \leq C h \vert w \vert_{H^2(\Omega)} (h+ \|\xi\|_{L^2(\Omega)}). \Bk %\|\xi\|_{L^2(\Omega \backslash N)}. \label{eq:bound_Theta_22}  
\end{align}
With \eqref{eq:global_bd_dg}and \eqref{eq:boundhU},  we have
\[
\Vert \xi \Vert_{L^2(\Omega\setminus N)} \leq C h \Vert f \Vert_{L^2(\Lambda)}.
\]
Thus, with \eqref{eq:boundonw}, we obtain
\begin{equation}
|\Theta_2^2|  \leq C h^2 \Vert \xi \Vert_{L^2(\Omega\setminus N)}.
\label{eq:bound_Theta_22}
\end{equation}
 \Bk
Combining bounds \eqref{eq:bound_theta11}, \eqref{eq:bound_theta12}, \eqref{eq:bound_Theta_21}, \eqref{eq:bound_Theta_22}
with \eqref{eq:startpoint} 
yields the result. 
\end{proof}
\Bk
The next step is to bound the local $L^2$ norm of the error $ U - \uCG$. 
\begin{lemma}\label{lemma:local_l2_bound_cg} Let Assumption~\textbf{A.2.} hold. 
There exist $h_0>0$ and  a constant $C$ independent of $h$ such that  for all $h\leq h_0$
    \begin{equation}
        \|U - \uCG\|_{L^2(\Omega \backslash N)} \leq Ch^{2-\theta}, \quad  \forall 0 < \theta < \frac{1}{2}. \label{eq:3d_1d_local_l2_bound_cg}  
    \end{equation}
\end{lemma}
\begin{proof}
Because the proof follows that of Lemma~\ref{lemma:linking_U_w}, it is sketched only and
details are omitted.
%. Thus, we give a brief proof sketch and omit the details for brevity. 
The starting point is the following dual problem
\begin{alignat}{2} 
    - \Delta z & = (U-\uCG) \chi_{\Omega \backslash N}, && \quad \text{in } \Omega, \label{eq:dual_pb_2}\\ 
    z  & = 0,&&  \quad  \text{on } \partial \Omega, 
\end{alignat}  
where $\chi_{\Omega\setminus N}$ is the characteristic function associated to $\Omega\setminus N$.
Let $S_h z$ denote the Scott-Zhang interpolant of $z$.  We multiply \eqref{eq:dual_pb_2} by $(U-\uCG)$ and integrate by parts. 
\begin{align}
\|U-&\uCG\|^2_{L^2(\Omega \backslash N)} = \int \nabla z \cdot \nabla (U-\uCG) = \int \nabla (z-S_hz) \cdot \nabla(U-\uCG) \nonumber\\ & \leq C \|\nabla(z- S_hz)\|_{L^\infty(N_1)} \|\nabla (U-\uCG)\|_{L^1(N_1)} + \|\nabla(z- S_hz)\|_{L^2(\Omega \backslash N_1)} \|\nabla (U-\uCG)\|_{L^2(\Omega \backslash N_1)}.
\label{eq:uCGint}
\end{align}
The first term is handled like $\Theta_1^1$. Let $\alpha = 1-\theta^2$ and use Lemma \ref{lemma:weighted_bd_fe_U} with $s=2/(3\theta)$ to obtain
for $h$ small enough: 
\begin{align}
    \|\nabla(z- S_hz)\|_{L^\infty(N_1)} \|\nabla (U-\uCG)\|_{L^1(N_1)} \leq & C h |z|_{W^{2,\infty}(N_2)} \|\nabla(U-\uCG)\|_{L^2_\alpha(\Omega)} \nonumber\\
  \leq & Ch^{2-\theta}\|U- \uCG\|_{L^2(\Omega \backslash N)}. \label{eq:3d_1d_local_cg_U_0}
\end{align}
For the second term, we use Theorem 9.1 in \cite{wahlbin1991local}, \eqref{eq:strict_h2_bd_U}, \eqref{eq:global_bd_interm_f.e},
\eqref{eq:U_convex_bd} and \eqref{eq:improved_bd_fh}.
\begin{equation}
    \|\nabla( U- \uCG )\|_{L^2(\Omega \backslash N_1)} \leq C (h \| U\|_{H^2(\Omega \backslash N_0)} + \|U - \uCG\|_{L^2(\Omega)}) \leq Ch. \nonumber
\end{equation} 
Therefore, with approximation properties and  convexity of the domain, we have
\begin{equation} 
    \|z - S_h z \|_{L^2(\Omega \backslash N_1)}\|\nabla(U - \uCG)\|_{L^2(\Omega\setminus N_1)} 
\leq C h^2 \|z\|_{H^2(\Omega)}\leq C h^2 \|U - \uCG\|_{L^2(\Omega \backslash N)}. \label{eq:3d_1d_local_cg_U_1}
\end{equation}
Bound \eqref{eq:3d_1d_local_l2_bound_cg} immediately follows from \eqref{eq:uCGint}, \eqref{eq:3d_1d_local_cg_U_0} and  \eqref{eq:3d_1d_local_cg_U_1}.
\end{proof}
\textbf{Proof of Theorem} \ref{thm:local_l2_k1}: The result follows by the triangle inequality:
\begin{equation}
\|u - \uDG\|_{L^2(\Omega \backslash N)} \leq \|u -\uCG\|_{L^2(\Omega \backslash N)} + \|\uCG - U\|_{L^2(\Omega \backslash N)} + \|U - \uDG\|_{L^2(\Omega \backslash N)}. 
\end{equation}
The first term is bounded in \cite{koppl2016local}: 
\[\|u - \uCG\|_{L^2(\Omega \backslash N)} \leq Ch^2 |\ln h|. \] 
The result then follows by using Lemma \ref{lemma:linking_U_w} and Lemma \ref{lemma:local_l2_bound_cg}. 

\subsection{Local $L^2$ bounds for $k \geq 2$}\label{subsec:local_l2_k2}
In this section, we use duality arguments to obtain a local $L^2$ estimate for $k \geq 2$. We use negative norms, recalled here. For any integer $m \geq 0$  and for $v \in L^2(\Omega)$,  
\begin{equation}
\|v\|_{H^{-m}(B)} = \sup_{\phi \in H_0^m(B)}\frac{|\int_{B} v\phi | }{\|\phi\|_{H^{m}(B)}}, \quad B \subseteq \Omega. 
\end{equation}
%\Rd We assume that $\epsilon = -1$, namely we consider the SIPG scheme. \Bk 
The main result of this section is given in Theorem \ref{thm:main_result_local_k2}. To begin this analysis, we first establish general local results for the dG approximation. Such results are shown with techniques adapted from Nitsche and Schatz \cite{nitsche1974interior}. In addition, for any convex domain $B\subseteq\Omega$,  we introduce the operator $Q_B:L^2(B) \rightarrow  H^2(\Omega)\cap H_0^1(\Omega) $ with $Q_B(\phi) = v$ such that
$v$ solves 
\begin{alignat}{2}
    -\Delta v &= \phi  && \quad\mathrm{in} \,\, B \label{eq:second_duality} \\ 
    v & = 0 , && \quad \mathrm{on}\,\,  \partial B.  
\end{alignat}
The following elliptic regularity result holds \cite{evans2010partial}. For any integer $m \geq 0$, 
\begin{equation}
    \label{eq:elliptic_regularity_general} 
    \| Q_B(\phi)\|_{H^{m+2}(B)} \leq  C\|\phi\|_{H^m(B)}. 
\end{equation}
\begin{lemma}
     Let $B  \subset \overline B \subset B_1 \subset   \overline{B_1}\subset  \Omega$ be open convex sets.  There exists $h_0>0$ 
such that for  any integer  \Bk $m \geq 0$  and all $0<h \leq h_0$
    \begin{align}
    \| U-\uDG \|_{H^{-m}(B)} & \leq C (h^{\min(k, m+1)}\vertiii{ U-\uDG }_{\DG(B_1)} + \|U-\uDG\|_{H^{-m-1}(B_1)}). \label{eq:error_H_s}
    \end{align} 
    In addition, we have %for any \Rd integer \Bk $m \geq 0$, 
    \begin{align}
    \| U-\uDG  \|_{L^2(B)} & \leq  C (h\vertiii{U-\uDG}_{\DG(B_1)} + \| U-\uDG \|_{H^{-m}(B_1)}). \label{eq:error_interm_l2}
    \end{align}
The constant $C$ is independent of $h$. 
\end{lemma} 
\begin{proof}
 Fix an integer  $m\geq 0$ and denote $\xi = U - \uDG$. 
    Let $\omega \in \mathcal{C}_0^{\infty}(\Omega)$ with $\omega = 1$ in $B$ and $\omega = 0$ in $\Omega \backslash B_0$ where $\bar{B} \subset B_0 \subset \bar{B}_0 \subset B_1$. Note that $\mathrm{supp}(\omega)\subset B_0$. We have  
\begin{align}
\|\xi\|_{H^{-m} (B)} = \|\omega \xi\|_{H^{-m} (B)} \leq \| \omega \xi\|_{H^{-m} (\Omega)} = \sup_{\phi \in H_0^m(\Omega)} \frac{ |\int_{\Omega} \omega \xi\phi| }{\|\phi\|_{H^m(\Omega)}}. \label{eq:l2norm_duality} 
\end{align} 
 Fix $\phi \in H^m_0(\Omega)$  and  define $v = Q_{\Omega} (\phi)$. 
We multiply \eqref{eq:second_duality} with $\omega \xi$ and integrate by parts. Since $v \in H^2(\Omega)$, we have 
\begin{align}
\int_{\Omega} \omega \xi \phi = \sum_{E \in \mesh_h} \int_E \nabla v \cdot \nabla(\omega \xi) - \sum_{e\in \Gamma_h \cup \partial \Omega}  \int_e \{\nabla v\}\cdot \bm{n}_e \omega [\xi]=  a (\omega \xi, v).  \label{eq:first_identity_a} 
\end{align}
In view of \eqref{eq:first_identity_a} and \eqref{eq:elliptic_regularity_general}, \eqref{eq:l2norm_duality} yields
\begin{align}
    \|\xi\|_{H^{-m}(B)} \leq   C  \sup_{v \in H^{m+2}(\Omega)} \frac{| a (\omega \xi,v)|}{\|v\|_{H^{m+2}(\Omega)}}. \label{eq:switching_to_v}
\end{align} 
Observe that 
\begin{align}
 a  (\omega \xi, v) = \sum_{E \in \mesh_h} \int_E \xi \nabla \omega \cdot \nabla v + \sum_{E \in \mesh_h} \int_E \nabla \xi \cdot (\nabla (\omega v) - v \nabla \omega )\nonumber  - \sum_{e\in \Gamma_h \cup \partial \Omega}  \int_e( \{\nabla (\omega v) - v\nabla \omega \}\cdot \bm{n}_e  [\xi].
\end{align}
In addition, with integration by parts and the fact that $v \nabla \omega$ is continuous, we have 
\begin{align}
- \sum_{E \in \mesh_h} \int_E \nabla \xi \cdot (v \nabla \omega) = \sum_{E \in \mesh_h} \int_E \xi \nabla \cdot (v \nabla \omega ) - \sum_{e\in \Gamma_h \cup \partial \Omega} \int_e \{v \nabla \omega\}\cdot \bm{n}_e [\xi].  
\end{align}
Hence, we obtain 
\begin{equation}
a(\omega \xi, v) = a (\xi , \omega v) 
+ \mathcal{I}(\xi,\omega v),\label{eq:moving_omega}
\end{equation}
with 
\[
\mathcal{I}(\xi,\omega v) = \sum_{E\in \mesh_h} \int_E \xi \left(\nabla \omega \cdot \nabla v 
+ \nabla \cdot (v \nabla \omega)\right).
\]
For $E \in \mesh_h$ with $E \cap B_1 \neq \emptyset $, let $y_{h,E} \in \mathbbm{P}^k(E)$ be the Lagrange  interpolant of $\omega v $ satisfying 
%\Rd here we need it for $d=1$ or $d=2$ integer, we need the semi-norms only \Bk
\begin{equation}
\|\omega v - y_{h,E} \|_{H^d(E)} \leq C h^{\min(k+1,m+2)-d}  \|\omega v\|_{H^{m+2}(E)}, \quad 0\leq d \leq 2.
 \label{eq:approximation_omegav} \end{equation}
Then, define $\chi_h \in V_h^{k}(\mesh_h)$ as $\chi_h \vert_E = y_{h,E}$ if $\omega v\vert_E \neq 0$ a.e in $E$. Otherwise, $\chi_h|_E= 0$. By construction, for $h$ small enough, all the terms involving elements and edges that do not intersect $B_1$ vanish. Using \eqref{eq:consistency_interm_pb} and continuity properties, we have 
 \begin{equation}
     a (\xi, \omega v ) =  a (\xi, \omega v - \chi_h) \leq C \vertiii{\xi}_{\DG(B_1)} \vertiii{\omega v - \chi_h}_{\DG(B_1)} \label{eq:bound_axiv}
 \end{equation}
 From trace estimates and \eqref{eq:approximation_omegav}, we have
 \begin{equation}
    \vertiii{\omega v - \chi_h}_{\DG(B_1)} \leq C h^{\min(k, m+1) } \|\omega v\|_{H^{m+2}(B_1)} \leq Ch^{\min(k, m+1) } \|v\|_{H^{m+2}(B_1)}  . 
 \end{equation} 
Therefore, \eqref{eq:bound_axiv} becomes
\begin{equation}
 a (\xi, \omega v)  \leq C h^{\min(k, m+1)} \vertiii{\xi}_{\DG(B_1)} \|v\|_{H^{m+2}(B_1)}.   \label{eq:bound_aeomega}
\end{equation}
For the second term in \eqref{eq:moving_omega}, since $\omega \in C^{\infty}(\Omega)$ with $\mathrm{supp}(\omega) \subset B_0$, 
\begin{equation} \mathcal{I}(\xi, \omega v) \leq C \|\xi\|_{H^{-m-1}(B_1)}\|v\|_{H^{m+2}(B_1)}. \label{eq:bound_I} \end{equation}
With \eqref{eq:bound_aeomega} and \eqref{eq:bound_I}, \eqref{eq:switching_to_v} yields \eqref{eq:error_H_s}.  To show \eqref{eq:error_interm_l2}, define a finite sequence of nested convex sets $D_0 = B \subset D_1 \subset \dots \subset D_{m-1}= B_1$ such that
 $\bar{D}_i \subset D_{i+1}$.  Applying \eqref{eq:error_H_s} with $s = 0$ for the sets $D_0 \subset D_1$ yields: 
\begin{equation}
\|\xi\|_{L^2(B)} \leq C h \vertiii{\xi}_{\DG(D_1)} + \|\xi\|_{H^{-1}(D_1)}. 
\end{equation}
Iteratively applying bound \eqref{eq:error_H_s} to the last term in the above inequality yields \eqref{eq:error_interm_l2}.
% \begin{equation}
% \|\xi\|_{L^2(B_1)} \leq C h\|\xi\|_{\DG(B_1)} + \|\xi\|_{H^{-1}(B_1)} \leq Ch^2 + \|\xi\|_{H^{-1}(N_1)}.  \label{eq:error_semi_final}
% \end{equation}
% It remains to estimate the last term in the above inequality. We have 
% \begin{equation}
%     \|\xi\|_{H^{-k}(N_1)} \leq \|\xi\|_{H^{-1}(\Omega)} = \sup\left\{\frac{\int_{\Omega} \xi\phi }{\|\phi\|_{H^{1}(\Omega)}}, \,\, \phi \in H^1(\Omega) \right\}. \label{eq:negative_1_norm}
%     \end{equation}
% We again let $\theta = Q_\Omega (\phi)$. 
% By elliptic regularity, it follows that 
% \[\|\theta\|_{H^3(\Omega)} \leq C \|\phi\|_{H^{1}(\Omega)}.  \] 
% Hence by approximation properties, there exists $\tilde{\theta} \in V_h^k(\mesh_h)$ such that 
% $$ \|\theta - \tilde{\theta} \|_{\DG} \leq Ch^2 |\theta|_{H^3(\Omega)} \leq C h^2\|\phi\|_{H^{1}(\Omega)}. $$
% In addition, a standard error bound, convexity of the domain, and \eqref{eq:bd_fh} yield: 
% \[\|\xi\|_{\DG} \leq C h \|U\|_{H^2(\Omega)} \leq C. \]
% Hence, by multiplying \eqref{eq:second_duality} by $\xi$ and using \eqref{eq:consistency_interm_pb}, we have 
% \[\int_{\Omega} \xi\phi = a(\theta, \xi) = a(\xi,\theta-\tilde{\theta}_h) \leq C \|\xi\|_{\DG}\|\theta - \tilde{\theta}\|_{\DG} \leq C h^2 \|\phi\|_{H^1(\Omega)}. \]
% With the above bound, \eqref{eq:negative_1_norm} yields 
% $$\|\xi\|_{H^{-1}(N)} \leq Ch^2. $$
% Substituting the above bound in \eqref{eq:error_semi_final} results in $$ \|\xi\|_{L^2(N)} \leq Ch^2. $$ Combining steps $1-3$ and \eqref{eq:triangle_local} concludes the proof. 
\end{proof}
\begin{theorem} \label{thm:main_result_local_k2}
Fix a convex set $B  \subset \overline B \subset \Omega$ with $\Lambda \subset \Omega \setminus \overline B$.  Fix 
$ 0 < \theta < \frac{1}{2}$ and $k \geq 2$. There exist $h_0>0$ and a constant $C$ independent of $h$,
\begin{equation}
\|u - \uDG\|_{L^2(B)} \leq C h^{k-\theta}.  \label{eq:local_estimate_result}
\end{equation}
\end{theorem}
\begin{remark}
We remark that this result is not optimal. However, it is an improvement to the order of convergence provided in Theorem \ref{theorem:global_estimate}. In addition, it allows us to show almost optimal estimates for the local energy norm, see Section \ref{subsec:local_energy_estimate},
\end{remark}
\begin{proof} 
First, we apply the triangle inequality to obtain 
\begin{equation}
\|u-\uDG\|_{L^2(B)} \leq \|u-\uCG\|_{L^2(B)} + \|\uCG - U\|_{L^2(B)} + \|U - \uDG\|_{L^2(B)}. \label{eq:triangle_local}
\end{equation} 
%Recall that $\uCG \in W_h^k(\mesh_h)$ solves \eqref{eq:discrete_F.E}. 
The remainder of the proof will consist of bounding each of the above terms. We divide this task into several steps. 
We select  convex sets $B_0, B_1, \dots, B_k$ with $ \bar{B} \subset B_0$,  $\bar{B_i} \subset B_{i+1}$ for
$i=0,\dots,k-1$, $\overline B_k \subset \Omega$ and $\Lambda \subset \Omega \setminus \overline B_k$.

\textit{Step 1:} Bounding $\|u-\uCG\|_{L^2(B)}$:  Since $W_h^k(\mesh_h) \subset W_0^{1,q}(\Omega)$, we have the following Galerkin orthogonality property. 
      \begin{equation}\int_{\Omega} \nabla (u-\uCG) \cdot \nabla v_h = 0, \quad \forall v_h \in W_h^{k}(\mesh_h).\label{eq:consistency_fem_original}
    \end{equation}
    Thus, we apply Theorem 5.1 in  \cite{nitsche1974interior}. There exists $h_1\geq 0$ such that for all $h\leq h_1$, we have 
    \begin{equation} 
        \|u-\uCG\|_{L^2(B)} \leq C \left( h^k \|u\|_{H^k(B_0)} + \|u-\uCG\|_{H^{-k}(\Omega)}\right). \label{eq:bounding_first_term_local_neg1}
    \end{equation} 
To estimate the second term, %note that 
%\begin{align}
    %\|u-\uCG\|_{H^{-k}(\Omega)} = \sup\left( \frac{|\int_{\Omega} (u-\uCG)\phi|  }{\|\phi\|_{H^{k}(\Omega)}}, \,\, \phi \in H^{k}_0(\Omega) \right). \label{eq:3d_1d_hk1_og}
%\end{align}
fix $\phi \in  H^{k}_0(\Omega)$. %, we consider $v = Q_\Omega(\phi)$. 
Observe that with a Sobolev embedding result and \eqref{eq:elliptic_regularity_general}, we have 
\begin{align*}
    \|Q_\Omega(\phi) \|_{W^{k+1,4}(\Omega)} \leq C \|Q_\Omega(\phi) \|_{H^{k+2} (\Omega)} \leq C \|\phi\|_{H^k(\Omega)} .
\end{align*}
We denote by $v_h$ the Scott-Zhang  interpolant of $Q_\Omega(\phi)$; we have% satisfying 
\[ \|\nabla( Q_\Omega(\phi)  -v_h)\|_{L^4(\Omega)}  \leq  Ch^k \|Q_\Omega(\phi)\|_{W^{k+1,4}(\Omega)} \leq Ch^k \|\phi\|_{H^{k} (\Omega)}. \] 
We multiply \eqref{eq:second_duality} by $u-\uCG$ and integrate by parts. By \eqref{eq:consistency_fem_original}, we have 
\begin{multline}
\int_{\Omega}(u-\uCG) \phi = \int_{\Omega} \nabla ( Q_\Omega(\phi) - v_h) \cdot \nabla (u-\uCG) \leq  \| \nabla (Q_\Omega(\phi)- v_h) \|_{L^4(\Omega)}    \|\nabla (u-\uCG) \|_{L^{4/3}(\Omega)} \\ \leq  Ch^k\|\phi\|_{H^{k}(\Omega)} \|\nabla (u-\uCG) \|_{L^{4/3}(\Omega)}.\label{eq:duality_og_local}
\end{multline}
Let $S_h u$ be the Scott--Zhang interpolant of $u$. With the stability of the interpolant,\eqref{eq:3d_1d_inverse_esimate_nabla}, and \eqref{eq:improved_L2estimate_F.E_h},  we have
\begin{align}
 \|\nabla (u-\uCG) \|_{L^{4/3}(\Omega)} \leq  \|\nabla (u-S_h u ) \|_{L^{4/3}(\Omega)} + \|\nabla (S_h u -\uCG) \|_{L^{4/3}(\Omega)} \nonumber \\ \leq C|u|_{W^{1,4/3}(\Omega)} + h^{-1} \|S_h u -\uCG \|_{L^{4/3}(\Omega)} 
\leq  C|u|_{W^{1,4/3}(\Omega)} + C(\theta) h^{-\theta}  \Vert f \Vert_{L^2(\Lambda)}.   \label{eq:estimating_nabla_u}
\end{align}
With \eqref{eq:estimating_nabla_u} and \eqref{eq:duality_og_local},  we have
\begin{equation}
\|u -\uCG \|_{H^{-k}(\Omega)} \leq Ch^{k-\theta}.  \label{eq:bounding_first_term_local_0}
\end{equation}
From \eqref{eq:bounding_first_term_local_0} and \eqref{eq:bounding_first_term_local_neg1}, we have 
\begin{equation} 
    \|u-\uCG\|_{L^2(B)} \leq C h^{k-\theta}. \label{eq:bounding_first_term_local} 
\end{equation} 
\textit{Step 2:} Bounding $\|U-\uCG\|_{L^2(B)}$:
 Let $N$ be a neighborhood of $\Lambda$ such that   $\overline B_k \subset \Omega \backslash N$. 
 There exists $h_2>0$ such that for all $h\leq h_2$ , 
$-\Delta U = 0$ in $\Omega \backslash N$.  %\Rd Why do we need this? where do we use it? \Bk
% for $h$ small and for some neighborhood $N$ such that $\bar{B} \subset \Omega \backslash N$, 
Theorem 8.10 in \cite{gilbarg2015elliptic} and Lemma \ref{lemma:bounding_U_H2_local} yield: 
\begin{equation}
\|U\|_{H^{k+1}( B_k )} \leq C \|U\|_{H^1(\Omega \backslash N)} \leq C. \label{eq:interior_reg_U_B}
\end{equation}
An application of Theorem 5.1 in \cite{nitsche1974interior} yields, for $h$ small enough, say $h\leq h_2$, for some $h_2\geq 0$:
\begin{equation}
\|U-\uCG\|_{L^2(B)} \leq C h^k\|U\|_{H^k(B_0)} + C \|U-\uCG\|_{H^{-k}(\Omega)}. \label{eq:l2_interm_cg_U}
\end{equation}
We perform a similar duality argument as above. For any $\phi \in  H_0^k(\Omega)$, we denote $z = Q_{\Omega} \phi$
and $S_h z$ the Scott-Zhang interpolant of $z$,  
\begin{multline}
\int_{\Omega} (U-\uCG) \phi = \int_{\Omega} \nabla(z - S_h z) \cdot \nabla(U - \uCG) \leq Ch^k \|z\|_{H^{k+1}(\Omega)} \|\nabla(U-\uCG)\| \\ \leq Ch^{k} \|\phi\|_{H^k(\Omega)} \|\nabla(U-\uCG)\|.  \label{eq:localkint}
\end{multline}
The last inequality holds by \eqref{eq:elliptic_regularity_general}. Noting that \eqref{eq:bounding_first_term} holds for the finite element solution $\uCG$ of any degree $k$, we have from \eqref{eq:localkint} 
\begin{equation}
 \|U -\uCG\|_{H^{-k}(\Omega)} \leq Ch^{k} \nonumber.  
\end{equation}
The above bound with \eqref{eq:l2_interm_cg_U} implies that \begin{equation} 
    \|U - \uCG\|_{L^2(B)} \leq Ch^{k}.
\label{eq:bounding_second_term_local}
\end{equation}
 \textit{Step 3}: Bounding $\|U - \uDG\|_{L^2(B)}$:
 %We make use of Lemma 4.1 in \cite{chen2004pointwise}, that states that for any sets $D \subset \overline D \subset \widetilde{D}$, we have 
 %\begin{equation}
%\vertiii{U-\uDG}_{\DG(D)} \leq C(h^{k} \|U\|_{H^{k+1}(\widetilde{D})} +   \|U-\uDG\|_{L^2(\widetilde{D})}). \label{eq:Lemma_chen_chen}
 %\end{equation}
We denote $\xi = U-\uDG$ and we iteratively use \eqref{eq:Lemma_chen_chen} and \eqref{eq:error_interm_l2} for the nested sets $B \subset B_0 \subset \ldots \subset B_{k}$. 
We obtain 
\begin{equation}
\|\xi\|_{L^2(B)} \leq C (h^{k+1}  \Vert U \Vert_{H^{k+1}(B_k)}  + h^k \|\xi\|_{L^2( \Omega )}) +  C  \|\xi\|_{H^{-k}(\Omega)}. 
\end{equation}
To estimate $\|\xi\|_{H^{-k}(\Omega)}$, we also use a duality argument. Let $\phi \in  H_0^{k}(\Omega) $ be given and
let $v = Q_\Omega \phi$. We multiply \eqref{eq:second_duality} by $v$, integrate by parts, use \eqref{eq:consistency_interm_pb}, the symmetry of $a(\cdot,\cdot)$, and \eqref{eq:xiDGfull}. 
\begin{multline}
\int_\Omega \phi \xi = a(v, \xi) = a(v - S_h v , \xi ) \leq C \vertiii{v - S_h v}_{\DG}\vertiii{\xi}_{\DG}  \leq Ch^{k} \|v\|_{H^{k+1}(\Omega)} \leq C h^k \|\phi\|_{H^{k}(\Omega) } . 
\end{multline} 
This implies that 
$$\|\xi\|_{H^{-k}(\Omega)} \leq Ch^k. $$ 
With the global estimate \eqref{eq:global_bd_dg}, the bound \eqref{eq:interior_reg_U_B}, %eq:improved_global_dg_bd} \Rd but this is not for $\xi$?\Bk 
and the above bound, we finally have that 
\begin{equation}
\|\xi\|_{L^2(B)} \leq Ch^k. \label{eq:bounding_third_term_local}
\end{equation}
This concludes the proof.  
\end{proof}
\subsection{Local energy estimate} \label{subsec:local_energy_estimate}
With the local $L^2$ results of the previous sections, we show a local energy estimate.  The second bound \eqref{eq:local_dg_k1}
is a stronger result in the sense that it is valid up to the boundary of $\Omega$ whereas \eqref{eq:local_dg_k2} is valid
for a domain that does not intersect with the boundary. 
\begin{theorem}\label{theorem:local_energy_bd}
Let Assumptions~\textbf{A.1.} and~\textbf{A.2.} hold.   Fix a convex set $B  \subset \overline B \subset \Omega$ with $\Lambda \subset \Omega \backslash \overline B$.  Fix 
$ \theta \in (0, \frac{1}{2})$ and $k \geq 1$. There exist $h_0>0$ and a constant $C$ independent of $h$
such that for all $h\leq h_0$
    \begin{equation}
        \|u - \uDG\|_{\DG(B)} \leq Ch^{k-\theta}. \label{eq:local_dg_k2}
    \end{equation}
    In addition, for $k = 1$ and  for any neighborhood  $N \subset \Omega$ such that $\Lambda \subset N$, 
    \begin{equation}
        \|u - \uDG\|_{\DG(\Omega \backslash N)} \leq Ch^{1-\theta}. \label{eq:local_dg_k1}
    \end{equation}
    \end{theorem}
    \begin{proof}
%    We show \eqref{eq:local_dg_k2}. 
By the triangle inequality, we have 
    \begin{equation}
        \|u - \uDG\|_{\DG(B)} \leq \| u -\uCG\|_{\DG(B)} + \| \uCG - U\|_{\DG(B)} + \| U - \uDG\|_{\DG(B)}. \label{eq:main_tr_theorem2}
    \end{equation}
    We proceed by providing bounds on each of the terms above. Let  $B_0$ be a convex  set such that $B\subset  \overline B  \subset B_0$ and $\Lambda \subset \Omega \setminus \overline B_0$.  Theorem 9.1 in \cite{wahlbin1991local} applied to problems \eqref{eq:model_pb} and \eqref{eq:intermediate_pb} results in the following two bounds. There exists $h_0>0$ such that for all $h\leq h_0$, 
    \begin{align}
         \Vert u-\uCG\Vert_{\mathrm{DG}(B)} =  \|\nabla(u - \uCG) \|_{L^2(B)} &\leq C( h^k\|u\|_{H^{k+1}(B_0)} + \|u -\uCG\|_{L^2(B_0)} ), \label{eq:interior_est_1} \\ 
         \Vert U-\uCG\Vert_{\mathrm{DG}(B)} =  \|\nabla(U- \uCG) \|_{L^2(B)} &\leq C(h^k\|U\|_{H^{k+1}(B_0)} + \|U -\uCG\|_{L^2(B_0)}). \label{eq:interior_est_2}
    \end{align}
  We apply Lemma 4.1 by Chen and Chen \cite{chen2004pointwise}:  \eqref{eq:Lemma_chen_chen} with $D=B$ and $\widetilde{D} = B_0$.  We obtain:
\begin{equation}
    \|U -\uDG \|_{\DG(B)}  \leq C (h^k \|U\|_{H^{k+1}(B_0)} + \| U -\uDG \|_{L^2(B_0)} ).\label{eq:interior_est_3}
     \end{equation}
    Employing bounds \eqref{eq:interior_est_1}, \eqref{eq:interior_est_2} and \eqref{eq:interior_est_3} in \eqref{eq:main_tr_theorem2}, 
we obtain 
    \begin{multline}
    \| u -\uDG\|_{\DG(B)} \leq Ch^k (\|u\|_{H^{k+1}(B_0)}+ \|U\|_{H^{k+1}(B_0)}) \\ + C(\|u -\uCG\|_{L^2(B_0)} +\|U -\uCG\|_{L^2(B_0)}  +\| U -\uDG \|_{L^2(B_0)}).\label{eq:interm_bd_theor2}
    \end{multline}
    Using \eqref{eq:bounding_first_term_local}, \eqref{eq:bounding_second_term_local} and \eqref{eq:bounding_third_term_local} in \eqref{eq:interm_bd_theor2} yields, 
    \begin{equation}
    \| u -\uDG\|_{\DG(B)} \leq Ch^k(\|u\|_{H^{k+1}(B_0)}+ \|U\|_{H^{k+1} (B_0)}) + C h^{k-\theta}.
    \end{equation}
    We conclude that \eqref{eq:local_dg_k2} holds  by using bound \eqref{eq:interior_reg_U_B} in the above estimate. The proof of  bound \eqref{eq:local_dg_k1} follows the same lines: we apply \eqref{eq:interior_est_3} with $B = \Omega\setminus N$ and $B_0 = \Omega\setminus \widetilde{N}$ where $N\subset \widetilde{N}$. 
    \end{proof}
% \begin{lemma}
%     \label{lemma:bd_w_near_B0}
% Let $k \geq 1$. Fix $i \in \{ 1,\ldots k\}$. Let $D$ be a subset such that $B_{i-1} \subset D \subset B_{i}$. We have the following estimate 
% \begin{equation}
% \|\nabla(w-S_hw) \|_{L^2(D)} + h\|\nabla^2(w-S_hw) \|_{L^2(D)}  \leq Ch^k \| U - \uDG \|_{L^2(\Omega \backslash N_{i})}.\label{eq:linking_problems}
% \end{equation}
% \end{lemma}
% \begin{proof}
% Consider a neighborhood $\tilde{D}_i$ such that $D_{i} \subset \tilde{D}_i \subset B_{i}$. We recall that $$ -\Delta w = 0,  \quad \mathrm{in} \,\, B_{i}. $$ Following \cite{koppl2014optimal}, we have by an interior regularity result on $w$ (Theorem 8.10 \cite{gilbarg2015elliptic}) % \textit{please see Questions \ref{sec:questions}}),

%\input{local_l2_negativenorms_BR}
\section{The parabolic problem}\label{sec:parabolic_analysis}
In this section, we consider the time dependent problem \eqref{eq:parabolic_1}-\eqref{eq:parabolic_3} with a Dirac line source. 
The domain $\Omega$ is assumed to be convex, the curve $\Lambda$ is a $\mathcal{C}^2$ curve such that $|E\cap \Lambda|\leq C h$
for any $E\in\mathcal{E}_h$. 
% \begin{alignat}{2}
% \partial_t u - \Delta u &= f \delta_{\Lambda}, && \quad  \mathrm{in } \,\,  \Omega \times (0,T], \label{eq:parabolic_1}\\   
% u  & = 0, && \quad \mathrm{on}\,\,  \partial \Omega \times (0,T], \label{eq:parabolic_2} \\ 
% u  & = u^0,  && \quad \mathrm{in} \,\,  \{0\} \times \Omega.  \label{eq:parabolic_3}
% \end{alignat}
% Here, $f(x,t) \delta_{\Lambda}$  is given by: 
% $$ \langle f(x,t), v \rangle = \int_{\Lambda} f(x,t) v, \quad  \forall v \in L^2(0,T;L^{\infty}(\Omega)). $$
% In the above, we assume that $f(x,t) \in L^2(0,T;L^2(\Lambda))$.
 A very weak solution $u$ to \eqref{eq:parabolic_1}-\eqref{eq:parabolic_3} can be defined via the method of transposition, see \cite{gong2013error,gong2016finite}. To this end, for a given function $g \in L^2(0,T;L^2(\Omega))$, define the backward in time parabolic problem: 
\begin{alignat}{2}
  - \partial_t \psi - \Delta \psi & = g,  && \quad \mathrm{in }\,\,  \Omega \times (0,T], \label{eq:backward_parabolic_1}\\   
  \psi  & = 0, && \quad \mathrm{on}\,\,  \partial \Omega \times (0,T], \label{eq:backward_parabolic_2} \\ 
  \psi(T) & = 0,  && \quad \mathrm{in} \,\,  \{T\} \times \Omega.  \label{eq:backward_parabolic_3}
\end{alignat}
The solution $\psi$ belongs to $L^2(0,T;H^2(\Omega))$ and the following bounds hold (see Theorem 5 in Section 7.1.3 
and Theorem 4 in Section 5.9.2 in \cite{evans2010partial})
\begin{equation}
\Vert \psi \Vert_{L^{\infty}(0,T;H^1(\Omega))} \leq C\left(\|\psi\|_{L^2(0,T;H^2(\Omega))} + \|\partial_t \psi \|_{L^2(0,T;L^2(\Omega))} \right) \leq C \|g\|_{L^2(0,T; L^2(\Omega))}.\label{eq:parabolic_regularity_backward}
\end{equation}
%From Theorem 4 in Chapter 5 in \cite{evans2010partial}, we also have 
%\begin{equation}
  %\|\psi\|_{L^{\infty}(0,T;H^1(\Omega))} \leq C (\|\psi\|_{L^2(0,T;H^2(\Omega))} + \|\partial_t \psi \|_{L^2(0,T;L^2(\Omega))} )\leq C \|g\|_{L^2(0,T; L^2(\Omega))}. \label{eq:parabolic_regularity_backward_1}
  %\end{equation}
If for all $g\in L^2(0,T;L^2(\Omega))$, $u$ satisfies   
\begin{align}
  \int_0^T \int_\Omega u g = \int_{0}^T \int_{\Lambda} f \psi + \int_\Omega u^0 \psi(0),    \label{eq:linear_functional_g}
\end{align}
where $\psi \in L^{2}(0,T; H^2(\Omega))$ solves \eqref{eq:backward_parabolic_1}-\eqref{eq:backward_parabolic_3},  then $u$ is referred to as a very weak solution to \eqref{eq:parabolic_1}-\eqref{eq:parabolic_3}. From a Sobolev inequality and \eqref{eq:parabolic_regularity_backward}, we have 
\begin{align*}
  \left| \int_0^T \int_\Omega u g \right| &\leq \|f\|_{L^2(0,T;L^2(\Lambda))} \|\psi\|_{L^2(0,T;L^{\infty}(\Omega))} + \|u^0\|_{L^2(\Omega)}\|\psi\|_{L^{\infty}(0,T;L^2(\Omega))} \\ &
  \leq C \left( \|f\|_{L^2(0,T;L^2(\Lambda))}\|\psi\|_{L^2(0,T;H^2(\Omega))} + \|u^0\|_{L^2(\Omega)}   \|\psi\|_{L^{\infty}(0,T;L^2(\Omega))} \right) \\ &\leq C ( \|f\|_{L^2(0,T;L^2(\Lambda))} + \|u^0\|_{L^2(\Omega)} ) \|g\|_{L^2(0,T; L^2(\Omega))}.
\end{align*} 
Hence, the right hand side of \eqref{eq:linear_functional_g} defines a bounded linear functional on $L^2(0,T;L^2(\Omega))$. Thus, with the Lax-Milgram Theorem, a unique solution $u$ exists in the sense of \eqref{eq:linear_functional_g}. In addition, if $u^0 \in H^1(\Omega)$, then the very weak solution $u$ belongs to  $L^2(0,T;W^{1,\sigma}(\Omega)) \cap H^1(0,T;W^{-1,\sigma}(\Omega))$ for $\sigma  \in (1,2)$ and  
satisfies \cite{gong2016finite} 
\begin{equation}
\int_0^T \langle \partial_t u ,v \rangle + \int_0^T (\nabla u ,\nabla v)_\Omega = \int_0^T \int_{\Lambda} f v, \quad \forall v \in L^2(0,T;W^{1,\sigma'}_0 (\Omega)). 
\end{equation} 
We denote by $(\cdot, \cdot)_\Omega$  the $L^2$ inner product over $\Omega$. In the above, $\sigma'$ is the conjugate pair of $\sigma$,  $W^{-1,\sigma} (\Omega)$  is the dual space of $W^{1,\sigma'}_0(\Omega)$, and $\langle \cdot, \cdot \rangle$ denotes the duality pairing between $L^2(0,T;W_0^{1,\sigma}(\Omega))$ and $L^2(0,T;W^{-1,\sigma}(\Omega))$.
\subsection{Semi-discrete formulation} We introduce the continuous in time  dG approximation $\uDG(t)$ which belongs to $V_h^{k}(\mesh_h)$ for all $t > 0$ and satisfies:
% the following. For all $t>0$ and for all $v\in V_h^{k}(\mesh_h)$, 
\begin{alignat}{2}
  \int_\Omega \frac{\partial}{\partial t} \uDG(t) v + a(\uDG(t),v) &= \int_{\Lambda} f(t) v, \quad \forall t>0, \quad \forall v\in V_h^{k}(\mesh_h),\label{eq:semi_discrete_dG_1} \\ 
  \int_\Omega \uDG(0) v & = \int_\Omega u^0 v, \quad \forall v\in V_h^{k}(\mesh_h). \label{eq:semi_discrete_init}
\end{alignat}
We recall that $a$ is the symmetric bilinear form ($\epsilon = -1$ in \eqref{eq:dg_bilinear_form} and $\beta = 1$). 
We also introduce the dG approximation $\psi_h(t) \in V_h^{k}(\mesh_h)$ to $\psi(t)$ the solution of \eqref{eq:backward_parabolic_1}-\eqref{eq:backward_parabolic_3}. %For $0 \leq t < T$ and for all $ v \in V_h^{k}(\mesh_h)$, 
\begin{alignat}{2}
 - \int_\Omega \frac{\partial}{\partial t} \psi_h(t) \, v + a(\psi_h(t),v) &= \int_{\Omega} g(t) v, \quad \forall 0\leq t<T, \quad \forall v\in V_h^{k}(\mesh_h),\label{eq:semi_discrete_dG_backward} \\ 
  \psi_h(T) & = 0 . 
\end{alignat}
The main goal of this section is to establish a global estimate in $L^2(0,T;L^2(\Omega))$ for the error $\uDG-u$, see Theorem \ref{theorem:convergence_semidiscrete_3d1d}. We first establish estimates for the error $\psi_h(t) - \psi(t)$. Such estimates  that depend on the
time derivative of $\psi$ are standard \cite{riviere2008discontinuous}.  Here, we follow the arguments in \cite{chrysafinos2002error} 
and derive error bounds with constants that depend only on $\psi$ and not on $\partial_t \psi$. %The next Lemmas differ from the estimates in \cite{riviere2008discontinuous} since the estimates we obtain do not depend on any norms involving the time derivative of $\psi$. 
\begin{lemma} \label{lemma:linfinity_conv}
There exists a constant $C$ independent of $h$ such that  
\begin{equation}
\|\psi(0) - \psi_h(0)\|_{L^2(\Omega)} +  \|\psi-\psi_h\|_{L^2(0,T;\DG)} \leq Ch\left( \|\psi\|_{L^\infty(0,T;H^1(\Omega))} + \|\psi \|_{L^2(0,T;H^2(\Omega))} \right).
\end{equation} 
\end{lemma}
\begin{proof}
The proof applies  the arguments in \cite{chrysafinos2002error} to a dG discretization of the backward problem. Define $R_h \psi(t) \in V_h^k (\mesh_h)$ as the elliptic projection of $\psi(t)$ 
\begin{equation}
a(R_h \psi(t) - \psi(t),v) = 0, \quad \forall v \in V_h^{k}(\mesh_h),  \,\, \forall t \in (0,T].   \label{eq:elliptic_projection}
\end{equation}
From the consistency property of the dG discretization,  \eqref{eq:elliptic_projection} and  \eqref{eq:semi_discrete_dG_backward}, we have the following  relation. 
\begin{equation}
-(\partial_t \psi(t) - \partial_t \psi_h (t), v)_\Omega + a(R_h \psi (t) - \psi_h(t),v )  = 0, \quad \forall v \in V_h^k(\mesh_h). \label{eq:orthog_semi_discrete} 
\end{equation} 
Let $P_h \psi(t)$ be the $L^2$ projection of $\psi(t)$. Thus, with the above, we can write
\begin{multline}
 - \frac12 \frac{d}{dt}\| \psi - \psi_h \|_{L^2(\Omega)}^2
+ a(R_h \psi_h (t) - \psi_h(t), R_h \psi(t) - \psi_h(t)) \\ 
= - (\partial_t \psi(t)- \partial_t \psi_h(t), \psi(t)- P_h \psi(t) )_\Omega + a(R_h\psi(t) - \psi_h(t), R_h\psi(t) - P_h \psi (t) ). 
\label{eq:Phpsiid}
\end{multline}
Using the definition of the $L^2$ projection repeatedly yields: 
%In addition, let $P_h u(t)$ denote the $L^2$ projection of $u(t)$ onto $V_h^k(\mesh_h)$ for all $t$. 
%Observe that 
%\begin{align*}
%- \frac12 \frac{d}{dt}\| \psi &- \psi_h\|_{L^2(\Omega)}^2 + a(R_h \psi (t) - \psi(t), R_h \psi(t) - \psi_h(t) ) \\ 
%&= - ( \partial_t \psi(t) - \partial_t \psi_h(t), \psi(t) - \psi_h(t) )_\Omega + a(R_h \psi(t) - \psi_h(t), R_h \psi (t) - \psi_h(t)). 
%\end{align*} 
\begin{multline*}
  (\partial_t \psi(t) - \partial_t \psi_h(t), \psi(t) - P_h \psi(t) )_\Omega 
 =   (\partial_t \psi (t), \psi(t) - P_h \psi(t) )_\Omega  \\ 
= (\partial_t \psi(t) - \partial_t P_h \psi(t), \psi(t) - P_h \psi(t))_\Omega
= \frac{1}{2} \frac{d}{dt}\| \psi(t) - P_h \psi (t)\|_{L^2(\Omega)}^2. 
\end{multline*}
With the coercivity and continuity properties \eqref{eq:coercivity_property}, \eqref{eq:continuity_prop}, and the above relation,  equation \eqref{eq:Phpsiid} becomes:
\begin{multline*}
-  \frac12 \frac{d}{dt}\| \psi - \psi_h\|_{L^2(\Omega)}^2 +\frac12 \|R_h \psi(t) - \psi_h(t)\|_{\DG}^2 \\
%& \leq  - (\partial_t \psi(t)- \partial_t \psi_h(t), \psi(t)- P_h \psi_h(t) ) + a(R_h\psi(t) - \psi_h(t), R_h \psi(t) - P_h \psi(t)) ) \\ 
   \leq -\frac{1}{2} \frac{d}{dt}\| \psi(t) - P_h \psi(t)\|_{L^2(\Omega)}^2 
+ C\|R_h \psi(t) - \psi_h(t)\|_{\DG}\|R_h \psi(t) - P_h \psi(t)\|_{\DG}.
 \end{multline*}
An application of Young's inequality, integration from $0$ to $T$ and approximation properties yield:
\[
\Vert \psi(0)-\psi_h(0)\Vert_{L^2(\Omega)}^2 + \frac12 \int_0^T  \Vert R_h\psi(t)-\psi_h(t)\Vert_{\DG}^2
\leq C h^2 \Vert \psi(0)\Vert_{H^1(\Omega)}^2 + C h^2 \Vert \psi\Vert_{L^2(0,T;H^2(\Omega))}^2.
\]
The final result follows with a triangle inequality.
\end{proof}
\begin{lemma}\label{lemma:bound_dg_l2rhs}
 Assume that $\psi$ belongs to $L^2(0,T;H^{s}(\Omega))$ for $s > 3/2$. Then, there exists a constant $C>0$ independent of $h$ such that 
    $$ \|\psi - \psi_h \|_{L^2(0,T;L^2(\Omega))} \leq C h^{\min(k+1,s)} \|\psi\|_{L^2(0,T;H^{s}(\Omega))}. $$
\end{lemma} 
\begin{proof}
The proof extends the arguments of Theorem 2.5 in \cite{thomee2007galerkin} given for  the continuous Galerkin discretization and adapts it to the backward parabolic problem. We define two linear operators $Q: L^2(\Omega) \rightarrow H^1_0(\Omega) \cap H^2(\Omega)$ and $Q_h: L^2(\Omega) \rightarrow V_h^k(\mesh_h)$ as follows. For $\phi \in L^2(\Omega)$,
\begin{alignat*}{3}
Q\phi &= z, \,\,&& \mathrm{with} -\Delta z = \phi \, \mathrm{ in } \, \Omega \,\, && \mathrm{and} \,\, z \vert_{\partial \Omega} = 0, \\
Q_h \phi &= z_h, \,\,&& \mathrm{with}  \,\, a (z_h, v ) = (\phi, v)_\Omega, \,\, && \forall v \in V_h^k(\mesh_h).
\end{alignat*}
It is clear that 
\begin{equation}
Q(\Delta w) = -w, \quad \forall w \in H^2(\Omega).\label{eq:Qprop}
\end{equation}
The operator  $Q_h$ is selfadjoint since $a$ is symmetric. Indeed, for any $z, w \in L^2(\Omega)$, 
\begin{equation}
(Q_h z, w)_\Omega=a(Q_h w, Q_h z ) = a(Q_h z, Q_h w) = (z,Q_h w)_\Omega. 
\end{equation}
We also define the discrete Laplacian operator $\Delta_h:  V_h^k(\mesh_h) \rightarrow V_h^k(\mesh_h)$ satisfying 
$$ (\Delta_h w_h , v)_\Omega = -a(w_h, v), \quad \forall v \in V_h^k(\mesh_h). $$
% From the above definitions, we observe that 
% \begin{equation}
% a(Q_h(\Delta_h u_h) + u_h,v) = (\Delta_h u_h, v) + a(u_h,v) = 0.  
% \end{equation}
Since $a$ is coercive, we also have that $Q_h(\Delta_h w_h) = - w_h.$ With the discrete Laplacian, we can write \eqref{eq:semi_discrete_dG_backward} as 
$$-\partial_t \psi_h(t)- \Delta_h \psi_h(t) = P_h g(t). $$
Applying the operator $Q_h$ to the above equality, we obtain 
$$ -Q_h \partial_t \psi_h(t) + \psi_h (t) = Q_h  P_h g(t) = Q_h g(t).  $$
On the continuous level, we also have 
$$ - Q \frac{\partial}{\partial t} \psi(t) + \psi(t) = Q g(t). $$
Define $e_h = \psi_h - \psi$ and $\rho_h = -\psi - Q_h (\Delta \psi)$, then 
\begin{multline}
-Q_h \partial_t e_h + e_h =  Q_h g + (Q_h-Q) \partial_t  \psi - Q g
 =   (Q- Q_h )\left(-\partial_t \psi - g\right) = (Q-Q_h)(\Delta \psi) = \rho_h. \label{eq:error_eq}
\end{multline} 
The last equality is obtained with \eqref{eq:Qprop}. This implies
\[
(-Q_h \partial_t e_h, e_h)_\Omega + \frac12 \Vert e_h \Vert_{L^2(\Omega)}^2 \leq \frac12 \Vert \rho_h\Vert_{L^2(\Omega)}^2.
\]
Since $Q_h$ is self-adjoint and $Q_h$ commutes with the derivative in time operator, we obtain
%This implies that \cite{thomee2007galerkin}: $$ 2\left(Q_h \frac{\partial}{\partial t} e, e\right) = \frac{\partial }{\partial t} (e, Q_h e).$$
%We multiply \eqref{eq:error_eq} by $e(t)$, integrate over $\Omega$, use the above identity, multiply by 2 and use Cauchy-Schwarz and Young's inequality to bound the last term. We obtain 
\begin{equation}
   - \frac{\partial }{\partial t} (e_h, Q_h e_h)_\Omega + \|e_h\|_{L^2(\Omega)}^2 \leq \|\rho_h \|_{L^2(\Omega)}^2. 
\end{equation}
We integrate from $t=0$ to $t = T$ and observe that by coercivity we have 
$$ (e_h,Q_h e_h)_\Omega = a(Q_h e_h,Q_h e_h) \geq \frac{1}{2}\|Q_h e_h\|^2_{\DG}.  $$ 
Hence, since $e_h(T) = 0$, 
\begin{equation}
    \frac{1}{2}\|Q_h e_h(0)\|^2_{\DG} + \int_{0}^T \|e_h\|_{L^2(\Omega)}^2 \leq \int_{0}^T \|\rho_h\|_{L^2(\Omega)}^2. \label{eq:error_eq1}
\end{equation}
In addition, note that by consistency of the dG discretization 
$$  a (Q_h( -\Delta \psi), v) = ( -\Delta \psi  , v) = a(\psi,v), \quad  \forall v \in V_h^k(\mesh_h).$$
Thus, we have, if $\psi$ belongs to $L^2(0,T;H^{s}(\Omega))$
\[
\Vert \rho_h \Vert_{L^2(\Omega)} = \Vert \psi + Q_h(\Delta \psi)\Vert_{L^2(\Omega)} \leq C  h^{\min(k+1,s)} \|\psi\|_{L^2(0,T;H^{s}(\Omega))}.
\]
We can then conclude with \eqref{eq:error_eq1}.
\end{proof}
With Lemma \ref{lemma:linfinity_conv} and Lemma \ref{lemma:bound_dg_l2rhs}, we show the main result of this section. 
\begin{theorem}\label{theorem:convergence_semidiscrete_3d1d}
Let $u$ be the very weak solution to \eqref{eq:parabolic_1}-\eqref{eq:parabolic_3} and let $\uDG$ satisfies \eqref{eq:semi_discrete_dG_1}-\eqref{eq:semi_discrete_init}.
There exists a constant $C$ independent of $h$  such that for any $\theta \in (0,\frac{1}{2})$, 
\begin{align}
  %\|\uDG\|_{L^{\infty}(0,T;L^2(\Omega))} & \leq C, \label{eq:stability_parabolic} \\
\|\uDG - u\|_{L^2(0,T;L^{2}(\Omega))} & \leq C(\theta) h^{1-\theta} (\|f\|_{L^2(0,T;L^2(\Lambda))} + \|u^0\|_{L^2(\Omega)}). \label{eq:convergence_parabolic}
\end{align}
\end{theorem}
\begin{proof}
  The proof is based on a duality argument and follows similar techniques as the proof of Theorem 3.4 in \cite{gong2013error}. Define $\chi(t) = \uDG(t) - u(t)$. Fix $g \in L^2(0,T;L^2(\Omega))$ and  let $\psi$ solve \eqref{eq:backward_parabolic_1}-\eqref{eq:backward_parabolic_3}. With \eqref{eq:linear_functional_g},  consistency of the dG discretization for \eqref{eq:backward_parabolic_1}-\eqref{eq:backward_parabolic_3}, and the definition of $\uDG(0)$ (see \eqref{eq:semi_discrete_init}),  we have 
  \begin{align*}
  \int_0^T (\chi, g)_\Omega &= \int_0^T ( \uDG, -\partial_t \psi - \Delta \psi )_\Omega 
- \int_{0}^T \int_\Lambda f \psi - (u^0 ,\psi(0) )_\Omega\\
  &=  \int_0^T -(\partial_t \psi, \uDG)_\Omega + \int_0^T a(\psi, \uDG) - \int_{0}^T \int_\Lambda f \psi - (u^0, \psi(0))_\Omega \\
  & =\int_0^T (-\partial_t \psi_h,  \uDG)_\Omega + \int_0^T a(\psi_h, \uDG) - \int_{0}^T \int_\Lambda f \psi - ( u^0, \psi(0))_\Omega \\
  & = ( \psi_h(0), \uDG(0))_\Omega + \int_0^T (\partial_t \uDG , \psi_h)_\Omega   + \int_0^T a(\psi_h, \uDG) - \int_{0}^T \int_\Lambda f \psi - (u^0, \psi(0))_\Omega \\ 
  & = (u^0 ,\psi_h(0) - \psi(0))_\Omega + \int_0^T \int_{\Lambda} f (\psi_h -\psi)  = R_1 + R_2. 
  \end{align*}
For $R_1$, we use Cauchy-Schwarz's inequality, Lemma~\ref{lemma:linfinity_conv} and \eqref{eq:parabolic_regularity_backward}: 
\begin{equation}
  |R_1| \leq \|u^0\|_{L^2(\Omega)} \|\psi_h(0) - \psi(0)\|_{L^2(\Omega)} \leq C h \|u^0\|_{L^2(\Omega)} \|g\|_{L^2(0,T; L^2(\Omega))}. \label{eq:bound_T2}
\end{equation} 
  For the term $R_2$, we use the following trace inequality valid for
any $2<q<3$ and $q\leq r < q/(3-q)$  (see Theorem 4.12 in \cite{adams2003sobolev} and Proposition 2.3 in \cite{nguyen2001}).
%\Rd Do we cite Theorem 7.58 in Adams 1975? \Bk
  \begin{equation}
  \|v\|_{L^r(\Lambda)} \leq C(q)\|v\|_{W^{1,q}(\Omega)}, \quad  \forall v \in W^{1,q}(\Omega).  \label{eq:trace_inequality_Lambda}
  \end{equation}
We denote by $L_h \psi$  the Lagrange interpolant of $\psi$ in  $W_h^k(\mesh_h)$. From Theorem 3.1.6 in \cite{ciarlet2002finite}, we have 
\begin{equation}
\|\psi - L_h \psi\|_{W^{1,q}(E)} \leq C(q) h^{\frac{3}{q} - \frac12} |\psi|_{H^2(E)}, \quad \forall E \in \mesh_h. 
 \end{equation} 
 From the above bound  and Jensen's inequality, we obtain
 \begin{multline}
 \|\psi - L_h \psi\|_{W^{1,q}(\Omega)} = \left( \sum_{E \in \mesh_h}  \|\psi - L_h \psi\|^q_{W^{1,q}(E)} \right)^{1/q}  \leq h^{\frac{3}{q}-\frac{1}{2}} \left(\sum_{E\in \mesh_h}  |\psi|^q_{H^2(E)} \right)^{1/q} \leq    h^{\frac{3}{q}-\frac{1}{2}} |\psi|_{H^2(\Omega)}. \label{eq:approximation_psi}
 \end{multline}
% In the above, we used that $\left(\sum_m a_m^q\right)^{1/q} \leq \left(\sum_m a_m^2\right)^{1/2}$ for $q>2$ and $\{a_m\} \geq 0$.  
Let $r$ and $q$ satisfy the conditions in \eqref{eq:trace_inequality_Lambda} and let $r'$ be the conjugate exponent of $r$ ($1/r+1/r'=1$). 
Note that $L_h \psi \in W^{1,q}(\Omega)$. Hence, with \eqref{eq:trace_inequality_Lambda} and \eqref{eq:approximation_psi}, we obtain 
\begin{equation}
  \|\psi - L_h \psi \|_{L^r(\Lambda)} \leq C(q) \|\psi - L_h \psi\|_{W^{1,q}(\Omega)} \leq C(q) h^{\frac{3}{q} - \frac{1}{2}} |\psi|_{H^2(\Omega)}. \label{eq:3d1d_trace_lr_approx}
\end{equation}
 With Cauchy-Schwarz's inequality, \eqref{eq:3d_1d_inverse_estimate}, and \eqref{eq:3d1d_trace_lr_approx},  we have 
  \begin{align}
  \int_{\Lambda} f(\psi_h - \psi) &= \sum_{E\in \mathcal{T}_\Lambda}\int_{E\cap \Lambda} f (\psi_h - L_h \psi_h) + \int_{\Lambda} f(L_h \psi_h - \psi) \nonumber \\ 
& \leq  \sum_{E\in \mathcal{T}_\Lambda} \|f\|_{L^1(E\cap \Lambda)} \| \psi_h - L_h \psi_h\|_{L^{\infty}(E)} + \|f\|_{L^{r'}(\Lambda)}\|L_h \psi_h - \psi\|_{L^r(\Lambda)} \nonumber \\
& \leq  C \sum_{E\in \mathcal{T}_\Lambda}  |E \cap \Lambda|^{1/2}\|f\|_{L^2(E\cap \Lambda)}h^{-3/2} \| \psi_h - L_h \psi_h\|_{L^{2}(E)} + C(q)h^{\frac{3}{q} - \frac12} \|f\|_{L^{r'}(\Lambda)}|\psi|_{H^2(\Omega)} \nonumber \\ 
& \leq Ch^{-1} \|f\|_{L^2(\Lambda)} \|\psi_h - L_h \psi\|_{L^2(\Omega)} +  C(q)h^{\frac{3}{q} - \frac12} \|f\|_{L^{2}(\Lambda)}|\psi|_{H^2(\Omega)}.  \label{eq:fully_discrete_error_0}
  \end{align}
The last inequality holds since $r' < 2$. From Lemma \ref{lemma:bound_dg_l2rhs}, approximation properties, and \eqref{eq:parabolic_regularity_backward}, it then follows that 
\begin{align}
|R_2| & \leq C h^{-1} \|f\|_{L^2(0,T;L^2(\Lambda))}\|\psi_h - L_h \psi\|_{L^2(0,T;L^2(\Omega))}+C(q) h^{\frac{3}{q} - \frac12} \|f\|_{L^{2}(0,T;L^2(\Lambda))}|\psi|_{L^2(0,T;H^2(\Omega))} \nonumber \\ 
& \leq  C h \|f\|_{L^2(0,T;L^2(\Lambda))}\|\psi\|_{L^2(0,T;H^2(\Omega))}+C(q) h^{\frac{3}{q} - \frac12} \|f\|_{L^{2}(0,T;L^2(\Lambda))}|\psi|_{L^2(0,T;H^2(\Omega))} \nonumber \\ 
& \leq  C h \|f\|_{L^2(0,T;L^2(\Lambda))}\|g\|_{L^2(0,T; L^2(\Omega))}+C(q) h^{\frac{3}{q} - \frac12} \|f\|_{L^{2}(0,T;L^2(\Lambda))}\|g\|_{L^2(0,T; L^2(\Omega))}.  \nonumber 
\end{align}  
For any $\theta \in (0,1/2)$, choose $q = 6/(3-2\theta)$. The bound for $R_2$ becomes
\begin{equation}
  |R_2| \leq C(\theta) h^{1-\theta}\|f\|_{L^2(0,T;L^2(\Lambda))} \|g\|_{L^2(0,T; L^2(\Omega))}. \label{eq:bound_T1} 
\end{equation}
We remark that
\[
\|\chi\|_{L^2(0,T;L^2(\Omega))}  =    \sup_{\begin{array}{c} g \in L^2(0,T;L^2(\Omega))\\ g \neq 0\end{array}} \frac{|\int_{0}^T (\chi, g)_\Omega |}{\|g\|_{L^2(0,T; L^2(\Omega))}}.
\]
Therefore, with \eqref{eq:bound_T2} and \eqref{eq:bound_T1}, we can conclude.
%\begin{equation}
%\|\chi\|_{L^2(0,T;L^2(\Omega))} \Rd = \Bk ?\leq   \sup_{\begin{array}{c} g \in L^2(0,T;L^2(\Omega))\\ g \neq 0\end{array}} \frac{|\int_{0}^T (\chi, g)_\Omega |}{\|g\|_{L^2(0,T; L^2(\Omega))}} \leq C(\theta) h^{1-\theta}(\|f\|_{L^2(0,T;L^2(\Lambda))} + \|u^0\|_{L^2(\Omega)}) . 
%\end{equation}
%The result is then concluded. 
\end{proof}
\subsection{Fully discrete formulation}
In this section, we consider a backward Euler discretization of problem \eqref{eq:parabolic_1}-\eqref{eq:parabolic_3}.  To simplify notation, we drop the subscript $\DG$ on the discrete solution,  namely $u_h^n = u_h^{\mathrm{DG},n}$. Let $\tau>0$ denote the time step size and consider a uniform partition of the
time interval $(0, T]$ into $N_T$ subintervals. We define a sequence of  dG approximations $ (u_h^n)_{0\leq n\leq N_T} \in V_h^{k}(\mesh_h)$ such that for all $n = 1,\ldots, N_T$  
\begin{align}
(u_h^n - u_h^{n-1},v)_\Omega + \tau a(u_h^n, v) = \tau \int_{\Lambda} f(t^n)   v, \quad  \forall v \in V_h^k(\mesh_h),  \label{eq:fully_discrete_3d1d}
\end{align}  with $u_h^0 = \uDG(0)$ defined by \eqref{eq:semi_discrete_init}. 
The existence and uniqueness of $ (u_h^n)_{0\leq n\leq N_T}$ follows from a standard proof by contradiction where the coercivity of $a$ \eqref{eq:coercivity_property} is used. From the fully discrete solutions, we construct a piecewise constant in time solution, denoted by $u_{h,\tau}$, as follows:
\[
u_{h,\tau}(t,\bm{x}) = u_h^n(\bm{x}), \quad t^{n-1} < t \leq t^n, \quad n\geq 1, \quad u_{h,\tau}(0,\bm{x}) = u_h^0(\bm{x}), \quad
\forall \bm{x}\in\Omega.
\]
The main result of this section is the following convergence theorem. For convenience, we define
\[
\Vert f \Vert_{\ell^2(0,T;L^2(\Lambda))} = \left(\tau \sum_{n=1}^{N_T} \Vert f(t^n) \Vert_{L^2(\Lambda)}^2\right)^{1/2}.
\]
\begin{theorem}\label{thm:discretecv}
Assume that $\partial_t f \in L^2(0,T;L^1(\Lambda))$ and let $\theta$ be in $(0,\frac{1}{2})$. 
There exists a constant $C$ independent of $h$ and $\tau$, but depending of $\theta$,  such that 
\begin{multline}
\|u-u_{h,\tau}\|_{L^2(0,T;L^2(\Omega))} \leq C  (\tau h^{-1} +h )
\left( \Vert f \Vert_{\ell^2(0,T;L^2(\Lambda))}   
+\|\partial_t f\|_{L^2(0,T; L^1(\Lambda))} +  \|u^0\|_{L^2(\Omega)} \right) \\  + C h^{1-\theta}   \|f\|_{L^2(0,T;L^2(\Lambda))} .
\end{multline}
As a consequence, if $\tau \leq h^{2-\theta}$, we have 
\begin{equation}
  \|u-u_{h,\tau} \|_{L^2(0,T;L^2(\Omega))} \leq C h^{1-\theta} 
(\|f\|_{L^2(0,T;L^2(\Lambda))} + \|\partial_t f\|_{L^2(0,T; L^1(\Lambda))} +   \|f\|_{\ell^2(0,T;L^2(\Lambda))} + \|u^0\|_{L^2(\Omega)}).
\end{equation} 
\end{theorem}
The proof of the theorem requires an intermediate bound on the discrete solutions, that is stated in the following lemma.
\begin{lemma} \label{lemma:stability_fullydiscrete_3d1d}
  There exists a constant $C$ independent of $\tau$ and $h$ such that the following estimate holds. For $1 \leq m \leq N_T$, 
\begin{equation}
\sum_{n=1}^m  \|u_h^n - u_h^{n-1}\|_{L^2(\Omega)}^2 +  \tau \sum_{n=1}^m \|u_h^n - u_h^{n-1}\|^2_{\DG}  
+ \tau \|u_h^m\|_{\DG}^2 \leq C\tau h^{-2} \left( \|u^0\|_{L^2(\Omega)}^2 + \Vert f \Vert_{\ell^2(0,T;L^2(\Lambda))}^2 \right). 
\end{equation}
\end{lemma}
\begin{proof}
Let $v = u_h^n - u_h^{n-1}$ in \eqref{eq:fully_discrete_3d1d}. Using the symmetry of $a$, we obtain 
\begin{multline*}
  \|u_h^n - u_h^{n-1}\|_{L^2(\Omega)}^2 +\frac{\tau}{2} \left( a(u_h^n,u_h^n) 
-a(u_h^{n-1},u_h^{n-1}) + a(u_h^n - u_h^{n-1}, u_h^n -u_h^{n-1})\right) = \tau \int_{\Lambda} f(t^n) (u_h^{n} - u_h^{n-1}). 
\end{multline*}
We observe that by H\"{o}lder's inequality and \eqref{eq:3d_1d_inverse_estimate},
\begin{align*}
  \int_{\Lambda} f(t^n) (u_h^{n} - u_h^{n-1}) & \leq \sum_{E \in \mathcal{T}_\Lambda} |E \cap \Lambda|^{1/2} \|f(t^n)\|_{L^2(E)} \|u_h^n - u_h^{n-1}\|_{L^\infty(E)} \\&\leq C \sum_{E\in \mathcal{T}_\Lambda} h^{-1} \|f(t^n)\|_{L^2(E\cap \Lambda)}\|u_h^n - u_h^{n-1}\|_{L^2(E)}.  
\end{align*}
With the coercivity \eqref{eq:coercivity_property} and the above bound, we obtain
\begin{multline*}
  \|u_h^n - u_h^{n-1}\|_{L^2(\Omega)}^2 +\frac{\tau}{2} a(u_h^n,u_h^n) - \frac{\tau}{2} a(u_h^{n-1},u_h^{n-1}) 
+ \frac{\tau}{4}\|u_h^n - u_h^{n-1}\|^2_{\DG}  \\  
%\leq C\tau h^{-1} \|f\|_{L^2(\Lambda)} \|u_h^{n} - u_h^{n-1}\|_{L^2(\Omega)} 
\leq C \tau^2 h^{-2} \|f(t^n)\|^2_{L^2(\Lambda)} + \frac12  \|u_h^{n} - u_h^{n-1}\|_{L^2(\Omega)}^2.
\end{multline*}
We sum the resulting inequality from $n = 1$ to $n =m $ and use the coercivity \eqref{eq:coercivity_property}
\begin{equation*}
 \frac12 \sum_{n=1}^{m}\|u_h^n - u_h^{n-1}\|_{L^2(\Omega)}^2 
+ \frac{\tau}{4} \|u_h^m\|_{\DG}^2 + \frac{\tau}{4} \sum_{n=1}^m \|u_h^n - u_h^{n-1}\|^2_{\DG}  
\leq \frac{\tau}{2} a(u_h^0, u_h^0) +   C\tau^2 h^{-2} \sum_{n=1}^{m} \|f(t^n)\|^2_{L^2(\Lambda)}.
 \end{equation*}
% From \eqref{eq:3d_1d_inverse_esimate_nabla} and \eqref{eq:3d_1d_trace_estimate_discrete}, it easy to see that $\|v_h\|_{\DG} \leq Ch^{-1}\|v_h\|. $ 
With the continuity of $a$ \eqref{eq:continuity_prop}, an inverse inequality  and  the stability of the $L^2$ projection, we have 
 \begin{equation}
a(u_h^0, u_h^0)\leq C\|u_h^0\|^2_{\DG} \leq C h^{-2}\|u_h^0\|_{L^2(\Omega)}^2 \leq  C h^{-2}\|u^0\|_{L^2(\Omega)}^2.
 \end{equation}
With the above bound, we conclude the proof. 
\end{proof}
%We now prove Theorem~\ref{thm:discretecv}
%The previous bound allows us to show the following convergence result.
%\begin{theorem}
%Define the error function $\eta(t) \in L^2(0,T;L^2(\Omega))$ as 
%\begin{equation}
 %\eta(t) = u_h^n - u(t), \quad \forall  t^{n-1}<t\leq t^n, \,\, \forall n \geq 1,  \quad \eta(0) = u_h^{0}-u(0).
%\end{equation}  
%Assume that $\partial_t f \in L^2(0,T;L^2(\Lambda))$, we have for $\theta \in (0,\frac{1}{2})$ 
%\begin{multline}
%\|\eta\|_{L^2(0,T;L^2(\Omega))} \leq C  (\tau h^{-1} +h )\left(\|f\|_{\ell^2(0,T;L^2(\Lambda))}  +\|\partial_t f\|_{L^2(0,T; L^2(\Lambda))} +  \|u^0\| \right) \\  + C(\theta) h^{1-\theta}   \|f\|_{L^2(0,T;L^2(\Lambda))} .
%\end{multline}
%In particular, if $\tau \leq h^{2-\theta}$, we have 
%\begin{equation}
  %\|\eta\|_{L^2(0,T;L^2(\Omega))} \leq C(\theta) h^{1-\theta} (\|f\|_{H^1(0,T; L^2(\Lambda))} +   \|f\|_{\ell^2(0,T;L^2(\Lambda))} + \|u^0\|).
%\end{equation} 
%\end{theorem}
\begin{proof}
[Proof of Theorem~\ref{thm:discretecv}]. The proof uses some techniques from the proof of Theorem 3.4 in \cite{gong2016finite}. We first fix  $g \in L^2(0,T;L^2(\Omega))$ and consider $\psi$ the solution of \eqref{eq:backward_parabolic_1}-\eqref{eq:backward_parabolic_3}. From \eqref{eq:linear_functional_g}, we have 
\begin{align}
\int_{0}^T (u_{h,\tau}-u, g)_\Omega = \sum_{n=1}^{N_T} \int_{t^{n-1}}^{t^n} (u_h^n, g)_\Omega 
- (u^0 , \psi(0))_\Omega - \int_{0}^T \int_\Lambda f \psi.  \label{eq:fully_discrete_3d1d_0}
\end{align}
We rewrite the first term in the right-hand side as
\begin{align*}
\int_{t^{n-1}}^{t^n}( u_h^n, g)_\Omega = \int_{t^{n-1}}^{t^n} (u_h^n, -\partial_t \psi - \Delta \psi )_\Omega 
=  - ( u_h^n, \psi(t^n) - \psi(t^{n-1}))_\Omega + \int_{t^{n-1}}^{t^n} a(u_h^n, \psi) \\ 
=  (u_h^n - u_h^{n-1}, \psi(t^{n-1}) )_\Omega - \left((u_h^n , \psi(t^n) )_\Omega  -(u_h^{n-1},\psi(t^{n-1}))_\Omega  \right)  
+ \int_{t^{n-1}}^{t^n} a(u_h^n, \psi).
\end{align*}
Since $\psi(T) = 0$,  \eqref{eq:fully_discrete_3d1d_0} reads 
\begin{multline}
  \int_{0}^T ( u_{h,\tau}-u , g)_\Omega = \sum_{n=1}^{N_T} (u_h^n - u_h^{n-1}, \psi(t^{n-1}))_\Omega 
+ \sum_{n=1}^{N_T} \int_{t^{n-1}}^{t^n} a(u_h^n, \psi)  \\ 
- (u^0 - u_h^0, \psi(0))_\Omega - \int_{0}^T \int_\Lambda f \psi.  \label{eq:fully_discrete_3d1d_1}
  \end{multline}
For each $t \in (t^{n-1}, t^n]$, choose $v = R_h \psi(t)$ in \eqref{eq:fully_discrete_3d1d} (recall that $R_h \psi$ is defined by
\eqref{eq:elliptic_projection}). Integrate the resulting equation from $t^{n-1}$ to $t^n$, sum from $n=1$ to $n=N_T$, and divide by $\tau$. We obtain  
\begin{equation}
  \sum_{n=1}^{N_T} \int_{t^{n-1}}^{t^n} a(u_h^n, R_h \psi(t)) 
= - \frac{1}{\tau}\sum_{n=1}^{N_T} \int_{t^{n-1}}^{t^n} (u_h^n - u_h^{n-1}, R_h \psi(t))_\Omega 
+ \sum_{n=1}^{N_T} \int_{t^{n-1}}^{t^n} \int_{\Lambda} f(t^n) \, R_h \psi(t). 
\end{equation}
With the definition of \eqref{eq:elliptic_projection}, \eqref{eq:fully_discrete_3d1d_1} becomes 
\begin{align}
  \int_{0}^T  ( u_{h,\tau}-u, g )_\Omega = & \frac{1}{\tau}\sum_{n=1}^{N_T} \int_{t^{n-1}}^{t^n} (u_h^n - u_h^{n-1},\psi(t^{n-1}) - R_h \psi(t))_\Omega \nonumber\\   
& - (u^0 - u_h^0,\psi(0))_\Omega
+ \sum_{n=1}^{N_T} \int_{t^{n-1}}^{t^n} \int_{\Lambda} ( f(t^n) R_h \psi(t) - f(t) \psi(t)) %\nonumber\\
= E_1 +E_2 + E_3.  \label{eq:fully_discrete_3d1d_2}
  \end{align}
For $E_1$, we introduce $\psi(t)$ and  write  
\[
(u_h^n-u_h^{n-1}, \psi(t^{n-1}) - R_h \psi(t))_\Omega = 
-(u_h^n-u_h^{n-1},\psi(t)-R_h \psi(t)
+\int_{t^{n-1}}^t \partial_t \psi)_\Omega.
\]
Therefore,  
%\[ \psi^{n-1} - \psi(t) = \overline{\partial_t \psi} = -  \int_{t^{n-1}}^{t} \partial_t \psi.  \]
%Then, we have 
%\begin{align}
  %(u_h^n -u_h^{n-1},  \overline{\partial_t \psi}) \leq \|u_h^n -u_h^{n-1}\| \|\overline{\partial_t \psi}\|\leq \|u_h^n -u_h^{n-1}\| |t - t^{n-1} |^{1/2}\|\partial_t \psi\|_{L^2(t^{n-1},t; L^2(\Omega))}.
%\end{align}
using error bounds of the elliptic projection, we obtain  
%\begin{align}
%(u_h^{n} - u_h^{n-1}, \psi(t) - R_h\psi(t)) \leq C h^2 \|u_h^{n} - u_h^{n-1}\| \|\psi(t)\|_{H^{2}(\Omega)}. 
%\end{align}
%Hence, the term $E_1$ is bounded by 
\begin{align}
|E_1| & \leq 
 C \tau^{-1} h^2 \sum_{n=1}^{N_T} \int_{t^{n-1}}^{t^{n}}  \|u_h^{n} - u_h^{n-1}\|_{L^2(\Omega)} \|\psi(t)\|_{H^{2}(\Omega)}   \nonumber \\
&\quad + \tau^{-1} \sum_{n=1}^{N_T} \int_{t^{n-1}}^{t^{n}} \|u_h^n -u_h^{n-1}\|_{L^2(\Omega)} \,  (t - t^{n-1} )^{1/2}\|\partial_t \psi\|_{L^2(t^{n-1},t; L^2(\Omega))} \nonumber  \\ 
& \leq  C \tau^{-\frac12} h^2 \sum_{n=1}^{N_T}  \|u_h^{n} - u_h^{n-1}\|_{L^2(\Omega)} \|\psi\|_{L^2(t^{n-1}, t^{n};H^{2}(\Omega))} \nonumber +  C \tau^{\frac12} \sum_{n=1}^{N_T} \|u_h^n -u_h^{n-1}\|_{L^2(\Omega)} \|\partial_t \psi\|_{L^2(t^{n-1},t^n; L^2(\Omega))}  \nonumber \\ 
& \leq C \left( \sum_{n=1}^{N_T} \|u_h^n - u_h^{n-1}\|_{L^2(\Omega)}^2 \right)^{1/2} 
(\tau^{-1/2} h^2 \|\psi\|_{L^2(0, T; H^{2}(\Omega))}
+\tau^{1/2}\|\partial_t \psi\|_{L^2(0,T;L^2(\Omega))}).  \label{eq:bound_E1_0_3D1D}
\end{align}
With Lemma \ref{lemma:stability_fullydiscrete_3d1d} and \eqref{eq:parabolic_regularity_backward}, \eqref{eq:bound_E1_0_3D1D} reads 
\begin{equation}
  |E_1| \leq C (\tau h^{-1} +h )\|g\|_{L^2(0,T;L^2(\Omega))} \left(\|f\|_{\ell^2(0,T;L^2(\Lambda))}  + \|u^0\|_{L^2(\Omega)} \right). \label{eq:bound_E1_3D1D}
\end{equation}
The term $E_2$ is easily handled 
since $u_h^0$ is the $L^2$ projection of $u^0$. We use approximation properties  of the Lagrange operator $L_h$  and \eqref{eq:parabolic_regularity_backward} 
 \begin{equation} 
   E_2 = (u_h^0 - u^0, \psi(0) - L_h\psi(0))_\Omega  \leq  C h\|u^0\|_{L^2(\Omega)} \|\psi(0)\|_{H^1(\Omega)} \leq  Ch\|u^0\|_{L^2(\Omega)}\|g\|_{L^2(0,T;L^2(\Omega))}. \label{eq:bound_E_3_fullydiscrete_3d1d}
 \end{equation}
For the term $E_3$, we write 
\begin{multline*}
  \int_{\Lambda} (f(t^n) R_h \psi(t) - f(t) \psi(t)) = 
\sum_{E \in \mathcal{T}_\Lambda}  \int_{E\cap \Lambda} (f(t^n) - f(t)) R_h \psi(t) 
+  \sum_{E \in \mathcal{T}_\Lambda} \int_{E \cap \Lambda} f(t) (R_h \psi(t) -  \psi(t))
= \mathcal{W}_1 +\mathcal{W}_2. % + \int_\Lambda f(t) (\Pi_h \psi(t) - \psi(t)) = \mathcal{W}_1 +\mathcal{W}_2 + \mathcal{W}_3. 
\end{multline*}
For $\mathcal{W}_1$, we H\"{o}lder's inequality, \eqref{eq:3d_1d_inverse_estimate} ($q= \infty, p = 6)$ and \eqref{eq:3d_1d_Poincare_inequality}. We obtain 
\begin{multline*}
  |\mathcal{W}_1| \leq \|f(t^n) - f(t)\|_{L^1(\Lambda)} \|R_h \psi(t)\|_{L^{\infty}(\Omega)}   \\ 
\leq C h^{-\frac{1}{2}} \|f(t^n) - f(t)\|_{L^1(\Lambda)} \|R_h \psi(t)\|_{L^6(\Omega)} 
\leq C h^{-\frac{1}{2}}\|f(t^n) - f(t)\|_{L^1(\Lambda)} \|R_h \psi(t)\|_{\DG}.
\end{multline*}
Since $R_h\psi$ is the elliptic projection of $\psi$, we note that $\|R_h \psi\|_{\DG} \leq C \|\psi\|_{H^2(\Omega)}$ and 
%\[\|f(t^n) - f\|_{L^2(\Lambda)} \leq |t^n - t|^{1/2} \|\partial_t f\|_{L^2(t, t^n; L^2(\Lambda))},  \] 
we obtain % \footnote{\[ \int_{\Lambda} \left \vert \int_{t}^{t^n} \partial_t f \right \vert \leq \int_\Lambda \int_{t}^{t^n} |\partial_t f| = \int_{t}^{t^n} \|\partial_t f\|_{L^1(\Lambda)} \leq (t^n - t)^{1/2} \|f\|_{L^2(t,t^n; L^1(\Lambda))},\] 
%where the last inequality follows form Cauchy-Schwarz's inequality. 
%}
\begin{equation}
  |\mathcal{W}_1| \leq C (t^n - t)^{1/2} h^{-\frac{1}{2}} \|\partial_t f\|_{L^2(t, t^n; L^1(\Lambda))} \|\psi(t)\|_{H^2(\Omega)}.
\end{equation}
For $\mathcal{W}_2$, we apply a similar argument as for the derivation of \eqref{eq:fully_discrete_error_0} (by introducing
the Lagrange interpolant $L_h\psi$) and obtain for any $2<q<3$
\begin{equation}
\mathcal{W}_2 \leq C h^{-1}\|f(t)\|_{L^2(\Lambda)} \|R_h \psi(t) - L_h \psi(t)\|_{L^2(\Omega)} 
+ C(q) h^{\frac{3}{q} - \frac12} \|f(t)\|_{L^{2}(\Lambda)}|\psi(t)|_{H^2(\Omega)}.
\end{equation}
Hence, with approximation properties,  choosing $q = 6/(3-2\theta)$ for $0<\theta < 1/2$, and \eqref{eq:parabolic_regularity_backward},
the bound on $E_3$ reads 
\begin{align}
 |E_3| & \leq 
C \tau h^{-\frac12} \|\partial_t f\|_{L^2(0,T; L^1(\Lambda))} \|\psi\|_{L^2(0,T;H^2(\Omega))}
+ C h^{-1}\|f\|_{L^2(0,T;L^2(\Lambda))} \|R_h \psi - L_h \psi\|_{L^2(0,T;L^2(\Omega))}  \nonumber \\ 
& \quad  + C(\theta) h^{1-\theta}   \|f\|_{L^2(0,T;L^2(\Lambda))} \|\psi\|_{L^2(0,T;H^2(\Omega))} \nonumber\\
%& \leq C \tau h^{-\frac12} \|\partial_t f\|_{L^2(0,T; L^2(\Lambda))} \|\psi\|_{L^2(0,T;H^2(\Omega))} + C(\theta) h^{1-\theta}   \|f\|_{L^2(0,T;L^2(\Lambda))} \|\psi\|_{L^2(0,T;H^2(\Omega))} \nonumber \\ 
& \leq (C\tau h^{-\frac12}\|\partial_t f\|_{L^2(0,T; L^1(\Lambda))} + C(\theta) h^{1-\theta}   \|f\|_{L^2(0,T;L^2(\Lambda))} )\|g\|_{L^2(0,T;L^2(\Omega))}. \label{eq:bound_E_2_fullydiscrete_3d1d}
 \end{align} 
Therefore, with \eqref{eq:fully_discrete_3d1d_2} and the bounds \eqref{eq:bound_E1_3D1D}, \eqref{eq:bound_E_3_fullydiscrete_3d1d} and 
\eqref{eq:bound_E_2_fullydiscrete_3d1d},  we conclude that  for any non-zero $g\in  L^2(0,T;L^2(\Omega))$
 \begin{multline}
   \frac{ \int_0^T (u_{h,\tau}-u, g)_\Omega  }{\|g\|_{L^2(0,T;L^2(\Omega))}} \leq   C (\tau h^{-1} +h )\left(\|f\|_{\ell^2(0,T;L^2(\Lambda))}  +  \|u^0\|_{L^2(\Omega)} \right)\\
+ C \tau h^{-1} \|\partial_t f\|_{L^2(0,T; L^1(\Lambda))}  + C(\theta) h^{1-\theta}   \|f\|_{L^2(0,T;L^2(\Lambda))} .  
 \end{multline}
We conclude by taking supremum over all $g$. 
\end{proof}

\section{Numerical Results for Elliptic Problem}\label{sec:numerical_results}

We employ the method of manufactured solutions to test the convergence rates of the scheme~\ref{eq:DG_weak_form}.  
The domain is $(0,1)\times(0,1)\times(0,0.25)$  and the line $\Lambda$ is the vertical line passing through the point $(2/3,1/3,0)$.  The function $f$ is
chosen to be the constant function equal to $1$. 
The exact solution is defined by
\begin{equation}
u(x,y,z)= -\frac{1}{2 \pi} \ln\left( ((x-\frac{2}{3})^2 +(y-\frac{1}{3})^2 )^{1/2}\right).
\end{equation}
We compute the numerical errors on a series of uniformly refined meshes made of tetrahedra. We vary the mesh size and the polynomial degree.
The parameters in the definition of the bilinear form are chosen: $\epsilon = -1, \beta=1$. For $k=1$, we choose $\sigma = 5$
and for $k=2$, the penalty value is $\sigma = 12$.
Figure~\ref{fig:dg_solution} shows the dG solution for $k=1$; the size of the mesh is $h=1/16$ and the domain has been sliced for visualization. 
Table~\ref{tab:l2error} displays the $L^2$ errors and convergence rates for the numerical solution with $k=1$ and $k=2$. 
When errors are computed over the whole domain $\Omega$, they converge with a rate equal to one, which
is consistent with  our bound \eqref{eq:improved_global_dg_bd}.   
Next, we verify the accuracy of the solution away from the line singularity by computing the $L^2$ error in two subdomains 
$C_1 = (0.25,0.5)\times(0.5,0.75)\times(0,0.25)$ and $C_2 = (0.0,0.25)\times(0.75,0.1)\times(0,0.25)$. 
Table~\ref{tab:l2error} shows the errors in the $L^2$ norm over $C_1$ and over $C_2$ as the mesh is uniformly refined.
Errors converge with a rate equal to $2$, which is optimal for piecewise linear approximations and suboptimal for piecewise quadratic approximation. The numerical rates are consistent with \eqref{eq:localL2DGk1} for $k=1$  and \eqref{eq:local_estimate_result} for $k=2$. We also remark that the errors in $C_1$ and
in $C_2$ are several order of magnitude smaller than the errors in $\Omega$.
%
%
%\Rd How are the errors computed away from the singularity? The mesh is too coarse... \Bk
%linear and quadratic basis functions. We verify that the numerical approximations agree with the expected convergence rates.  
%

\begin{figure}[H]
\centering
\includegraphics[width=0.49\linewidth]{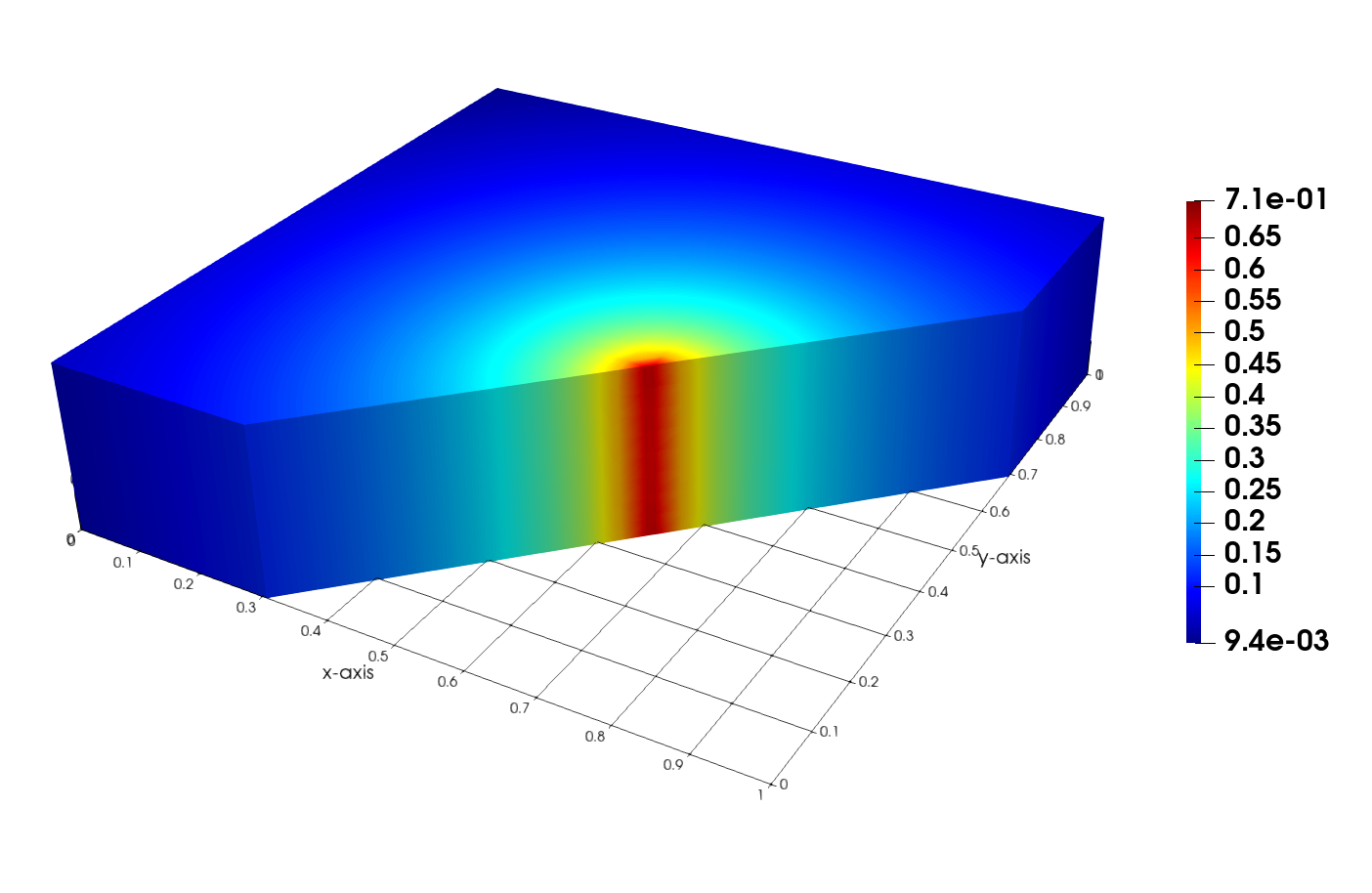}
\caption{View on sliced domain of the dG approximation obtained on mesh of size $h=1/16$.}
\label{fig:dg_solution} 
\end{figure}
To show the robustness of the scheme~\ref{eq:DG_weak_form}, we now consider a sinusoidal-like curve $\Lambda$ made of segments.  The numerical parameters
are the same as for the manufactured solution but here, we do not know the exact solution.  Figure~\ref{fig:sine_solution}
displays the DG solution on a mesh of size $h=1/10$.
\begin{figure}[H]
\centering
\includegraphics[width=0.49\linewidth]{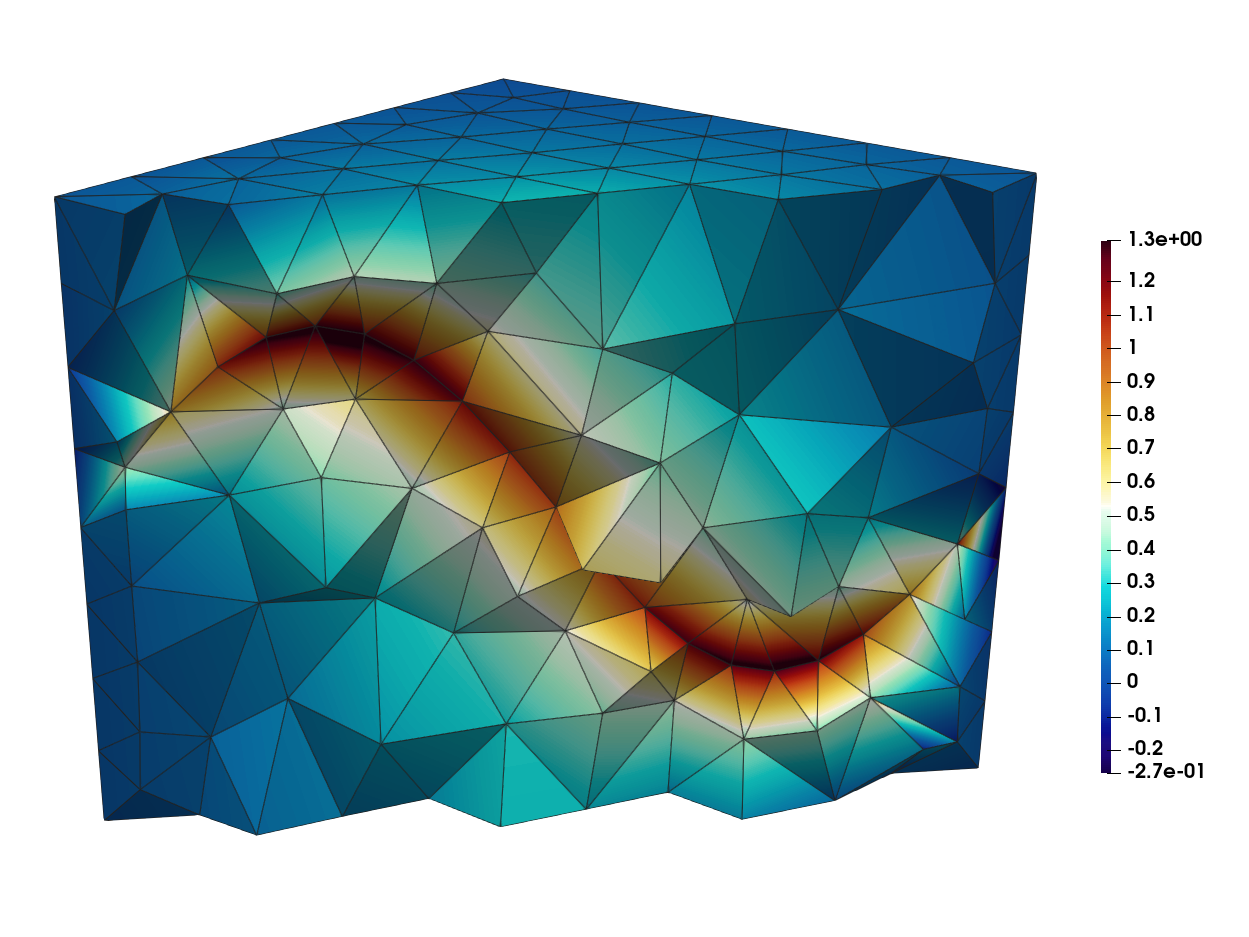}
\caption{Sliced view of the numerical solution for a piecewise linear curve $\Lambda$.}
\label{fig:sine_solution} 
\end{figure}
%\begin{figure}[H]
%\centering
%\includegraphics[width=0.45\linewidth]{Figures/grid_h00625.png}
%\caption{Top view of the computational domain with h=1/16, C1 region (red), C2 region (blue),  the region contains source term (yellow)}
%\label{fig:convergence_study_domain} 
%\end{figure}

\bgroup
\def\arraystretch{1.3}
\begin{table}[H]
\centering
\footnotesize
\begin{tabular}{|c|c|c|c|c|c|c|c|}
\hline
\rowcolor{Gainsboro!20}
&   & $\mathbf{ \Vert u-\uDG \Vert_{L^2(\Omega)}}$ &  & $\mathbf{\Vert u-\uDG  \Vert_{L^2(C_1)}}$  &  & $\mathbf{ \Vert u-\uDG  \Vert_{L^2(C_2)}}$  &  \\
 \rowcolor{Gainsboro!20}
\textbf{k} & \textbf{h}  & \textbf{Error} & \textbf{Rate}  &  \textbf{Error}  & \textbf{Rate}  & \textbf{Error}  & \textbf{Rate}  \\
\hline
1 &1/4    & 6.99e-03   &        &   1.28e-04  &        & 2.54e-05   &       \\
&1/8    & 2.28e-03   &  1.31  &   3.00e-05  & 2.09   & 6.70e-06   & 1.92  \\
&1/16   & 1.33e-03   &  1.08  &   6.60e-06  & 2.18   & 1.84e-06   & 1.86  \\
&1/32   & 7.12e-04   &  0.90  &   1.63e-06  & 2.02   & 5.05e-07   & 1.87  \\
\hline
 \rowcolor{Gainsboro!20}
2 & 1/4 & 1.14e-02 &  & 1.09e-04 &  & 4.37e-06 &   \\
 \rowcolor{Gainsboro!20}
& 1/8 &  4.27e-03 & 1.42 & 1.98e-05 & 2.46 & 7.48e-07 & 2.55  \\
 \rowcolor{Gainsboro!20}
& 1/16  & 1.56e-03 & 1.45 & 6.22e-06 & 1.67 & 1.11e-07 & 2.75  \\
 \rowcolor{Gainsboro!20}
 & 1/32 & 6.14e-04 & 1.35 &  1.50e-06 & 2.05 & 1.77e-08 & 2.65  \\
\hline
\end{tabular}
\caption{Numerical errors and convergence rates for the numerical solution over the whole domain and the two subdomains.}
\label{tab:l2error}
\end{table}
\egroup

\section{Conclusions}

Convergence of the class of interior penalty discontinuous Galerkin methods applied to elliptic and parabolic equations
with Dirac line-source is proved by deriving error estimates in different norms.  Almost optimal error bounds are shown in regions away
from the line singularity. The proofs of the error estimates are technical and utilize dual problems and weighted Sobolev spaces.
Stronger results are obtained for the case of piecewise linear approximation since local error bounds are valid in regions that
may reach the boundary of the domain.  In the general case of approximation of degree $k\geq 2$, local error bounds are subpoptimal and
valid in regions strictly included in the domain. Most of the paper is dedicated to the analysis of the elliptic problem and convexity
of the domain is assumed.
For the parabolic problem, global error bounds in $L^2$ in time and in space are shown.  Future work would address relaxing the
convexity assumption and obtaining local
error bounds for the time-dependent problem.

\textit{Acknowledgment}: the authors are partially supported by NSF-DMS 1913291 and NSF-DMS 2111459.
%toto
%\end{acknowledgment}
%\input{appendix.tex}
%%-----------------------------
%%      your bibliography
%%-----------------------------

\bibliographystyle{plain}
\bibliography{references}
\end{document}